\documentclass[11pt, reqno]{amsart}
\usepackage[dvipsnames]{xcolor}
\usepackage{subfiles}
\usepackage{multicol}
\usepackage[utf8]{inputenc}
\usepackage{amscd,amsmath,amssymb,amsthm,yfonts,mathrsfs,hhline,tikz,dsfont,tikz-cd,mathabx,enumitem,microtype}

\usepackage{marvosym}
\DeclareUnicodeCharacter{1F547}{\Cross}
\usetikzlibrary{arrows, decorations.markings, patterns, calc, intersections,through}
\tikzset{>=latex}
\tikzcdset{arrow style=math font}
\usetikzlibrary{decorations.pathmorphing}
\usetikzlibrary{decorations}
\pgfdeclaredecoration{simple line}{initial}{
  \state{initial}[width=\pgfdecoratedpathlength-1sp]{\pgfmoveto{\pgfpointorigin}}
  \state{final}{\pgflineto{\pgfpointorigin}}
}
\tikzset{
   shift left/.style={decorate,decoration={simple line,raise=#1}},
   shift right/.style={decorate,decoration={simple line,raise=-1*#1}},
}

\usepackage[centering,margin=1.2in]{geometry}
\usepackage[matrix,arrow,curve]{xy}
\global\binoppenalty=10000
\usepackage{setspace}
\usepackage{verbatim}
\usepackage{graphicx}
\DeclareMathAlphabet{\mathpzc}{OT1}{pzc}{m}{it}

\usepackage[backend=biber,
            isbn=false,
            doi=false,
            url=false,
            giveninits=true,
            style=alphabetic,
            useprefix=true,
            maxcitenames=99,
            maxbibnames=99,
            maxalphanames=99,
            minnames=5,
            sorting=nyt,
            citestyle=alphabetic
            ]{biblatex}
\usepackage{url}

\usepackage{breakurl}
\usepackage[breaklinks]{hyperref}

\hypersetup{pdfauthor={Name}}

\usepackage[capitalise]{cleveref}

\newtheorem{theorem}{Theorem}[section]
\newtheorem{lemma}[theorem]{Lemma}
\newtheorem{prop}[theorem]{Proposition}
\newtheorem{cor}[theorem]{Corollary}

\theoremstyle{definition}
\newtheorem{defn}[theorem]{Definition}
\newtheorem{remark}[theorem]{Remark}
\newtheorem{example}[theorem]{Example}
\newtheorem{notation}[theorem]{Notation}

\numberwithin{equation}{section}

\def\beq{\begin{equation}}
\def\eeq{\end{equation}}

\def\longra{\longrightarrow}

\newcommand{\hr}[1]{\left(#1\right)} 
\newcommand{\ha}[1]{\left\langle#1\right\rangle} 
\newcommand{\hs}[1]{\left[#1\right]} 
\newcommand{\hc}[1]{\left\{#1\right\}} 

\newcommand\nord[1]{\hspace{.2em}\boldsymbol{:}\hspace{-.2em}#1\hspace{-.2em}\boldsymbol{:}\hspace{.2em}}

\def\le{\leqslant}
\def\ge{\geqslant}

\def\act{\mathrm{act}}
\def\Ac{\mathcal A}

\def\b{\mathfrak b}

\def\Bc{\mathcal B}

\def\C{\mathbb C}
\def\Cc{\mathcal C}
\def\Ch{\operatorname{Ch}}
\def\Conf{\operatorname{Conf}^{fr}}

\newcommand{\OqConftri}[1]{\Oq(\Conf_{#1})[S_\tri^{-1}]}
\newcommand{\OqConfe}[2]{\Oq(\Conf_{#1})[{#2}^{-1}]}

\def\dim{\operatorname{dim}}

\def\Disp{\mathbb{D}\mathrm{isk}}
\def\Dist{\mathrm{Dist}}

\def\eps{\varepsilon}

\def\End{\operatorname{End}}

\def\Fc{\mathcal F}

\def\For{\operatorname{For}}

\def\g{\mathfrak g}
\def\Gc{\mathcal G}

\def\hatS{\widehat S}

\def\Hom{\operatorname{Hom}}

\def\coker{\operatorname{coker}}
\def\Free{\operatorname{Free}}

\def\id{\mathrm{id}}

\def\la{\lambda}

\def\longra{\longrightarrow}

\def\Lf{\mathbf L}

\def\mod{\mathrm{-mod}}

\def\Mc{\mathcal M}

\def\n{\mathfrak n}
\def\nsp{\:\!}

\def\one{\mathds 1}
\def\op{\mathrm{op}}
\def\Oc{\mathcal O}

\def\Pc{\mathcal P}

\def\Pr{\mathbf{Pr}}

\def\QC{\mathcal{QC}}

\def\rhu{\rightharpoonup}
\def\R{\mathbb R}
\def\Rc{\mathcal R}
\def\Rep{\mathrm{Rep}}

\def\sl{\mathfrak{sl}}
\def\Surp{\mathbb{S}\mathrm{urf}}

\def\t{\mathfrak t}

\def\Tc{\mathcal T}
\def\Tf{\mathbf T}
\def\trib{\tri_\bullet}

\def\wt{\operatorname{wt}}

\def\Xc{\mathcal X}

\def\cC{\mathcal C }
\def\K{k}
\def\Z{\mathbb Z}
\def\Zc{\mathcal Z}

\def\Bc{\mathcal B}
\def\Xc{\mathcal X}
\def\Lbf{\mathbf L}

\def\Zq{\mathcal{Z}}

\def\Aut{\operatorname{Aut}}

\def\iHom{\underline{\Hom}}
\def\id{\operatorname{id}}
\def\iEnd{\underline{\End}}
\def\oo{\ensuremath{\infty}}
\def\SL{\operatorname{SL}}
\def\PGL{\operatorname{PGL}}
\def\ot{\otimes}
\def\gt{\mathfrak{t}}
\def\gb{\mathfrak{b}}
\def\DD{\mathbb{D}}

\def\colim{\operatorname{colim}}
\def\bt{\boxtimes}
\def\gn{\mathfrak{n}}
\def\cC{\mathcal C }
\def\cD{\mathcal D }
\def\cE{\mathcal E }
\def\cP{\mathcal P }
\def\Id{\operatorname{Id}}
\def\cA{\mathcal A}
\def\cM{\mathcal M}

\def\Torsion{\mathrm{Tors}}
\def\Repq{\mathrm{Rep}_q}
\newcommand{\bop}[1]{\overline{#1}}

\newcommand{\Oq}{\Fc_q}
\newcommand{\rev}[1]{\overline{#1}}
\newcommand{\defterm}[1]{\textbf{\emph{#1}}}

\newcommand{\into}{\hookrightarrow}

\newcommand{\decS}{\mathbb{S}}
\newcommand{\decX}{\mathbb{X}}
\newcommand{\decE}{\mathbb{E}}
\newcommand{\tri}{{\underline{\Delta}}}

\newcommand{\Fq}{\mathcal{F}_q}
\newcommand{\TBG}{T\leftarrow B \rightarrow G}
\newcommand{\Vect}{\operatorname{Vect}}
\newcommand{\BrTens}{\operatorname{BrTens}}

\newcommand{\Res}{\operatorname{Res}}

\newcommand{\Spec}{\operatorname{Spec}}

\newcommand{\col}[1]{~{\raise-2pt\hbox{\tiny$\bullet$}\hskip -4.7pt \raise4pt\hbox{\tiny$\bullet$}{{#1}} \raise-2pt\hbox{\tiny$\bullet$}\hskip -4.7pt \raise4pt\hbox{\tiny$\bullet$}}\,~}

\begin{document}

\title[Quantum decorated character stacks]{Quantum decorated character stacks}

\author{David Jordan}
\author{Ian Le}
\author{Gus Schrader}
\author{Alexander Shapiro}
\maketitle

\begin{abstract}
We initiate the study of decorated character stacks and their quantizations using the framework of stratified factorization homology. We thereby extend the construction by Fock and Goncharov of (quantum) decorated character varieties to encompass also the stacky points, in a way that is both compatible with cutting and gluing and equivariant with respect to canonical actions of the modular group of the surface. In the cases $G=\SL_2,\PGL_2$ we construct a system of categorical charts and flips on the quantum decorated character stacks which generalize the well--known cluster structures on the Fock--Goncharov moduli spaces. 
\end{abstract}

\setcounter{tocdepth}{1}
\tableofcontents

\section{Introduction}
In this paper we introduce \emph{decorated character stacks} and their quantizations using the framework of \emph{stratified factorization homology}, as recently developed by Ayala, Francis, and Tanaka \cite{AFT}. Given a reductive group $G$ with Borel subgroup $B$ and its Cartan quotient $T$, a \emph{decorated surface} is a bipartite surface with open regions colored $G$ and $T$ and codimension one defects between them labelled $B$.  The decorated character stack $\Ch(\decS)$ is the moduli stack of $G$- and $T$-local systems over each corresponding region of $\decS$, together with additional $B$-reduction data along the defect lines. The quantum decorated character stack is a quantization $\Zc(\decS)$ of the category $\QC(\Ch_G(\decS))$, obtained by integration over $\decS$ with coefficients in a \emph{parabolic induction algebra}, which prescribes the ribbon braided tensor categories $\Repq(G)$ and $\Repq(T)$ in their respective regions, and their central tensor category $\Repq(B)$ along line defects.

The resulting quantum decorated character stacks form a basic ingredient in the study of Betti quantum geometric Langlands, a mathematical model for the Kapustin-Witten twist of $\mathcal{N}=4$ SUSY $4d$ Yang-Mills theory compactified on $\R^2$ (that is, evaluated on decorated surfaces times $\mathbb{R}^2$).  In mathematical terms, the structure of the parabolic induction algebra is simply that of a $1$-morphism between $\Repq(G)$ and $\Repq(T)$, these categories being regarded as objects in the 4-category of braided tensor categories \cite{BJS} which is the target of the 4d TFT, so that $\Repq(B)$ defines a natural transformation between the $G$ and $T$-theories. In physical terms, the parabolic induction algebra defines a \emph{domain wall} between the $G$- and $T$-theories, providing the structure of local operators at interfaces of bipartite manifolds colored by $G$ and $T$, and hence our constructions give computations of global operators for such manifolds.

In this paper we develop a number of tools which allow for very concrete computations with quantum decorated character stacks.  Our lodestar is the following proposal of David Ben-Zvi:
\smallskip
\begin{quote}
{\it The quantum cluster algebras associated to a marked surface by Fock, Goncharov, and Shen describe charts on a proper open locus of $\Zc(\decS)$, consisting wholly of non-stacky points.}
\end{quote}
\smallskip

Let us briefly recall the quantum cluster algebra constructions of Fock, Goncharov and Shen.  In~\cite{FG06}, Fock and Goncharov consider a pair of moduli spaces. The first one, $\Xc_{G,S}$, is the moduli space of \emph{framed} $G$-local systems on a decorated surface $S$, defined for an arbitrary split reductive group $G$. The second one, $\Ac_{G,S}$, is the moduli space of \emph{decorated twisted} $G$-local systems on $S$, defined for any simply-connected reductive group $G$. Among other results, it was shown in~\cite{FG06} that $\Xc_{PGL_n,S}$ and $\Ac_{SL_n,S}$ are respectively cluster Poisson and cluster $K_2$-varieties\footnote{Cluster Poisson and a cluster $K_2$-varieties are also known as cluster $\Xc$-- and  cluster $\Ac$--varieties, respectively.}, and moreover, form a cluster ensemble. In~\cite{GS19}, \cite{Le19}, \cite{Ip18}, these results were extended to arbitrary Dynkin types. In~\cite{GS19} the moduli space $\Xc_{G,S}$ was promoted to a new one, $\Pc_{G,S}$, parameterizing framed $G$-local systems with \emph{pinnings.} This moduli space also has a cluster Poisson structure, but in contrast with $\Xc_{G,S}$, allows for frozen variables at the boundary of $S$ and thus admits gluing maps that are not available for the spaces $\Xc_{G,S}$.

Non-commutative deformations of cluster Poisson varieties were constructed in~\cite{FG09a,FG09b}, by associating to each cluster Poisson chart a quantum torus algebra, and promoting the transition maps between charts (called ``cluster Poisson transformations") to algebra isomorphisms between the skew fields of fractions of the corresponding quantum tori.  The output of this construction is best understood as a functor $\Xc^q_{G,S}$ from the \emph{cluster modular groupoid}, whose objects are cluster Poisson charts on $\Xc_{G,S}$ and morphisms cluster Poisson transformations, to the category of quantum torus algebras with birational isomorphisms between them. One may then pass to ``global sections'' to obtain an algebra $\Lbf_q(\Xc_{G,S})$ consisting of all quantum torus elements which remain regular under any sequence of quantum cluster transformations.  

Ben-Zvi's proposal predicts on the one hand that the category $\Zc(\decS)$ should capture the rich combinatorics of the $q$--deformed moduli spaces $\Pc^q_{G,S}, \Xc^q_{G,S}$, and on the other hand that $\Zc(\decS)$ provides a ``stacky'' enhancement of each one.  Such an enhancement is important:  whereas removing the stacky points allows for more explicit descriptions of open loci in cluster terms, it crucially destroys the topological functoriality present in our construction.  Indeed, the assignment $\decS\mapsto\Zc(\decS)$ is \emph{a priori} an invariant of surfaces, naturally equivariant for the appropriate mapping class group. In the traditional approach, the corresponding equivariance statement can only be deduced \emph{as a consequence of} the highly non-trivial construction of a cluster structure on $\Pc^q_{G,S}, \Xc^q_{G,S}$.

The second important reason to work within the framework of stacks is that doing so allows to construct the moduli spaces and their quantizations in a way that is much more local with respect to the decorated surface $\decS$. The central idea of the standard cluster approach is also to work locally in $\decS$ by regarding it as being glued from decorated triangles -- these are the smallest decorated surfaces for which the moduli stack has an open rational subvariety. Yet our techniques allow us to localize all the way down to the very simplest nontrivial decorated surface -- a disk $\DD_1$ with a single domain wall between $G$ and $T$ regions. Then, in a precise sense (explained in \cref{prop-Pr}), the parabolic induction algebra $\Zc(\DD_1)=\Repq(B)$ completely determines $\Zc(\decS)$ for any other decorated surface $\decS$. In particular, the decorated triangles fundamental to the standard cluster formalism simply arise by gluing together three copies of $\DD_1$, as indicated in \cref{fig:D3}.

\begin{figure}[t]
\begin{tikzpicture}[every node/.style={inner sep=0.5, thick, circle}, thick, x=0.5cm, y=0.5cm]

\def\Greg{(-5,0) to [out=90,in=180] (-3,2) to (3,2) to [out=0,in=90] (5,0) to (5,-3) to [out=-90,in=0] (4,-4) to [out=180,in=-90] (3,-3) to [out=90,in=0] (2,-1) to [out=180,in=90] (1,-3) to [out=-90,in=0] (0,-4) to [out=180,in=-90] (-1,-3) to [out=90,in=0] (-2,-1) to [out=180,in=90] (-3,-3) to [out=-90,in=0] (-4,-4) to [out=180,in=-90] (-5,-3) to (-5,0)};

\def\Tone{(-5,-2.5) to [bend left = 50] (-3,-2.5) to (-3,-3) to [out=-90,in=0] (-4,-4) to [out=180,in=-90] (-5,-3) to (-5,-2.5)};
\def\Ttwo{(-1,-2.5) to [bend left = 50] (1,-2.5) to (1,-3) to [out=-90,in=0] (0,-4) to [out=180,in=-90] (-1,-3) to (-1,-2.5)};
\def\Tthree{(3,-2.5) to [bend left = 50] (5,-2.5) to (5,-3) to [out=-90,in=0] (4,-4) to [out=180,in=-90] (3,-3) to (3,-2.5)};

\fill[Orchid!20] \Greg;
\fill[Goldenrod!60] \Tone \Ttwo \Tthree;

\draw[very thick, cyan] (-5,-2.5) to [bend left = 50] (-3,-2.5);
\draw[very thick, cyan] (-1,-2.5) to [bend left = 50] (1,-2.5);
\draw[very thick, cyan] (3,-2.5) to [bend left = 50] (5,-2.5);
\draw[ultra thick, dashdotted] (-2,-1) to (-2,2);
\draw[ultra thick, dashdotted] (2,-1) to (2,2);
\draw \Greg;

\draw[ultra thick, Goldenrod!60, domain=-120:-60] plot ({-4+cos(\x)}, {-3+sin(\x)});
\draw[ultra thick, Goldenrod!60, domain=-120:-60] plot ({cos(\x)}, {-3+sin(\x)});
\draw[ultra thick, Goldenrod!60, domain=-120:-60] plot ({4+cos(\x)}, {-3+sin(\x)});

\draw[domain=0.85:1.15] plot ({-4+\x*cos(-60)}, {-3+\x*sin(-60)});
\draw[domain=0.85:1.15] plot ({-4+\x*cos(-120)}, {-3+\x*sin(-120)});
\draw[domain=0.85:1.15] plot ({\x*cos(-60)}, {-3+\x*sin(-60)});
\draw[domain=0.85:1.15] plot ({\x*cos(-120)}, {-3+\x*sin(-120)});
\draw[domain=0.85:1.15] plot ({4+\x*cos(-60)}, {-3+\x*sin(-60)});
\draw[domain=0.85:1.15] plot ({4+\x*cos(-120)}, {-3+\x*sin(-120)});

\node at (-4,-4.5) {\footnotesize 1};
\node at (0,-4.5) {\footnotesize 2};
\node at (4,-4.5) {\footnotesize 3};

\begin{scope}[shift={(-15,0)}]

\def\Greg{(-5,0) to [out=90,in=180] (-3,2) to (3,2) to [out=0,in=90] (5,0) to (5,-3) to [out=-90,in=0] (4,-4) to [out=180,in=-90] (3,-3) to [out=90,in=0] (2,-1) to [out=180,in=90] (1,-3) to [out=-90,in=0] (0,-4) to [out=180,in=-90] (-1,-3) to [out=90,in=0] (-2,-1) to [out=180,in=90] (-3,-3) to [out=-90,in=0] (-4,-4) to [out=180,in=-90] (-5,-3) to (-5,0)};

\def\Gregl{(-2.2,-1) to [bend left = 10] (-2.2,2) to (-1.8,2) to [bend left = 10] (-1.8,-1) to (-2.2,-1)};
\def\Gregr{(2.2,-1) to [bend right = 10] (2.2,2) to (1.8,2) to [bend right = 10] (1.8,-1) to (2.2,-1)};

\def\Tone{(-5,-2.5) to [bend left = 50] (-3,-2.5) to (-3,-3) to [out=-90,in=0] (-4,-4) to [out=180,in=-90] (-5,-3) to (-5,-2.5)};
\def\Ttwo{(-1,-2.5) to [bend left = 50] (1,-2.5) to (1,-3) to [out=-90,in=0] (0,-4) to [out=180,in=-90] (-1,-3) to (-1,-2.5)};
\def\Tthree{(3,-2.5) to [bend left = 50] (5,-2.5) to (5,-3) to [out=-90,in=0] (4,-4) to [out=180,in=-90] (3,-3) to (3,-2.5)};

\fill[Orchid!20] \Greg;
\fill[Goldenrod!60] \Tone \Ttwo \Tthree;

\draw[very thick, cyan] (-5,-2.5) to [bend left = 50] (-3,-2.5);
\draw[very thick, cyan] (-1,-2.5) to [bend left = 50] (1,-2.5);
\draw[very thick, cyan] (3,-2.5) to [bend left = 50] (5,-2.5);
\begin{scope}
\clip\Greg;
\draw (-2.2,-1) to [bend left = 10] (-2.2,2);
\draw (-1.8,-1) to [bend right = 10] (-1.8,2);
\draw (1.8,-1) to [bend left = 10] (1.8,2);
\draw (2.2,-1) to [bend right = 10] (2.2,2);
\fill[Orchid!50] \Gregl;
\fill[Orchid!50] \Gregr;
\end{scope}
\draw \Greg;

\draw[ultra thick, Goldenrod!60, domain=-120:-60] plot ({-4+cos(\x)}, {-3+sin(\x)});
\draw[ultra thick, Goldenrod!60, domain=-120:-60] plot ({cos(\x)}, {-3+sin(\x)});
\draw[ultra thick, Goldenrod!60, domain=-120:-60] plot ({4+cos(\x)}, {-3+sin(\x)});

\draw[domain=0.85:1.15] plot ({-4+\x*cos(-60)}, {-3+\x*sin(-60)});
\draw[domain=0.85:1.15] plot ({-4+\x*cos(-120)}, {-3+\x*sin(-120)});
\draw[domain=0.85:1.15] plot ({\x*cos(-60)}, {-3+\x*sin(-60)});
\draw[domain=0.85:1.15] plot ({\x*cos(-120)}, {-3+\x*sin(-120)});
\draw[domain=0.85:1.15] plot ({4+\x*cos(-60)}, {-3+\x*sin(-60)});
\draw[domain=0.85:1.15] plot ({4+\x*cos(-120)}, {-3+\x*sin(-120)});

\node at (-4,-4.5) {\footnotesize 1};
\node at (0,-4.5) {\footnotesize 2};
\node at (4,-4.5) {\footnotesize 3};

\end{scope}

\end{tikzpicture}
\caption{Triangle $\DD_3$ with a gate in each $T$-region, composed of three copies of $\DD_B$ via excision along $G$-regions. In what follows, we abbreviate overlapping regions by thick dashed and dotted lines, as shown on the right.}
\label{fig:D3}
\end{figure}
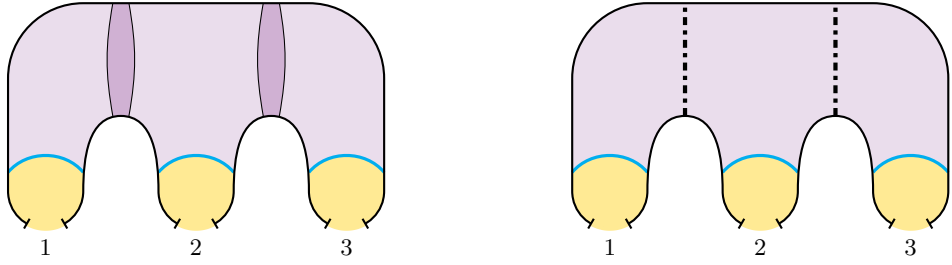

Finally -- and most importantly -- quantum decorated character stacks fit by construction into the framework of fully extended 4-dimensional topological field theory with defects, so that in particular one can define functors $\Zc(\decS)\to\Zc(\decS')$ from decorated cobordisms between $\decS$ and $\decS'$.  These functors do not preserve -- and hence cannot be defined via -- the non-stacky locus.  In future work we intend to explore applications of the functoriality properties of quantum decorated character stacks to the construction of invariants of links in 3--manifolds, and in particular to the proposals in  \cite{Dim,DGG}.

Our main results in this paper establish Ben--Zvi's proposal in rank one -- i.e. for the groups $\SL_2$ and $\PGL_2$.  We exhibit a collection of categorical open ``cluster charts" $\Zc(\tri)\subset \Zc(\decS)$, with transition functors given by simple ``cluster mutations" $\Zc(\tri)\to\Zc(\tri')$ relating triangulations.  In \cref{thm:Ahat} and \cref{thm:X}, we describe each chart $\Zc(\tri)$ explicitly as a category of modules for a quantum torus arising via monadic reconstruction, and we also describe flips as certain explicit birational transformations between these tori.  We show that in special cases (see \cref{thm:X},  \cref{rmk:A}, \cref{rmk:X}, \cref{cor:A}) our charts and flips coincide with those on $\Ac^{q=1}_{\SL_2,\decS}$, $\Pc^q_{\PGL_2,\decS}$, $\Xc^q_{\PGL_2,\decS}$. We note that in the remaining cases -- $\Ac^{q\neq 1}_{G,\decS}$, $\Ac^{q=1}_{\PGL_2,\decS}$, $\Pc^q_{\SL_2,\decS}$, $\Xc^q_{\SL_2,\decS}$ -- analogues of cluster structures were not previously known, so there is nothing to which we should compare our construction. In a similar vein, we extend to arbitrary $q$ the cluster ensemble map, which was previously only defined for $q=1$, or in the absence of punctures (see \cref{rmk:clust-ens}).

\medskip
For the remainder of the introduction, we give a more detailed overview of our basic definitions and main results.

\subsection{Decorated local systems and their moduli spaces}
Fix a connected reductive group $G$, a Borel subgroup $B\into G$ and its universal Cartan quotient $T=B/[B,B]$.

By a \defterm{walled surface} we will mean an oriented surface $S$, together with a collection $\Cc = \hc{C_1, \dots, C_r}$ of non-intersecting simple curves embedded into $S$, such that $\partial C_i \subset \partial S$ and $(C_i \smallsetminus \partial C_i) \cap \partial S = \emptyset$ for every $C_i \in \Cc$. \defterm{Walls} and \defterm{regions} of the walled surface $\decS$ are respectively the curves $C_i$ and the connected components of $S \smallsetminus \Cc$.

\begin{defn}
A \defterm{decorated surface} $\decS$ is a walled surface $S$ together with a labeling of each region from the alphabet $\{G,T\}$, such that any two regions sharing a wall have distinct labeling. We denote the union of all regions with label $G$ by $\decS_G$, and that of all regions with label $T$ by $\decS_T$.
\end{defn}

\begin{defn}
A \defterm{$G$-gate} (resp. $T$-gate) is an interval in $\decS_G\cap \partial S$ (resp. $\decS_T\cap \partial S$). A set $\Gc$ of disjoint gates in $\decS$ will be called a \defterm{gating} of $\decS$.  In the special case that $\Gc$ consists of a unique gate in each $T$-region, we will say $\decS$ is \defterm{notched}.
\end{defn}

\begin{defn}
An \defterm{$n$-gon} $\DD_n$ is a $G$-disk decorated with $n$ contractible $T$-regions along the boundary (see \cref{fig:D4-disk}).  
\end{defn}

\begin{defn}
A \defterm{triangulation} $\tri$ is a presentation of $\decS$ as a union of triangles $\DD_3$, such that the intersection of any pair of triangles is either empty or a union of (one, or two) digons $\DD_2$, and the intersection of any triple of triangles is empty.  We shall often refer to the digons in a triangulation of a decorated surface as the \defterm{edges} of the triangulation. A \defterm{notched triangulation} of a notched $\decS$ is such that each end of each digon lies in the unique $T$-gate.
\end{defn}

\begin{figure}[t]
\begin{tikzpicture}[every node/.style={inner sep=0.5, thick, circle}, x=0.5cm, y=0.5cm]

\def\Greg{(-3,2) to [out=90,in=180] (-2,3) to (2,3) to [out=0,in=90] (3,2) to (3,-2) to [out=-90,in=0] (2,-3) to (-2,-3) to [out=180,in=-90] (-3,-2) to (-3,2)};
\def\Tone{(-3,2) to [out=90,in=180] (-2,3) to (-1.5,3) to [out=-90,in=0] (-3,1.5) to (-3,2)};
\def\Ttwo{(2,3) to [out=0,in=90] (3,2) to (3,1.5) to [out=180,in=-90] (1.5,3) to (2,3)};
\def\Tthree{(3,-2) to [out=-90,in=0] (2,-3) to (1.5,-3) to [out=90,in=180] (3,-1.5) to (3,-2)};
\def\Tfour{(-2,-3) to [out=180,in=-90] (-3,-2) to (-3,-1.5) to [out=0,in=90] (-1.5,-3) to (-2,-3)};

\fill[Orchid!20] \Greg;
\fill[Goldenrod!60] \Tone \Ttwo \Tthree \Tfour;

\begin{scope}
	\clip\Greg;
	\draw[ultra thick, dashdotted] (-3,3) to (3,-3);
\end{scope}


\draw[thick] \Greg;
\draw[very thick, cyan] (-1.5,3) to [out=-90,in=0] (-3,1.5);
\draw[very thick, cyan] (3,1.5) to [out=180,in=-90] (1.5,3);
\draw[very thick, cyan] (1.5,-3) to [out=90,in=180] (3,-1.5);
\draw[very thick, cyan] (-3,-1.5) to [out=0,in=90] (-1.5,-3);


\begin{scope}[shift={(12,0)}]

\def\Gcirc{(0,0) circle (3) (0,0) circle (.5)};
\fill[Orchid!20, even odd rule] \Gcirc;
\def\Tone{(90:3) circle (1)};
\def\Ttwo{(-90:3) circle (1)};
\def\Tthree{(0,0) circle (1) (0,0) circle (.5)};

\begin{scope}
	\clip\Gcirc;
	\fill[Goldenrod!60, even odd rule] \Tone \Ttwo \Tthree;
	\draw[very thick, cyan] \Tone \Ttwo (0,0) circle (1);
\end{scope}

\draw[thick] \Gcirc;

\draw[dashdotted, ultra thick] (0,3) to (0,.5);

\draw[dashdotted, ultra thick] (0,-3) to (0,-.5);

\end{scope}

\end{tikzpicture}
\caption{At left: triangulation of~$\DD_4$. At right: triangulation of the punctured digon~$\DD_2^\circ$.}
\label{fig:D4-disk}
\end{figure}
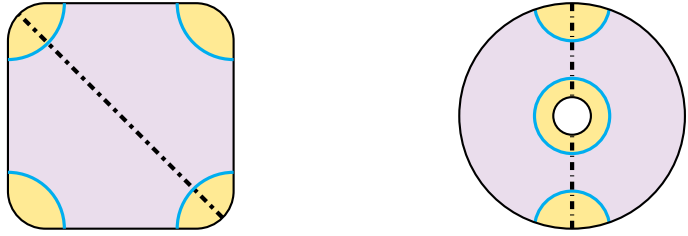

\begin{remark}
\label{rem:triangulations}
A decorated surface $\decS$ is called \defterm{simple} if it admits a triangulation.  In a simple decorated surface, each $T$-region is either a disk contracting to a point of $\partial S$ or an annulus contracting onto an entire component of $\partial S$.  We call such $T$-regions \defterm{marked points} and \defterm{punctures} respectively.  Isotopy classes of simple decorated surfaces are in bijection with marked surfaces appearing in \cite{FG06}, while isotopy classes of triangulations of decorated surfaces are in bijection with those of  \defterm{ideal triangulations} -- maximal collections of non-intersecting arcs with endpoints at marked points or punctures -- of the corresponding marked surfaces $S$. A notched triangulation determines a \defterm{notched ideal triangulation} -- an ideal triangulation with a distinguished angle (see \cref{fig:punct-disk}).
\end{remark}

\begin{defn} Let $\decS$ be a decorated surface. A \defterm{decorated local system} $\decE$ on $\decS$ is a triple $\decE = (E_G,E_T,E_B)$, consisting of a $G$-local system $E_G$ on $\decS_G$, a $T$-local system $E_T$ on $\decS_T$, and a reduction $E_B$ to $B$ of the product $G\times T$-local system $E_G\times E_T$ over $\cC$.
\end{defn}

Recall that a $G$-local system means a principal $G$-bundle with flat connection. A reduction of a local system to a subgroup $H$ means the specification of a flat $H$-subbundle.  We shall consider the group embedding $B \into G\times T$ obtained by composing the diagonal embedding of $B$ into $B \times B$ with the inclusion of $B$ in $G$ and the projection to $T$. In this case, a $B$-reduction amounts to specifying an element of the coset space $B\backslash(G\times T) \cong N\backslash G$, satisfying a certain compatibility with the $G$ and $T$-monodromies near the curve $\cC$ (see \cref{ex:punc-annuli}).

\begin{defn}
The \defterm{decorated character stack} $\Ch(\decS)$ is the moduli stack of decorated local systems on $\decS$.
\end{defn}

Now fix a gating $\Gc$ of $\decS$.  Then by a framing\footnote{Hence our use of `framing' differs from \cite{FG06}, where it refers to a $B$-reduction rather than a trivialization.} of $\decE$ we will mean a trivialization of the local systems $E_G$ and $E_T$ at each $G$- and $T$-gate, respectively.

\begin{defn} 
The \defterm{framed decorated character stack} $\Ch^{fr}_\Gc(\decS)$ is the moduli stack of pairs, consisting of a decorated local system, and a framing of the local system over $\Gc$.
\end{defn}

Suppose that $\Gc$ consists of at least one gate in every region of $\decS$.  In this case, the moduli stack is in fact a variety, since an isomorphism $\decE\to\decE'$ of decorated local systems is uniquely determined by the condition that it preserve trivializations at each gate.  This allows us to present the decorated character stack concretely as the stack quotient:
\[
\Ch(\decS) \cong \Ch^{fr}_\Gc(\decS)/(G^m\times T^n),
\]
of an affine variety by a reductive algebraic group, where $G^m\times T^n$ acts by changing framings at the $m$ $G$-gates and $n$ $T$-gates of $\Gc$.

\begin{example}
Consider the $n$-gon $\DD_n$.  The only $G$- and $T$-local systems are the trivial one, hence the only data are the $B$-reductions, given by elements of $B\backslash(G\times T)\cong N\backslash G$.  Hence, putting a $G$-gate in the $G$-region and a $T$-gate in each $T$-region, we have:
$$
\Ch^{fr}_\Gc(\DD_n) = (N\backslash G)^{\times n}
\qquad\text{and}\qquad
\Ch(\DD_n)=T^n\,\backslash(N\backslash G)^{\times n}/\,G.
$$
\end{example}

\begin{example}\label{ex:punc-annuli}
Let us consider the four decorated and gated surfaces depicted in \cref{fig:four-variations}, and their decorated character stacks.  In each one, the $G$-local system $E_G$ is prescribed by an element $g$ of $G$ recording the monodromy in the $G$-region. For $(\hat{A})$ and $(P)$, the $T$-local system $E_T$ is given by an element $h$ of $T$, while for ($A$) and ($\overline{P}$) we must have $h=e$ since $E_T$ must be trivial.  Finally, in each case we need to specify a principal $B$-sub-bundle of the restriction of $E_G\times E_T$ over the wall.  Such a sub-bundle is specified by a Borel subgroup $B'$ containing $g$.  To summarize, in each case $\Ch_\Gc^{fr}(\decS)$ parameterizes:
\begin{itemize}
\item[($\hat{A}$)] An element $\widetilde{F}\in N\backslash G$ at the puncture, with a requirement that the monodromy in the $G$-region preserves $\pi(\widetilde{F})\in B\backslash G$.
\item[($P$)] An element $F\in B\backslash G$ at the puncture, with a requirement that the monodromy in the $G$-region preserves $F$.
\item[($A$)] An element $\widetilde{F} \in N\backslash G$ at the puncture, with a requirement that the monodromy in the $G$-region fixes $\widetilde{F}$.
\item[($\overline{P}$)] An element $F \in B\backslash G$ at the  puncture, with a requirement that the monodromy in the $G$-region preserves $F$, and is unipotent.
\end{itemize}
\end{example}
 
\begin{figure}[t]
 
\begin{tikzpicture}[every node/.style={inner sep=0.5, thick, circle}, x=0.75cm, y=0.75cm]

\begin{scope}[shift={(0,0)}, scale=0.3]
\def\Greg{(0,0) circle (5)};

\def\Tcirc{(0,0) circle (2)}

\fill[Orchid!20] \Greg;
\begin{scope}
	\clip\Greg;
	\fill[Goldenrod!60] \Tcirc;
	\draw[very thick, cyan] \Tcirc;
\end{scope}

\draw[thick, fill=white] (0,0) circle (1);

\draw[thick] \Greg;

\draw[ultra thick, Orchid!20, domain=-10:10] plot ({5*cos(\x)}, {5*sin(\x)});
\draw[thick, domain=4.7:5.3] plot ({\x*cos(10)}, {\x*sin(10)});
\draw[thick, domain=4.7:5.3] plot ({\x*cos(-10)}, {\x*sin(-10)});

\draw[ultra thick, Goldenrod!60, domain=-30:30] plot ({cos(\x)}, {sin(\x)});
\draw[thick, domain=.7:1.3] plot ({\x*cos(30)}, {\x*sin(30)});
\draw[thick, domain=.7:1.3] plot ({\x*cos(-30)}, {\x*sin(-30)});

\node at (-90:7) {$(\widehat{A})$};
\end{scope}

\begin{scope}[shift={(4,0)}, scale=0.3]
\def\Greg{(0,0) circle (5)};

\def\Tcirc{(0,0) circle (2)}

\fill[Orchid!20] \Greg;
\begin{scope}
	\clip\Greg;
	\fill[Goldenrod!60] \Tcirc;
	\draw[very thick, cyan] \Tcirc;
\end{scope}

\draw[thick, fill=white] (0,0) circle (1);

\draw[thick] \Greg;
\draw[ultra thick, Orchid!20, domain=-10:10] plot ({5*cos(\x)}, {5*sin(\x)});
\draw[thick, domain=4.7:5.3] plot ({\x*cos(10)}, {\x*sin(10)});
\draw[thick, domain=4.7:5.3] plot ({\x*cos(-10)}, {\x*sin(-10)});


\node at (-90:7) {($P$)};
\end{scope}

\begin{scope}[shift={(8,0)}, scale=0.3]
\def\Greg{(0,0) circle (5)};

\def\Tcirc{(0,0) circle (2)}

\fill[Orchid!20] \Greg;
\begin{scope}
	\clip\Greg;
	\fill[Goldenrod!60] \Tcirc;
	\draw[very thick, cyan] \Tcirc;
	\node[draw, dashed, fill=Goldenrod!60] at (0:0) {\tiny$\DD_T$};
\end{scope}


\draw[thick] \Greg;
\draw[ultra thick, Orchid!20, domain=-10:10] plot ({5*cos(\x)}, {5*sin(\x)});
\draw[thick, domain=4.7:5.3] plot ({\x*cos(10)}, {\x*sin(10)});
\draw[thick, domain=4.7:5.3] plot ({\x*cos(-10)}, {\x*sin(-10)});


\node at (-90:7) {$(A)$};
\end{scope}

\begin{scope}[shift={(12,0)}, scale=0.3]
\def\Greg{(0,0) circle (5)};

\def\Tcirc{(0,0) circle (2)}

\fill[Orchid!20] \Greg;
\begin{scope}
	\clip\Greg;
	\fill[Goldenrod!60] \Tcirc;
	\draw[very thick, cyan] \Tcirc;
\end{scope}


\draw[thick] \Greg;
\draw[ultra thick, Orchid!20, domain=-10:10] plot ({5*cos(\x)}, {5*sin(\x)});
\draw[thick, domain=4.7:5.3] plot ({\x*cos(10)}, {\x*sin(10)});
\draw[thick, domain=4.7:5.3] plot ({\x*cos(-10)}, {\x*sin(-10)});


\node at (-90:7) {$(\overline{P})$};
\end{scope}

\end{tikzpicture}
\caption{Gated decorated surfaces corresponding to moduli problems at a puncture are depicted. See \cref{rmk:A} for explanation of the ($A$) case.}
\label{fig:four-variations} 
\end{figure}
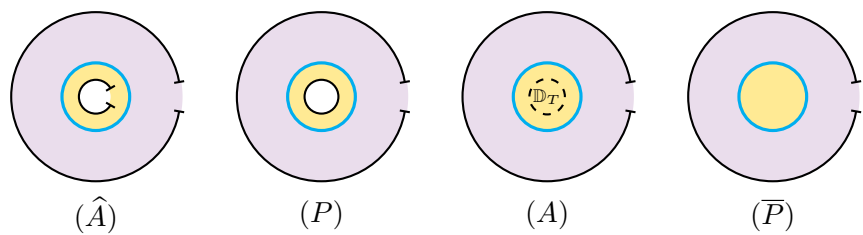

The $(\hat{A})$ space serves as our main workhorse in this paper, because it controls the others by regulation of $T$-framings.  The $(A)$ space is the preimage of the identity with respect to a multiplicative moment map $\mu$ valued in $T$.  The $(P)$ space is the quotient of the $(\hat{A})$ space by the $T$-action changing the framing. Accordingly, the $(\overline{P})$ space is the multiplicative Hamiltonian reduction of the $\hat{A}$ space. These inter-relations are summarized in the diagram below
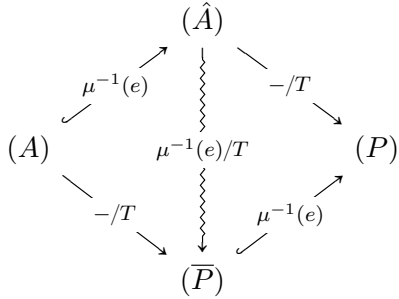
\[
\begin{tikzcd}[row sep=large, column sep = large]
& (\hat{A})\arrow["- / T" description]{dr} \arrow[rightsquigarrow, "\mu^{-1}(e)/T" description]{dd} \\
(A) \arrow[hook, "\mu^{-1}(e)" description]{ur} \arrow["-/T" description]{dr} && (P) \\
&(\overline{P}) \arrow[hook,"\mu^{-1}(e)" description]{ur}
\end{tikzcd}
\]

\begin{remark}
For simplicity, we have considered only a single puncture and no marked points in the above discussion.  In general, each puncture contributes one datum of each type, while each marked point contributes a factor of $N\backslash G$, with no condition. There are also ($X$) and $(\overline{X})$ spaces, which have ($P$) data at each puncture, but instead have a $B\backslash G$ factor for each marked point.  These are obtained from $(P)$ spaces by quotienting the $T$-action at each marked point. 
\end{remark}
\subsection{Construction via factorization homology}
Our starting point is the following fundamental result, which is a direct corollary of Theorem \cite[Theorem 1.2]{BZFN}:

\begin{theorem}\label{thm:BZFN-1}
For any $n$-manifold $X$, we have an equivalence of categories,
\[ \QC(\Ch(X)) = \int_{X}\Rep(G).\]
\end{theorem}

On the left we denote the category of quasi-coherent sheaves on the (ordinary, not decorated) character stack.  The integral notation on the righthand side denotes the factorization homology of $X$ with coefficients in the symmetric monoidal category $\Rep(G)$, which is here regarded as an $E_n$-algebra.  The theorem reconstructs the character stack $\Ch(X)$ on any $n$-manifold $X$ in terms of its factorization homology.

We may treat decorated character stacks analogously in the framework of \emph{stratified factorization homology}. Whereas an ordinary surface has only a single basic open disk (to which is attached $\Rep(G)$ above), a decorated surface has three basic open disks: one contained in $\decS_G$, one contained in $\decS_T$, and one straddling a wall.

Accordingly, factorization homology of decorated surfaces requires us to specify coefficients for each type of basic open disk, together with a \defterm{$\Disp$-algebra} --- a system of functors and higher coherences relating the coefficients.  For instance, in the undecorated case we were required to specify not only $\Rep(G)$, but also its $E_2$ (= braided monoidal) structure.  In order to recover \emph{decorated} character stacks, we will specify the so-called ``parabolic induction algebra'', being the triple of categories
$$
\Rep(\TBG) := (\Rep(G),\Rep(B),\Rep(T)),
$$
together with the pullback functors along $i \colon B\to G$ and $\pi \colon B\to T$.  The following theorem again follows as a corollary of \cite[Theorem 1.2]{BZFN}:
\begin{theorem}
We have an equivalence of categories,
\[ \QC(\Ch(\decS)) = \int_{S}\Rep(\TBG).\]
\end{theorem}

In the works \cite{BBJ18a, BBJ18b}, \cref{thm:BZFN-1} was taken as the starting place for a \emph{quantization procedure.} Namely, replacing $\Rep(G)$ with $\Repq(G)$, yields a functorial deformation quantization of~$\Ch(S)$. Such quantizations of character stacks via factorization homology were subsequently related to combinatorial Chern--Simons theory and Alekseev--Grosse--Schomerus algebras \cite{AGS} in \cite{BBJ18a}, double affine Hecke algebras \cite{Cherednik} in \cite{BBJ18b}, and to skein algebras and skein categories \cite{Turaev, Przytycki} in \cite{Coo19}.

In the present paper we extend this paradigm to the study of decorated surfaces by replacing $\Rep(\TBG)$ with the ``quantum parabolic induction algebra''.  This consists of the triple,
$$
\Repq(\TBG):=\hr{\Repq(G),\Repq(B),\Repq(T)}
$$
along with 
the ribbon braided tensor structure on $\Repq(G)$ and $\Repq(T)$, the tensor structure on $\Repq(B)$, and a braided tensor functor
$$
\Repq(G)\bt\rev{\Repq(T)} \longra Z(\Repq(B)),
$$
where $Z(\Repq(B))$ denotes the \emph{Drinfeld center} of the monoidal category $\Repq(B)$, and $\rev{\Repq(T)}$ carries the opposite of its standard braiding. This data defines a \emph{local coefficient system} for stratified factorization homology of \cite{AFT}, see also \cite[Section 3]{BJS}.  Hence, we make the following
\begin{defn}\label{def:qdec-ch-stack}
The \defterm{quantum decorated character stack} $\Zc(\decS)$ is the stratified factorization homology,
\[ \Zc(\decS) := \int_{\decS} \Repq(\TBG).\]
\end{defn}

The two most important features of stratified factorization homology are that it is functorial for stratified embeddings, and satisfies the excision property. To recall the excision property, we first recall that any \emph{cylindrical} decorated $\decX$ carries a monoidal structure by stacking in the cylinder direction, and that any embedding of $\decX$ into a neighborhood of the boundary of $\decS$ induces a $\Zc(\decX)$-module category structure on $\Zc(\decS)$.

Excision computes the stratified factorization homology of two decorated surfaces $\decS_1$ and $\decS_2$ glued over a cylindrical decorated surface $\decS_{12}$ as a relative tensor product
$$
\Zc(\decS) = \Zc(\decS_1) \underset{\Zc(\decS_{12})}{\bt} \Zc(\decS_2).
$$
For example, \cref{fig:D3} shows the equivalence
$$
\Zc(\DD_3) \simeq \Repq B\underset{\Repq G}{\bt}\Repq B\underset{\Repq G}{\bt}\Repq B,
$$
while~\cref{fig:D4-disk} illustrates
$$
\Zc(\DD_4) \simeq \Zc(\DD_3) \underset{\Zc(\DD_2)}{\bt}\Zc(\DD_3)
\qquad\text{and}\qquad
\Zc(\DD_2^\circ) \simeq \Zc(\DD_3)\!\!\!\!\!\!\!\!\underset{\Zc(\DD_2)\bt\Zc(\DD_2)}{\boxtimes}\!\!\!\!\!\!\!\! \Zc(\DD_3).
$$

 \begin{notation}
Given a set $\Gc$ of $m$ $G$-gates and $n$ $T$-gates, we will use the following abbreviations where convenient.
\[
\Repq(\Gc) = \Repq(G)^{m}\bt\Repq(T)^{n}=\Repq(G^m\times T^n).
\]
We refer to the action of $\Repq(\Gc)$ on $\Zc(\decS)$ induced by $\Gc\times I\into \decS$ as \defterm{disk insertion}.  To improve readability throughout the paper, we will write $G$, $T$, $G\times T$, etc. in place of $\Gc$ for emphasis as needed, and we will abbreviate ``$\Repq(\Gc)$-module category" and ``$\Repq(\Gc)$-module functor" by $\Gc$-module category, $\Gc$-module functor, respectively.
\end{notation}

A final important ingredient of stratified factorization homology is the existence of a canonical \defterm{distinguished object} $\Dist_\decS$ in each category $\Zc(\decS)$.  The assignment $\decS \mapsto \Zc(\decS)$ is functorial for stratified embeddings, and the morphism in factorization homology induced by the empty embedding $\emptyset \into \decS$ determines a functor $\Vect\to \Zc(\decS)$; its image on the one-dimensional vector space is, by definition, the distinguished object.  Classically, the distinguished object is precisely the structure sheaf of $\Ch(\decS)$, hence the distinguished object quantizes the structure sheaf to $\Zc(\decS)$, and the functor $\Hom(\Dist_\decS,-)$ quantizes the global sections functor.

\subsection{Aside on noncommutative algebraic geometry} Our main results are phrased in the language of noncommutative algebraic geometry, for which we follow~\cite{Smith}. The main idea is to treat an arbitrary abelian category $\cC$, such as $\Zc(\decS)$, as the category of sheaves on some putative noncommutative algebraic variety (more generally, stack), and to study $\cC$ functorially drawing on geometric intuition coming from the commutative case.  Hence, each definition below, when applied in the commutative case $q=1$, will recover precisely the usual geometric notion, and for generic $q$ will give its deformation.

In this context, a full subcategory $i \colon \cC\to\cD$ is called \defterm{open} if the inclusion functor $i$ admits an exact left adjoint.  The name ``open'' comes from the following basic example: given an inclusion $ i \colon U \into X$ of an open subvariety (more generally, an open substack), the pushforward $i_* \colon \QC(U)\to \QC(X)$ is fully faithful with exact left adjoint $i^*$.  For this reason, we will refer to $i^L$ as the \defterm{restriction functor}, and talk about restricting objects to open subcategories.  Important examples of open subcategories arise from Ore localizations: the inclusion $A[S^{-1}]\mod \subset A\mod$ has an exact left adjoint --- the localization functor --- given by tensoring over $A$ with $A[S^{-1}]$.

Dually, we may view an open subcategory $\cC$ as a quotient of $\cD$. Indeed $\ker(i^L)$ is a Serre subcategory, hence $i^L$ induces an equivalence $\cC\simeq \cD/\ker(i^L)$.  Classically, $\ker(i^*)$ consists of the sheaves supported on the complement of $U$.  The \defterm{geometric union} of open subcategories $i_1 \colon \cC_1\to\cD$ and $i_2 \colon \cC_2\to\cD$ is defined as
$\cC_1\dot{\cup}\cC_2 = \cD/(\ker i_1 \cap \ker i_2)$.  We emphasize that this is not simply the full subcategory obtained by taking unions of their objects.

We say that an open subcategory $\cC\subset \cD$ is a \defterm{chart} (alternatively, is \defterm{affine}) if $\cC$ admits a compact projective generator $X$, and hence an equivalence $\cC\simeq \End(X)^{\op}\mod$ with the category of modules for a (typically, noncommutative) algebra. Classically, this would mean that the corresponding open set is affine.  We say that the chart is \defterm{toric} if, moreover, the algebra $\End(X)^{\op}$ is a quantum torus, i.e. if there exists an isomorphism of algebras,
$$
\End(X)^{\op} \simeq \C \ha{x_1^{\pm 1}, \dots, x_m^{\pm 1}} / \ha{x_ix_j=q^{m_{ij}}x_jx_i},
$$
for some skew symmetric integer matrix $M=(m_{ij})$.

Given two charts $\cC$ and $\cC'$ of $\cD$, the \defterm{transition functor} $\phi_{\cC,\cC'} \colon \cC\to\cC'$ is the composite exact functor,
$$
\phi_{\cC,\cC'} \colon \cC\xrightarrow{i_\cC}\cD \xrightarrow{i_{\cC'}^L}\cC',
$$
where $i_\cC, i_{\cC'}$ are inclusion functors, and $i_{\cC'}^L$ is left adjoint to $i_{\cC'}$.
In the case $\cC=A\mod$ and $\cC'=A'\mod$ are affine charts, the transition functor $\phi_{\cC,\cC'} \colon \cC\to\cC'$ is given by tensoring with some $A-A'$ bimodule $M$.  Classically, if $\cC = \QC(U)$ and $\cC'=\QC(U')$ with $U,U'$ being open affine subvarieties of some variety $X$, the bimodule $M$ is merely the algebra of functions on the intersection $U\cap U'$.

An additional feature in our story is the presence of quantum categorical $G$- and $T$-actions coming from disk insertions through $G$ and $T$-gates.  We will say that a subcategory is a \defterm{$\Gc$-chart} (alternatively, is \defterm{$\Gc$-affine}) if it admits a compact projective $\Repq(\Gc)$-generator $X$, hence an equivalence,
\[
\cC\simeq \iEnd_\Gc(X)^{\op}\mod_{\Gc},
\]
where the notation $A\mod_\Gc$ is a shorthand for the category of $A$-modules internal to $\Repq(\Gc)$ -- i.e. the category of $A$-modules equipped with compatible actions of the quantum group attached to $\Gc$.  

Classically for a quotient stack $X/\Gc$, $QC(X)$ being $\Gc$-affine means that the map $X\to \bullet/\Gc$ is affine.  For instance $\bullet/\Gc$ itself is not affine, since $\QC(\bullet/\Gc) = \Rep(\Gc)$ does not admit a compact projective generator; however it is $\Gc$-affine, since the trivial representation is tautologically a compact projective $\Gc$-generator.  See \cref{sec:monadic} for a discussion of monadic reconstruction and internal endomorphism algebras, and \cref{sec:barrbeckwarmup} for discussion of familiar classical examples in this language.

\subsection{Charts and flips on $\Zc(\decS)$}
As explained in \cref{sec:strat-facthom}, the framework of stratified factorization homology may be understood as the specification of a universal property to be satisfied by an desired invariant of surfaces, and a theorem stating that this universal property can be solved in suitable settings.  The universal property itself is phrased in $\infty$-categorical terms, and its {\it a priori} solution is expressed via colimits in higher categories.

At first encounter, such a definition involving $\oo$-colimits may sound hopelessly abstract.  However, a number of techniques --- most notably the excision formula, and its monadic reformulations --- allow us to wrangle this abstract definition into a concrete enough form from which we can extract the calculus of quantum cluster algebras. Our strategy is roughly as follows; we will focus the exposition on $G=\SL_2$, however, we will explain the modifications for the $\PGL_2$-case as well.

To begin, we treat the case of the $n$-gon $\DD_n$, with a unique gate in the $G$-region and in each $T$-region.  We compute the internal endomorphism algebra $\Oq(\Conf_n)=\iEnd_{G\times T^n}(\Dist_\decS)$, identifying it with the braided tensor product of $n$ copies of an algebra $\Oq(N\backslash G)$.  The functor $\iHom_{G\times T^n}(\Dist_\decS,-)$, of $G\times T^n$-equivariant global sections, defines an open embedding,
$$
i \colon \Zc(\DD_n) \into \Oq(\Conf_n)\mod_{G\times T^n}
$$
In ~\cref{disk-computation} we compute the orthogonal of $\Zc(\DD_n)$ to be the Serre subcategory of \defterm{torsion} modules (see \cref{def:torsion}).  

We consider triangulations $\tri$ of $\DD_n$.  We define $U_q(\g)$-invariant Ore sets $S_\tri$ in $\Oq(\Conf_n)$, whose generators correspond to edges of $\tri$, and which quantize corresponding cluster $\Ac$-coordinates on $G \backslash (N\backslash G)^{\times n}$.   In fact, the most important cases are $n=2, 3$, which each have a unique triangulation. We denote these special triangulations by $\underline{\DD_2}$ (a degenerate triangulation with one edge, but no triangles) and $\underline{\DD_3}$, respectively. For $n=2$, $S_{\underline{D_2}}$ consists of monomials in a single element denoted $D_{12}$, and for $n=3$ $S_{\underline{\DD_3}}$ consists of monomials in elements denoted $D_{12},D_{23},D_{13}$, and coming from the three inclusions of $\DD_2$ as an edge of $\DD_3$.

In~\cref{sec:disk_charts}, we define an open subcategory $\Zc(\tri)\subset \Zc(\DD_n)$, by the condition that the elements of $S_\tri$ act invertibly.
We show that the restriction $\Dist_\tri=\Res_{\Zc(\tri)}(\Dist_{\DD_n})$ of the distinguished object to $\Zc(\tri)$ is a $G\times T^n$-compact-projective generator (we note that it is not so for $\Zc(\DD_n)$), hence yielding an equivalence of categories
\[\Zc(\tri) \simeq \Oq(\Conf_n)[S_\tri^{-1}]\mod_{G\times T^n},\]
and sandwiching $\Zc(\DD_n)$ between $G\times T^n$-affine categories:
$$
\Oq(\Conf_n)[S_\tri^{-1}]\mod_{G\times T^n} \subset \Zc(\DD_n) \subset \Oq(\Conf_n)\mod_{G\times T^n},
$$
Crucially, we also show that the functor of taking $G$-invariants is conservative on $\Zc(\tri)$ (we note again, it is not so for $\Zc(\DD_n)$), hence we obtain an equivalence of categories,
$$
\Zc(\tri)\simeq \iEnd_{T^n}(\Dist_\tri)\mod_{T^n} \cong \hr{\Oq(\Conf_n)[S_\tri^{-1}]}^{U_q(\g)}\mod_{T^n}.
$$
Henceforth let us abbreviate $\zeta(\tri):=\iEnd_{T^n}(\Dist_\tri)\mod_{T^n}$.  Returning to the special cases $n=2,3$, we calculate isomorphisms,
$$
\zeta(\underline{\DD_2}) \cong \C_q\ha{D_{12}^{\pm1}}
\qquad\text{and}\qquad
\zeta(\underline{\DD_3}) \cong \C_q\ha{D_{12}^{\pm 1},D_{23}^{\pm 1},D_{13}^{\pm 1}}/I,
$$
where $\C_q = \C(q^{\pm1/2})$, and $I$ consists of the q-commutation relations,
\[D_{13} D_{12}=q^{-\frac12}D_{13} D_{12}, \qquad D_{23} D_{13}=q^{-\frac12}D_{13} D_{23}, \qquad D_{23} D_{12}=q^{-\frac12}D_{12} D_{23}.
\]

We then turn to the construction of charts on a general simple decorated surface $\Zc(\decS)$, for which we appeal to excision.  In simple cases like $\DD_n$, one might imagine building an arbitrary surface inductively by gluing triangles of a triangulation in one by one. To give a more uniform formula and to allow for self-gluings, however, we elect instead to glue all triangles together in one go, by writing
\beq\label{eqn:gluing}
\decS = (\DD_3)^{\sqcup t} \bigsqcup_{(\DD_2)^{\sqcup 2\ell}} (\DD_2)^{\sqcup \ell}.
\eeq
Excision gives an equivalence of categories,
\begin{equation}\label{eqn:tensor-product}
\Zc(\decS) \simeq \Zq(\DD_3)^{t} \bigboxtimes_{\Zq(\DD_2)^{2\ell}}\Zq(\DD_2)^{\ell},
\end{equation}
Hence we define the open subcategory $\Zc(\tri)\subset \Zc(\decS)$ by gluing the open subcategories for triangles over digons, accordingly:
\begin{equation}\label{eqn:tilde-tensor-product}
\Zc(\tri) := \Zc(\underline{\DD_3})^{ t} \bigboxtimes_{\Zc(\underline{\DD_2})^{2\ell}}\Zc(\underline{\DD_2})^{\ell},
\end{equation}

\begin{remark}
The reader is encouraged to view \cref{eqn:tensor-product} in analogy with the formula for Hochschild homology of an associative algebra $A$ with coefficients in its bimodule $M$,
\beq
HH(A,M) = A \otimes_{A\otimes A^{op}} M.
\eeq
In this analogy, $A$ is the monoidal category obtained from the union of all the digons, $M$ is the bimodule category obtained as the unions of all the triangles, and the bimodule structure comes from the two inclusions of each edge into the triangles it borders.
\end{remark}

Having defined the charts by gluing, we proceed to computing their global sections.  We use the following particular case of~\cite[Theorem 4.12]{BBJ18a}:
\begin{lemma}
\label{lem:excision}
Let $\Cc$ be a braided monoidal category with a braided commutative algebra object $A \in \Cc$, and a pair $M$, $N$, of $A$-algebras in $\Cc$. Then we have an equivalence of categories,
$$
M\mod_\Cc \underset{A\mod_\Cc}\bt N\mod_{\Cc} \simeq \hr{M \otimes_A N^\op}\mod_\Cc.
$$
\end{lemma}

However, in order to apply \cref{lem:excision}, we turn out to require an intermediate step of opening additional $T$-gates to ensure that the algebras we wish to tensor over is in fact braided commutative.  This being done, we can glue using \cref{lem:excision}, and finally close all but one gate in each $T$-region to obtain our main results.  By \defterm{opening a gate}, we refer to the operation of pulling the $T$-action on $\Zq(\decS)$ back through the tensor functor,
$$
\otimes \colon \Repq(T)^2 \to \Repq(T).
$$
In order to \defterm{close a gate}, we apply the functor of taking $T$-invariants at the gate.

Now, let us open two gates in each $T$-region of $\DD_3$ and consider the category $\Zq(\underline{\DD_3})$ as a ${T^6}$-module category. The result is still $T^6$-affine with the $T^6$-progenerator $\Dist_{\underline{\DD_3}}$, and in~\cref{basic-hex-rels} we calculate the algebra,
$$
\widetilde \zeta(\underline{\DD_3}) = \iEnd_{T^6}\hr{\Dist_{\underline{\DD_3}}}\mod_{T^6}.
$$
It is again a quantum torus, now with six generators.  We have equivalences of categories
$$
\Zq(\underline{\DD_3}) \simeq \widetilde \zeta(\underline{\DD_3})\mod_{T^6} \simeq \zeta(\underline{\DD_3})\mod_{T^3}. 
$$
We similarly define an algebra $\widetilde \zeta(\underline{\DD_2})$ by opening a pair of $T$-gates in each $T$-region of a digon $\DD_2$, and have equivalences
$$
\Zq(\underline{\DD_2}) \simeq \widetilde\zeta(\underline{\DD_2})\mod_{T^4} \simeq \zeta(\underline{\DD_2})\mod_{T^2}.
$$
Thanks to~\cref{lem:excision}, we conclude that the chart $\Zc(\tri)$ is toric and \emph{$\Gc$-affine}, that is
$$
\Zc(\tri) \simeq \widetilde \zeta(\tri)\mod_{T^{6t}}
$$
for a quantum torus algebra $\widetilde \zeta(\tri)$ obtained as a braided tensor product as in \cref{lem:excision}. We provide an explicit presentation of $\widetilde \zeta(\tri)$ in~\cref{pre-glue-alg}.

We now choose a distinguished gate at every $T$-region of $\decS$, and close all but the distinguished $T$-gate.  This fixes notchings of $\decS$ and $\tri$; we denote the altter by $\trib$.  We obtain a quantum torus $\zeta(\trib)$ from $\widetilde \zeta(\tri)$, as the subalgebra of invariants with respect to the $T$-actions at all non-distinguished $T$-gates.

The quantum torus $\zeta(\trib)$ has a simple presentation, which depends only on the isotopy class of $\trib$.  There is one generator $Z_e$ for each edge $e$ of the notched triangulation, and an additional generator $\alpha_v$ for each puncture. The $q$-commutation relations between generators are described by simple formulas depending on the incidence relations and total ordering at each vertex, see \cref{thm:zeta-rel}.

\begin{remark}
We note that $\DD_n$ has no punctures, and the resulting algebra $\zeta(\trib)$ does indeed coincide with $\zeta(\tri)$ computed above as a subalgebra of $\Oq(\Conf_n)[S_\tri^{-1}]$. 
\end{remark}

Throughout the discussion above, we have fixed $G=\SL_2$.  In fact, our computations immediately give rise to charts $\Zc_{\PGL_2}(\tri)\subset\Zc_{\PGL_2}(\decS)$ as well, with the corresponding quantum tori $\zeta_{\PGL_2}(\trib)$ being realized naturally as subalgebras of $\zeta_{\SL_2}(\trib)$ (see \cref{cor:PGL2-D2-D3-gen} and
\cref{prop:cover}).

\begin{theorem}\label{thm:Ahat} Let $\decS$ be a simple decorated surface with $k$ $T$-regions, with gating $\Gc=T^k$ consisting of a single gate in each $T$-region.  Let $G$ be either $\SL_2$ or $\PGL_2$.
\begin{enumerate}
\item Each isotopy class of triangulations $\tri$ determines a $\Gc$-toric chart $\Zc(\tri) \subset \Zc_G(\decS)$. The additional choice of the isotopy class of a notched triangulation $\trib$ defines a quantum torus~$\zeta_G(\trib)$, and equivalence of categories:
$$
\Zc_G(\tri) \simeq \zeta_G(\trib)\mod_{T^k}.
$$
\item Given two notched triangulations $\trib$ and $\tri'_\bullet$ differing by a flip of a single edge, the functor $\Zc_G(\tri) \to \Zc_G(\tri')$ is induced by an explicit birational isomorphism $\zeta(\trib,\trib')$ between $\zeta_G(\trib)$ and $\zeta_G(\trib')$ described in~\cref{prop:Ahat-flip} for the case $G=\SL_2$, and by its restriction in the case $\PGL_2$.
\end{enumerate}
\end{theorem}

\begin{remark}\label{rmk:A}
Consider the cluster $K_2$-variety structure on the moduli space of twisted $\SL_2$-local systems on $S$, see~\cite{FG06} and~\cite{FST08}. Let $\Ac_\tri$ be a cluster chart arising from the ideal triangulation $\tri$. If the surface $S$ has only boundary marked points and no punctures, the cluster algebra $\Ac_\tri$ admits quantization in the sense of~\cite{BZ05}, see~\cite{Mul16}. In that case, we denote the quantum chart arising from the triangulation $\tri$ by $\Ac^q_\tri$. Then the following observations are immediate.
\begin{enumerate}
    \item If $S$ has no punctures, $\zeta(\tri)\cong\Ac^q_\tri$, and the flip between $\zeta(\trib)$ and $\zeta(\trib')$ coincides with the corresponding quantum cluster mutation.
    \item If $S$ has punctures, the elements $\alpha_v$ are not central in $\zeta(\trib)$. However, upon setting $q=1$, the algebra $\zeta^{q=1}(\trib)$ becomes commutative and one may consider an ideal $I \subset \zeta^{q=1}(\trib)$ generated by all elements of the form $\alpha_v-1$. Then, $\zeta^{q=1}(\tri_\bullet)/I \cong \Ac_\tri$, and the flip between $\zeta(\tri_\bullet)$ and $\zeta(\tri'_\bullet)$ descends to the cluster mutation between the charts $\Ac_\tri$ and $\Ac_{\tri'}$.
\end{enumerate}
In this way, one recovers all cluster charts which are labelled by triangulations without vertices incident to two or more tagged arcs; we refer the reader to~\cite{FST08} for the definition of a tagged arc.
\end{remark}

Now, let $\decS$ be a simple decorated surface with $p$ punctures and $m$ marked points, with a single gate in each $T$-region. Let $\chi(\tri) = \zeta(\trib)^{T^p}$ be the subalgebra of invariants with respect to the $T$-action at the punctures. Note that $\chi(\tri)$ is still a quantum torus and depends only on the triangulation $\tri$ rather than on the notched triangulation $\trib$. In this way, we can give an alternative formulation of \cref{thm:Ahat} using the $\chi(\tri)$-tori instead of $\zeta(\trib)$.

Before proceeding to this formulation in \cref{thm:X}, let us briefly discuss the case $G=\PGL_2$, and its relation to $G=\SL_2$. Recall that the algebra $\widetilde \zeta_{\SL_2}(\tri)$ carries a $T^{6t}$-action, where $t$ is the number of triangles in $\tri$. This action endows $\widetilde \zeta_{\SL_2}(\tri)$ with a $\Z^{6t}$-grading. Then $\widetilde \zeta_{\PGL_2}(\tri)$ coincides with the subalgebra of $\widetilde \zeta_{\SL_2}(\tri)$, generated by elements of even degree with respect to the $T$-action at any $T$-gate. Closing all gates but the distinguished gate in each $T$-annulus, we obtain an embedding $\zeta_{\PGL_2}(\trib) \subset \zeta_{\SL_2}(\trib)$. Furthermore, let $\chi_{\PGL_2}(\tri) \subset \zeta_{\PGL_2}(\trib)$ denote the subalgebra of invariants with respect to the $T$-action at punctures as described previously in the $\SL_2$ case. We have the following commutative diagram:
$$
\begin{tikzcd}[row sep=large, column sep = large]
\chi_{\SL_2}(\tri) \arrow[hook]{r} & \zeta_{\SL_2}(\trib)\\
\chi_{\PGL_2}(\tri) \arrow[hook]{u} \arrow[hook]{r} & \zeta_{\PGL_2}(\trib) \arrow[hook]{u}
\end{tikzcd}
$$

We describe the generators of the subalgebra $\chi_{\SL_2}(\tri) \subset \zeta_{\SL_2}(\trib)$ in \cref{sec:PGL2-comparison}, where we show that $\chi_{\SL_2}(\tri)$ is a free module over $\chi_{\PGL_2}(\tri)$ of rank $2^N$, where $N = \dim H_1(S,M)$ is the dimension of the first homology of $S$ relative to the subset $M$ of marked points.

Finally, let us recall from~\cite{FG06, FG09a}, the quantum $\Xc$-variety structure on the moduli space of $\PGL_2$-local systems with pinnings on a marked surface $S$. Let $\Xc^q_\tri$ denote the quantum cluster chart arising from the triangulation $\tri$. Then we obtain the following result.

\begin{theorem}
\label{thm:X}
Let $\decS$ be a simple decorated surface with $m$ singly gated marked points.  Let either $G=\SL_2$ and $m\geq 1$, or $G=\PGL_2$ and $m\geq 0$.
\begin{enumerate}
\item The isotopy class of each triangulation $\tri$ of $\decS$ determines a quantum toric chart $\Zc_{G}(\tri)\subset \Zc_{G}(\decS)$ with quantum torus $\chi_{G}(\tri)$ and an equivalence of categories:
$$
\Zc_{G}(\tri) \simeq \chi_{G}(\tri)\mod_{T^m}.
$$
\item Given two such triangulations $\trib$ and $\tri'_\bullet$ differing by a single flip, the functor $\Zc_G(\tri) \to \Zc_G(\tri')$ is the birational isomorphism obtained by restricting the $\hat{A}$-flip $\zeta(\trib,\trib')$ between $\zeta(\trib)$ and $\zeta(\trib')$ described in~\cref{prop:Ahat-flip} to the subalgebras $\chi(\tri)$, $\chi(\tri')$.
\item We have isomorphisms $\chi_{\PGL_2}(\tri) \cong \Xc^q_\tri$, and the flip between $\chi_{\PGL_2}(\tri)$ and $\chi_{\PGL_2}(\tri')$ coincides with the quantum cluster mutation of $\Xc$-variables between $\Xc^q_\tri$ and $\Xc^q_{\tri'}$.
\end{enumerate}
\end{theorem}

\begin{remark}\label{rmk:X}
One may also consider the subalgebra,
$$
\eta_G(\tri) := \chi_G(\tri)^{T^m} \subset \chi_{G}(\tri),
$$
of invariants with respect to the remaining $T$-actions. For $G=\PGL_2$, one obtains an equivalence of categories
$$
\Zc_{G}(\tri) \simeq \eta_G(\tri)\mod,
$$
and the quantum tori $\eta_G(\tri)$ coincide with quantum cluster charts on the $\Xc$-variety considered in~\cite{FG06}. In the language of cluster algebras, one obtains the quiver for $\eta_G(\tri)$ from the one for $\chi_G(\tri)$ by erasing all of the frozen nodes.

In the case $G=\SL_2$, some additional care is needed in closing the final gate, owing to the non-trivial center of $\SL_2$.  See \cref{rmk:SL2-eta} for more details.
\end{remark}

\begin{remark}\label{rmk:clust-ens}
Setting $q=1$, the homomorphism
$
\chi^{q=1}_{\PGL_2}(\tri) \hookrightarrow \zeta^{q=1}_{\SL_2}(\trib)/I
$
described above coincides with the cluster ensemble map between the $\Xc$- and $\Ac$-cluster structures discussed above. In case $S$ has no punctures, the homomorphism
$
\chi_{\PGL_2}(\tri) \hookrightarrow \zeta_{\SL_2}(\trib)
$ in \cref{thm:X}
coincides with the quantum cluster ensemble map.  In particular our construction may be regarded as the extension of the cluster ensemble map, to surfaces with punctures and to $q\neq 1$.
\end{remark}

To connect with cluster $\Ac$-varieties in previous remarks, we have set $q=1$.  This is necessary merely because the ideal $I$ is only a one-sided ideal for $q\neq 1$, which is in turn a manifestation of the fact that in the presence of punctures, $\Ac$-varieties are not naturally Poisson.  The promotion to $\hat{\Ac}$-varieties by allowing non-trivial monodromy of the $T$-connection has precisely fixed this issue, by adding a new generator for each puncture.  Classically $\hat{\Ac}$ varieties are Poisson (and in fact symplectic in the absence of marked points), and we have constructed their deformation quantizations.

Still, one might wonder if quantum $\Ac$-varieties nevertheless appear directly in factorization homology terms.  Topologically, the picture to consider is the disk labelled $(A)$ in \cref{fig:four-variations}, which we will denote here by $\DD^\bullet$; we will also denote by $\DD^\circ$ the disk labelled $(\hat{A})$.  We can include the disk $\DD_T$ into the filled inner $T$-region of $\DD^\bullet$ as depicted in the figure, and thereby obtain a functor $\Repq(T)\to\Zc(\DD^\bullet)$.  However, since $\DD_T$ does not enter $\DD^\bullet$ through a gate, this functor is \emph{not} a $T$-module functor.  We therefore \emph{can} construct $\iEnd_T(\Dist_{\DD^\bullet})$ formally as we do in the presence of a gate, however owing to this lack of module structure, it is only a plain object -- namely $\iEnd_T(\Dist_{\DD^\circ})/\langle\alpha_p-1\rangle\in\Repq(T)$ -- and not an algebra object.

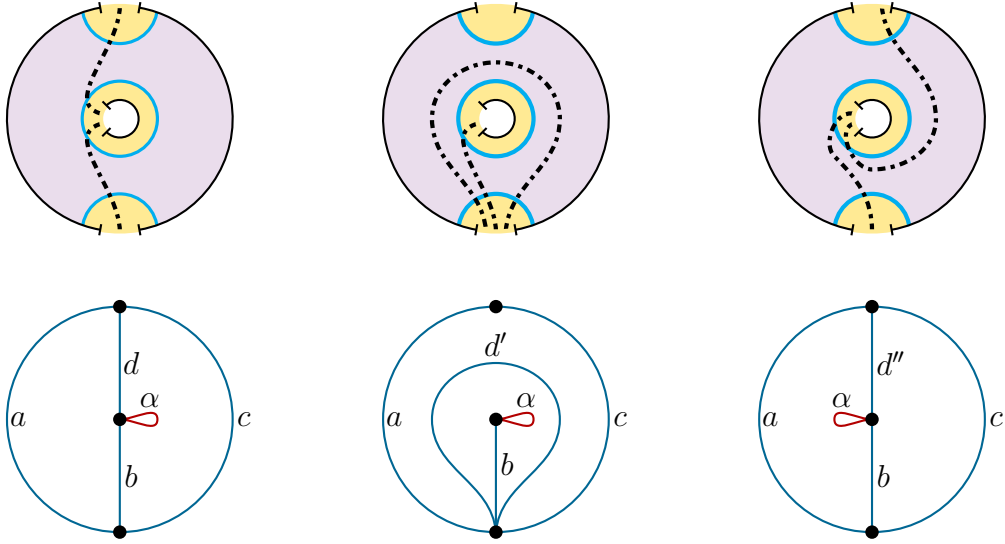
\begin{figure}[t]
\begin{tikzpicture}[every node/.style={inner sep=0.5, thick, circle}, thick, x=0.5cm, y=0.5cm]

\def\Gcirc{(0,0) circle (3) (0,0) circle (.5)};
\fill[Orchid!20, even odd rule] \Gcirc;
\def\Tone{(90:3) circle (1)};
\def\Ttwo{(-90:3) circle (1)};
\def\Tthree{(0,0) circle (1) (0,0) circle (.5)};

\begin{scope}
	\clip\Gcirc;
	\fill[Goldenrod!60, even odd rule] \Tone \Ttwo \Tthree;
	\draw[very thick, cyan] \Tone \Ttwo (0,0) circle (1);
\end{scope}

\draw[ultra thick, dashdotted] (0,3) to [out=-90,in=165] (165:0.5);
\draw[ultra thick, dashdotted] (0,-3) to [out=90,in=-165] (-165:0.5);

\draw[ultra thick, Goldenrod!60] (135:0.5) arc (135:225:0.5);
\draw[ultra thick, Goldenrod!60] (-100:3) arc (-100:-80:3);
\draw[ultra thick, Goldenrod!60] (80:3) arc (80:100:3);

\draw (-135:0.5) arc (-135:135:0.5);
\draw[domain=0.35:0.66] plot ({\x*cos(-135)}, {\x*sin(-135)});
\draw[domain=0.35:0.65] plot ({\x*cos(135)}, {\x*sin(135)});
\draw (-80:3) arc (-80:80:3);
\draw[domain=2.85:3.15] plot ({\x*cos(-80)}, {\x*sin(-80)});
\draw[domain=2.85:3.15] plot ({\x*cos(80)}, {\x*sin(80)});
\draw (100:3) arc (100:260:3);
\draw[domain=2.85:3.15] plot ({\x*cos(100)}, {\x*sin(100)});
\draw[domain=2.85:3.15] plot ({\x*cos(260)}, {\x*sin(260)});

\begin{scope}[shift={(10,0)}]

\def\Gcirc{(0,0) circle (3) (0,0) circle (.5)};
\fill[Orchid!20, even odd rule] \Gcirc;
\def\Tone{(90:3) circle (1)};
\def\Ttwo{(-90:3) circle (1)};
\def\Tthree{(0,0) circle (1) (0,0) circle (.5)};

\begin{scope}
	\clip\Gcirc;
	\fill[Goldenrod!60, even odd rule] \Tone \Ttwo \Tthree;
	\draw[ultra thick, cyan] \Tone \Ttwo (0,0) circle (1);
\end{scope}

\draw[ultra thick, dashdotted] (0,-3) to [out=90,in=-165, shift left=1pt] (-165:0.5);
\draw[ultra thick, dashdotted] (-85:3) to [out = 85, in =-90] (1.7,0) to [out=90,in=0] (0,1.5) to [out=180,in=90] (-1.7,0) to [out=-90,in=95] (-95:3);

\draw[ultra thick, Goldenrod!60] (135:0.5) arc (135:225:0.5);
\draw[ultra thick, Goldenrod!60] (-100:3) arc (-100:-80:3);
\draw[ultra thick, Goldenrod!60] (80:3) arc (80:100:3);

\draw (-135:0.5) arc (-135:135:0.5);
\draw[domain=0.35:0.66] plot ({\x*cos(-135)}, {\x*sin(-135)});
\draw[domain=0.35:0.65] plot ({\x*cos(135)}, {\x*sin(135)});
\draw (-80:3) arc (-80:80:3);
\draw[domain=2.85:3.15] plot ({\x*cos(-80)}, {\x*sin(-80)});
\draw[domain=2.85:3.15] plot ({\x*cos(80)}, {\x*sin(80)});
\draw (100:3) arc (100:260:3);
\draw[domain=2.85:3.15] plot ({\x*cos(100)}, {\x*sin(100)});
\draw[domain=2.85:3.15] plot ({\x*cos(260)}, {\x*sin(260)});

\end{scope}








\begin{scope}[shift={(20,0)}]

\def\Gcirc{(0,0) circle (3) (0,0) circle (.5)};
\fill[Orchid!20, even odd rule] \Gcirc;
\def\Tone{(90:3) circle (1)};
\def\Ttwo{(-90:3) circle (1)};
\def\Tthree{(0,0) circle (1) (0,0) circle (.5)};

\begin{scope}
	\clip\Gcirc;
	\fill[Goldenrod!60, even odd rule] \Tone \Ttwo \Tthree;
	\draw[ultra thick, cyan] \Tone \Ttwo (0,0) circle (1);
\end{scope}

\draw[ultra thick, dashdotted] (0,-3) to [out=90,in=-45] (-1,-1) to [out=135,in=165, shift left=1pt] (165:0.5);
\draw[ultra thick, dashdotted] (85:3) to [out = -85, in = 90] (1.7,0) to [out=-90,in=0] (0,-1.35) to [out=180,in=-165] (-165:0.5);

\draw[ultra thick, Goldenrod!60] (135:0.5) arc (135:225:0.5);
\draw[ultra thick, Goldenrod!60] (-100:3) arc (-100:-80:3);
\draw[ultra thick, Goldenrod!60] (80:3) arc (80:100:3);

\draw (-135:0.5) arc (-135:135:0.5);
\draw[domain=0.35:0.66] plot ({\x*cos(-135)}, {\x*sin(-135)});
\draw[domain=0.35:0.65] plot ({\x*cos(135)}, {\x*sin(135)});
\draw (-80:3) arc (-80:80:3);
\draw[domain=2.85:3.15] plot ({\x*cos(-80)}, {\x*sin(-80)});
\draw[domain=2.85:3.15] plot ({\x*cos(80)}, {\x*sin(80)});
\draw (100:3) arc (100:260:3);
\draw[domain=2.85:3.15] plot ({\x*cos(100)}, {\x*sin(100)});
\draw[domain=2.85:3.15] plot ({\x*cos(260)}, {\x*sin(260)});

\end{scope}

\begin{scope}[shift={(0,-8)}]

\draw[MidnightBlue] (0,0) circle (3);
\draw[MidnightBlue] (0,-3) to (0,3);
\draw[red!70!black] (0,0) to [out=0, in=-90] (1,0) to [out=90,in=0] (0,0);
\draw[fill] (0,-3) circle (0.15);
\draw[fill] (0,0) circle (0.15);
\draw[fill] (0,3) circle (0.15);

\node at (-2.7,0) {\large $a$};
\node at (0.3,-1.5) {\large $b$};
\node at (3.3,0) {\large $c$};
\node at (0.3,1.5) {\large $d$};
\node at (0.8,0.5) {\large $\alpha$};

\end{scope}

\begin{scope}[shift={(10,-8)}]

\draw[MidnightBlue] (0,0) circle (3);
\draw[MidnightBlue] (0,-3) to (0,0);
\draw[MidnightBlue] (0,-3) to [out = 85, in =-90] (1.7,0) to [out=90,in=0] (0,1.5) to [out=180,in=90] (-1.7,0) to [out=-90,in=95] (0,-3);
\draw[red!70!black] (0,0) to [out=0, in=-90] (1,0) to [out=90,in=0] (0,0);
\draw[fill] (0,-3) circle (0.15);
\draw[fill] (0,0) circle (0.15);
\draw[fill] (0,3) circle (0.15);

\node at (-2.7,0) {\large $a$};
\node at (0.3,-1.2) {\large $b$};
\node at (3.3,0) {\large $c$};
\node at (0,2) {\large $d'$};
\node at (0.8,0.5) {\large $\alpha$};

\end{scope}

\begin{scope}[shift={(20,-8)}]

\draw[MidnightBlue] (0,0) circle (3);
\draw[MidnightBlue] (0,-3) to (0,0);
\draw[MidnightBlue] (0,3) to (0,0);
\draw[red!70!black] (0,0) to [out=180, in=90] (-1,0) to [out=-90,in=180] (0,0);
\draw[fill] (0,-3) circle (0.15);
\draw[fill] (0,0) circle (0.15);
\draw[fill] (0,3) circle (0.15);

\node at (-2.7,0) {\large $a$};
\node at (0.3,-1.5) {\large $b$};
\node at (3.3,0) {\large $c$};
\node at (0.5,1.5) {\large $d''$};
\node at (-0.8,0.5) {\large $\alpha$};

\end{scope}

\end{tikzpicture}
\caption{Notched triangulations (above), corresponding to notched ideal triangulations (below), of the punctured digon $\DD^\circ_2$.}
\label{fig:punct-disk}
\end{figure}

\begin{example}\label{ex:punc-digon}
Consider a punctured digon $\DD^\circ_2$. Three isotopy classes of its notched triangulations are shown in the top row of~\cref{fig:punct-disk}, the bottom row of the same figure shows the corresponding notched ideal triangulations. Let $\zeta_{\SL_2} := \zeta_{\SL_2}(\tri^l_\bullet)$, where $\tri^l_\bullet$ is the triangulation on the left. Then
$$
\zeta_{\SL_2} = \C[q^{\pm\frac12}]\ha{Z_a^{\pm1},Z_b^{\pm1},Z_c^{\pm1},Z_d^{\pm1},\alpha^{\pm1}}
$$
with the following relations:
\begin{align*}
Z_a Z_b &= q^{-\frac12} Z_b Z_a,  &  Z_b Z_c &= q^{-\frac12}Z_cZ_b,  &  \alpha  Z_a &= Z_a \alpha, \\
Z_a Z_c &= Z_c Z_a,  &  Z_b Z_d &= q^{-\frac12} Z_d Z_b,  &  \alpha Z_b &= q Z_b \alpha, \\
Z_a Z_d &= q^{\frac12} Z_d Z_a,  &  Z_c Z_d &= q^{-\frac12} Z_d Z_c  &  \alpha Z_c &= Z_c \alpha, \\[3pt]
&& && \alpha Z_d &= q Z_d \alpha.
\end{align*}

Let us also set $\zeta^c_{\SL_2} := \zeta_{\SL_2}(\tri^c_\bullet)$ and $\zeta^r_{\SL_2} := \zeta_{\SL_2}(\tri^r_\bullet)$, where $\tri^c_\bullet$ and $\tri^r_\bullet$ are the triangulations in the left and on the right of~\cref{fig:punct-disk}. Then the quantum tori $\zeta^c_{\SL_2}$ and $\zeta^r_{\SL_2}$ are obtained from $\zeta_{\SL_2}$ by replacing the generator $Z_d$ with $Z_{d'}$ and $Z_{d''}$ respectively. Then
$$
Z_{d''} = q^{-1}Z_d\alpha^{-1},
$$
while the flip between $\zeta_{\SL_2}$ and $\zeta^c_{\SL_2}$ reads
$$
Z_{d'} = \nord{Z_d^{-1}\alpha Z_aZ_b} + \nord{Z_d^{-1}Z_cZ_b} = q^{-\frac34}Z_d^{-1}\alpha Z_aZ_b + q^{\frac14}Z_d^{-1}Z_cZ_b,
$$
where $\nord{\;\;}$ denotes the \emph{Weyl ordering} introduced in~\cref{weyl-order}, and we obtain relations 
\begin{align*}
Z_c Z_{d'} &= q Z_{d'} Z_c, & Z_a Z_{d'} &= q^{-1} Z_{d'} Z_a, & Z_b Z_{d'} &= Z_{d'} Z_b, & \alpha Z_{d'} &= Z_{d'} \alpha.
\end{align*}

Let us now focus on the triangulation $\tri_\bullet^l$, and describe the quantum subtori of $\zeta_{\SL_2}$ obtained from the latter by taking $T$-invariants at the puncture or marked points, and by replacing $\SL_2$ with $\PGL_2$. We have
$$
\zeta_{\PGL_2} = \C\big[q^{\pm\frac12}\big]\ha{Z_a^{\pm2},Z_b^{\pm2},\alpha^{\pm2},(Z_aZ_bZ_d)^{\pm1},(Z_bZ_cZ_d\alpha)^{\pm1}} \subset \zeta_{\SL_2}
$$
and
$$
\chi_{\SL_2} = \C\big[q^{\pm\frac12}\big]\ha{Z_a^{\pm1},Z_c^{\pm1},\alpha^{\pm1},(Z_bZ_d^{-1})^{\pm1}} \subset \zeta_{\SL_2}.
$$
Then, setting
$$
X_a = \nord{Z_aZ_b^{-1}Z_d}, \quad X_b = \nord{\alpha Z_aZ_c^{-1}}, \quad X_c = \nord{\alpha^{-1}Z_cZ_d^{-1}Z_b}, \quad X_d = \nord{\alpha Z_cZ_a^{-1}},
$$
we obtain
$$
\chi_{\PGL_2} = \C\big[q^{\pm\frac12}\big]\ha{X_a^{\pm1},X_b^{\pm1},X_c^{\pm1},X_d^{\pm1}} = \zeta_{\PGL_2} \cap \chi_{\SL_2}.
$$
Finally, we also have
\begin{align*}
\eta_{\SL_2} &= \C\big[q^{\pm\frac12}\big]\ha{(Z_aZ_c^{-1})^{\pm1},\alpha^{\pm1}} \subset \chi_{\SL_2}, \\
\eta_{\PGL_2} &= \C\big[q^{\pm\frac12}\big]\ha{X_b^{\pm1},X_d^{\pm1}} \subset \chi_{\PGL_2}.
\end{align*}

We end this example with the following observation. In~\cite{SS19}, the third and fourth authors showed that the quantum group $U_q(\sl_n)$ is closely related to the quantum cluster Poisson variety on the $P$-space for $G = PGL_n$. Let us recall this relation in the setup of the present paper. In case $n=2$, the Drinfeld double $\mathfrak{D}$ of the quantum Borel $U_q(\mathfrak b) \subset U_q(\sl_2)$ is the algebra
$$
\mathfrak{D} = \C(q)\ha{e,f,k^{\pm1}_\omega,k^{\pm1}_{\omega'}}
$$
with relations
\begin{align*}
k_\omega e &= q e k_\omega, & k_{\omega'} e &= q^{-1} e k_{\omega'}, & k_\omega k_{\omega'} &= k_{\omega'} k_\omega, \\
k_\omega f &= q^{-1} f k_\omega, & k_{\omega'} f &= q f k_{\omega'},
\end{align*}
and
$$
[e,f] = (q^{-1}-q)(k_\alpha - k_{\alpha'}),
$$
where $k_\alpha = k_\omega^2$ and $k_{\alpha'} = k_{\omega'}^2$. Here we denote by $\omega$ the fundamental weight of $\sl_2$, and by $\alpha = 2\omega$ the positive root. Then one has an injective homomorphism $\iota \colon \mathfrak{D} \hookrightarrow \chi_{\SL_2}$ given by\footnote{In~\cite{SS19}, the above homomorphism is written in terms of $E = \mathbf{i}k_\omega e k_{\omega'}^{-1}$, $F = \mathbf{i} k_{\omega'} f k_\omega^{-1}$, $K = k_\alpha$, and $K' = k_{\alpha'}$, where $\mathbf{i} = \sqrt{-1}$ is the imaginary unit. Note that the subalgebra of $\mathfrak{D}$, generated by $E,F,K, K'$, is mapped into $\chi_{\PGL_2} \subset \chi_{\SL_2}$ under $\iota$.}
$$
e \longmapsto Z_{b'}, \qquad f \longmapsto Z_{d'}, \qquad k_\omega \longmapsto Z_a, \qquad k_{\omega'} \longmapsto Z_c.
$$

\end{example}

\subsection{Distinguished objects and global sections}

The charts constructed in \cref{thm:Ahat} and \cref{thm:X} are not arbitrary:  the required generator of each chart is the restriction of the distinguished object.  We may therefore consider the restriction $\Dist_{U} = \Res_U(\Dist_\decS)$ of the distinguished object to the subcategory $U = \dot{\cup}\Zc(\tri)$ covered by all open charts $\Zc(\tri)$.

As a corollary of \cref{thm:Ahat}, we have that $\iEnd_{T^k}(\Dist_U)$ is isomorphic to the subalgebra of $\zeta(\trib)$ consisting of those elements which remain regular, that is Laurent polynomials, under every quantum flip (clearly, this algebra is independent of the triangulation $\trib$). Similarly, as a corollary of \cref{thm:X}, we have that $\iEnd_{T^d}(\Dist_U)$ and $\End(\Dist_U)$ are isomorphic respectively to the subalgebras of $\chi(\tri)$ and $\eta(\tri)$, consisting of elements regular under all flips.

Now, let $\Lf(\Pc_\decS)$ and $\Lf(\Xc_\decS)$ denote the universally Laurent algebras from the cluster Poisson structure~\cite{FG06,GS19} on the $\Pc$- and $\Xc$-moduli space respectively. The algebras $\Lf(\Pc_\decS)$ and $\Lf(\Xc_\decS)$ consist of elements that are Laurent polynomials in the cluster coordinates of every chart. We denote by $\Lf_q(\Pc_\decS)$ and $\Lf_q(\Xc_\decS)$ the quantum versions~\cite{FG09a, FG09b} of these algebras. Similarly, let $\Lf(\Ac_\decS)$ be the universally Laurent algebra, also known as the upper cluster algebra, from the cluster $K_2$-structure on the moduli space of twisted $\SL_2$-local systems on $S$, considered in~\cite{FG06}. If $S$ has no punctures, we denote by $\Lf_q(\Ac_\decS)$ the quantized version~\cite{BZ05} of $\Lf(\Ac_\decS)$. Then we have the following Corollary, which follows from the fact that the tori $\zeta_{\trib}$ and flips between them coincide with the quantum cluster charts on $\Lf_q(\Ac_\decS)$ and their mutations.

\begin{cor}\label{cor:A}
If $\decS$ has no punctures and $G=SL_2$, we have the following isomorphism of algebras in $\Rep_q(T^k)$:
$$
\iEnd_{T^k}(\Dist_U) \simeq \Lf_q(\Ac_\decS).
$$
\end{cor}
We expect this description to extend to the following cases:
\begin{enumerate}
\item For $G=SL_2$, we have the following isomorphism of commutative rings:
$$
\End_{T^k}(\Dist_U)|_{q=1}/I \simeq \Lf(\Ac_\decS).
$$
\item For $G=PGL_2$, there is an isomorphisms of algebras in $\Rep_q(T^d)$
$$
\iEnd_{T^d}(\Dist_U) \simeq \Lf_q(\Pc_\decS).
$$
\item For $G=PGL_2$, there is an isomorphisms of algebras
$$
\iEnd(\Dist_U) \simeq \Lf_q(\Xc_\decS).
$$
\end{enumerate}

The only reason the above equivalences are non-obvious as stated is that we have not included certain \emph{special} charts in our definition of $U$, here we call a chart special if it contains a vertex incident to two or more tagged arcs. Hence to prove the identities, it would suffice to show that an element which is regular in every non-special (quantum) cluster chart, is in fact regular in every (quantum) cluster chart, or alternatively we could incorporate special charts into the definition of $U$.

A more interesting question is whether the above identities hold with $\Dist_\decS$ in place of $\Dist_U$.  Classically, this is the question whether $U^c$ has codimension at least two in the decorated character stack.  We intend to return to such questions in a future paper.
 
\subsection{Outlook}
The most self-evident challenge which presents itself is to extend the results of this paper to groups of higher rank, namely to discover cluster charts and flips within quantum decorated character stacks for higher rank groups.  While the general construction from this paper applies to arbitrary reductive groups, we expect the detailed determination of cluster charts beyond rank one will involve many interesting and challenging aspects, and will not be a straightforward modification of our computations here. In higher rank, decorated character stacks admit richer algebraic structure --- an arbitrary parabolic subgroup $P\subset G$ with Levi quotient $P\to L$ determining a domain wall between $G$ and $L$. These domain walls can not only stack (i.e, compose as 1-morphisms in $\BrTens$), but they may also overlap --- this is modelled by the 2-morphisms in $\BrTens$, and hence involves stratified surfaces with point defects.  We expect a correspondence between these interesting additional structures on the lattice of Levi subgroups, and the rich combinatorial structure of the cluster charts in higher rank.

Let us mention a number of more immediate applications which we intend to pursue:
\begin{itemize}
    \item Combining our work with that of \cite{Coo19} gives rise to embeddings of skein algebras into quantum cluster algebras.  These should be related to --- and extend --- the well-known constructions of \cite{Mul16} and \cite{BW}.  We will give an invariant construction of such embeddings in a forthcoming work.
    \item Alekseev--Grosse--Schomerus have obtained~\cite{AGS} a  generators-and-relations presentation of certain algebras, now commonly called AGS algebras, by quantizing the Poisson structure described in~\cite{FR99} by Fock and Rosly on framed character varieties (in the absence of $T$-regions). These algebras are now understood via monadic reconstruction for factorization homology, in \cite{BBJ18a}, and subsequently in terms of \emph{stated} \cite{CL} and \emph{internal} \cite{GJS} skein algebras. The techniques in this paper therefore lead us to quantum cluster embeddings of arbitrary AGS algebras, the simplest case of which is discussed already in~\cref{ex:punc-digon}. We will explore this systematically in a forthcoming work.
    \item While this paper focuses primarily on construction of the subcategory $\dot\cup\Zc(\tri)$ and its relation to quantum cluster algebras, an important direction of future inquiry is to understand precisely the \emph{complement}, which describes the stacky points of decorated character stacks and their quantizations.
    \item By appealing to the 4-category $\BrTens$, we may regard our constructions as giving the 2-dimensional part of a fully extended $(3+\epsilon)$-dimensional TFT. In particular, we may consider extension of our work to decorated 3-manifolds,
    for example, to hyperbolic knot complements. We expect close relations to \cite{Dim,DGG}.
    \item We have largely restricted attention to \emph{simple} decorated surfaces in the present work.  While there is no reason to expect something like cluster charts for decorated surfaces which are not simple, we nevertheless expect that many interesting examples will come from studying non-simple quantum decorated character stacks, with connections to Hecke categories, quantum Beilinson-Bernstein theorem \cite{BK-BB,Tan-BB}, dynamical quantum groups, and numerous other topics in quantum representation theory where parabolic induction plays a role.
\end{itemize}

\subsection*{Outline} Let us now outline of the contents of this paper. In Section 2, we recall basic notions from stratified factorization homology and representations of quantum groups, leading to the formal definition of $\Zc(\decS)$. In Section 3, we review monadic reconstruction, do numerous computations with it, and finally use it to construct cluster charts on $\Zc(\DD_n)$.  In Section 4, extend the construction to general $\decS$, and discuss the relation to $\Ac_{G,S}$, $\Xc_{G,S}$, $\Pc_{G,S}$ constructions. In Section 5, we close with a number of examples illustrating formulas for charts and flips.

\subsection*{Acknowledgements}
We are grateful to David Ben-Zvi for furnishing us our lodestar, and to Sam Gunningham for countless helpful discussions and clarifications around the categorical approach to algebraic geometry in which we obtained our results.

The work of D.J. and A.S. was supported by European Research Council (ERC) under the European Union's Horizon 2020 research and innovation programme (grant agreement no. 637618). D.J. was partially supported by NSF grant DMS-0932078, administered by the Mathematical Sciences Research Institute while in residence
at MSRI during the Quantum Symmetries program in Spring 2020. A.S. was partially supported by the NSF Postdoctoral Fellowship DMS-1703183, as well as by the International Laboratory of Cluster Geometry NRU HSE grant no. 2020-220-08-7077.

\section{Decorated surfaces and stratified factorization homology}\label{sec:strat-facthom}

The basic observation motivating the interaction between character stacks and factorization homology is as follows:  the assignment to each decorated surface of its decorated character stack is functorial for stratified embeddings and their isotopies, and characterized by a strong locality property called excision; meanwhile factorization homology provides universal such functorial invariants.  These observations are captured in the following definitions.

\begin{defn}
Let $\decS$ and $\decS'$ be decorated surfaces.  A \defterm{stratified embedding} $\iota \colon \decS \hookrightarrow \decS'$ is an oriented embedding of surfaces which preserving the partition into walls, regions and labels.
A \defterm{stratified isotopy} is a continuous path through the space of stratified embeddings.
\end{defn}

\begin{defn}
The symmetric monoidal $(2,1)$-category $\Surp$ is a category with
\begin{enumerate}
\item[] \textbf{Objects:} decorated surfaces,
\item[] \textbf{1-morphisms:} stratified embeddings of decorated surfaces,
\item[] \textbf{2-morphisms:} stratified isotopies of stratified embeddings.
\end{enumerate}
The symmetric monoidal structure is given by disjoint union.
\end{defn}

\begin{defn}
The full symmetric monoidal (2,1)-subcategory $\Disp\subset\Surp$ of \defterm{basic decorated disks} is generated by the following decorated surfaces shown on~\cref{fig:disks}:
\begin{itemize}
\item[] $\DD_G$ = $\R^2$ with no walls, the unique region labeled $G$;
\item[] $\DD_T$ = $\R^2$ with no walls, the unique region labeled $T$;
\item[] $\DD_B$ = $\R^2$ with the $x$-axis as its unique wall, the upper region labeled $G$, and the lower region labeled $T$.
\end{itemize}
\end{defn}

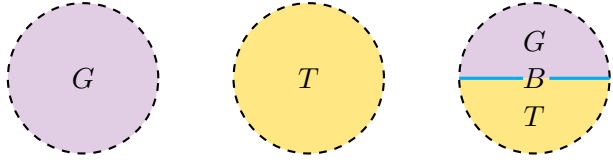
\begin{figure}[t]
\begin{tikzpicture}[every node/.style={inner sep=0, minimum size=0.2cm, thick, circle}, thick, dashed, x=0.5cm, y=0.5cm]
\draw[fill=Orchid!30] (-9,0) circle (2);
\node at (-9,0) {$G$};
\draw[fill=Goldenrod!70] (-3,0) circle (2);
\node at (-3,0) {$T$};

\fill[Orchid!30] (5,0) arc (0:180:2) to (5,0);
\fill[Goldenrod!70] (5,0) arc (0:-180:2) to (5,0);
\draw (3,0) circle (2);
\draw[very thick, cyan, solid] (1,0) to (2.7,0);
\draw[very thick, cyan, solid] (3.4,0) to (5,0);
\node at (3,1) {$G$};
\node at (3,-1) {$T$};
\node at (3,0) {$B$};


\end{tikzpicture}
\caption{Disks with parabolic structures $\DD_G$, $\DD_T$, $\DD_B$.}
\label{fig:disks}
\end{figure}

\begin{remark}
The disks above are a basis for the topology of decorated surfaces, in the sense that any point of a decorated surfaces sits in a neighborhood of the form $\DD_G, \DD_T$ or $\DD_B$.
\end{remark}

\begin{defn} A $\Disp$-algebra is a symmetric monoidal functor $\Fc \colon \Disp\to\Pr$.\end{defn}
Here, $\Pr$ is the symmetric monoidal (2,1)-category whose objects are presentable $\K$-linear categories, morphisms are colimit-preserving functors, and 2-morphisms are natural isomorphisms.

\begin{notation}
Let $\Cc = (\Cc, \otimes, \sigma)$ be a braided monoidal category with tensor multiplication and braiding denoted by
$$
\otimes \colon (V,W) \longmapsto V \otimes W
\qquad\text{and}\qquad
\sigma_{V,W} \colon V \otimes W \longmapsto W \otimes V
$$
respectively. From $\Cc$ we obtain a pair of braided monoidal categories
$$
\Cc^\op = (\Cc, \ot^\op, \sigma^\op)
\qquad\text{and}\qquad
\rev\Cc = (\Cc, \ot, \rev\sigma),
$$
where $\ot^\op$ is the opposite tensor product:
$$
\ot^\op \colon (V,W) \longmapsto W \otimes V,
$$
while
$$
\sigma^\op_{V,W}  = \sigma_{V,W}^{-1}
\qquad\text{and}\qquad
\rev\sigma_{V,W} = \sigma_{W,V}^{-1}.
$$
In fact these are equivalent as braided tensor categories, via the identity functor with tensor structure given by the braiding.
\end{notation}

\begin{defn} Given a pair $\Ac$, $\Bc$ of braided tensor categories, an $(\Ac,\Bc)$-central structure on a tensor category $\cC$ is a braided tensor functor $\Ac\bt\rev{\Bc}\to Z_{Dr}(\Cc)$.
\end{defn}

Note that by definition of the Drinfeld center, the data of a lift a monoidal functor $F\colon\Ac\bt\Bc\to \Cc$ to an $(\Ac,\Bc)$-central structure is equivalent to the data, for every object of $\Ac\bt\Bc$ of a \defterm{half-braiding}, i.e., a natural isomorphism $F(a\bt b)\ot c \cong c\bt F(a\bt b)$, which must be natural in all arguments, and moreover compatible with the braiding on $\Ac\bt\rev{\Bc}$ in the obvious way.

\begin{prop}
\label{prop-Pr}
A $\Disp$-algebra is uniquely determined by the following data:
\begin{enumerate}
\item The categories $\Fc(\DD_G), \Fc(\DD_T), \Fc(\DD_B) \in \Pr$,
\item The structures of a ribbon braided tensor category on $\Fc(\DD_G)$ and $\Fc(\DD_T)$,
\item The structure of a $(\Fc(\DD_G), \Fc(\DD_T))$-central tensor category on $\Fc(\DD_B)$.
\end{enumerate}
\end{prop}

We refer to \cref{sec:quagrp} for a definition of the notion of a central tensor category, and for examples.

\begin{defn}[Adapted\footnote{Strictly speaking, the labeling of regions of a decorated surface is additional data on the stratification (while it is also less general because we only allow simple smooth curves as defects).  However, the labelling data does not materially change any argument from \cite{AFT}, so we will simply adapt this definition and the formulation of the excision theorem to our setting, without further comment.} from \cite{AFT}]
The stratified factorization homology with coefficients in $\Fc$ is the functor,
$$
\Zc \colon \Surp \to \Pr
$$
defined as the unique left Kan extension of the functor $\Fc$ from the diagram
\[
\xymatrixcolsep{3pc}
\xymatrixrowsep{3pc}
\xymatrix{
\Disp \ar[dr] \ar[rr]^\Fc && \Pr\\
&\Surp \ar[ur]_\Zc
}
\]
\end{defn}

In other words, $\Zc(\decS)$ is defined as the following homotopy colimit:
$$\Zc(\decS) = \underset{\decX\to \decS}{\colim}\,\,  \Zc(\decX),$$
taken in the 2-category of locally presentable categories, indexed over all embeddings and their isotopies, of disjoint unions $\decX$ of finitely many copies of $\DD_G$, $\DD_B$, and $\DD_T$, into $\Sigma$.  Unpacking the definition of colimit, this means:
\begin{enumerate}
\item For each disjoint union of decorated disks,
\[\decX = \DD_G ^{\sqcup I} \sqcup \DD_B^{\sqcup J} \sqcup \DD_T^{\sqcup K},\quad
\textrm{ we have }\quad\Zc(\decX) = \Fc(\DD_G)^{I} \bt \Fc(\DD_B)^{J} \bt \Fc(\DD_T)^{K}. 
\]
\item Each stratified embedding $i_\decX\colon\decX\to \decS$, induces a colimit-preserving functor, \[i_{\decX*}\colon \Fc(\DD_G)^{I} \bt \Fc(\DD_B)^{J} \bt \Fc(\DD_T)^{K}\to \Zc(S),\]
\item Each stratified isotopy induces an isomorphism of induced functors, and 
\item Each factorization $i_\decX\colon\decX\to \decX'\to \decS$ of a disk embedding through an intermediate disk embedding (i.e. where $\decX,\decX'\in\Disp$), induces a factorization of the functor $i_{\decX*}$ through the fixed braided tensor, tensor, and central structures on $\Fc(\DD_G)$, $\Fc(\DD_B)$, and $\Fc(\DD_T)$.  Likewise, and isotopies thereof are given by the fixed braiding, associator, and action morphisms. 
\item Finally, the category $\Zc(\decS)$ is universal for such structures:  given any other category $\cC$ equipped with all of the above data, we obtain a canonical functor $\Zc(\decS)\to \cC$, such that the data on $\cC$ is obtained from that on $\Zc(\decS)$ by composition.
\end{enumerate}

\subsection{Cylinder algebras, module actions and excision}
The definition of factorization homology as a left Kan extension emphasizes universal properties and topological invariance.  These are useful in applications, but are not convenient for doing explicit computations.  Let us now recall the excision formula from \cite{AFT}, which gives a concrete algebraic algorithm for doing computations.

\begin{defn} A decorated 1-manifold $\decX$ is an oriented 1-manifold with a finite collection $\cP$ of points.  The \defterm{walls} and \defterm{regions} of $\decX$ are $\cP$ and the connected components of $\decX\smallsetminus\cP$, respectively.  We fix the further data of a labeling of the regions from the alphabet $\{G,T\}$, such that two regions sharing a wall have distinct labelings.
\end{defn} 
\begin{defn} A decorated 1-disk is a disjoint union of the four decorated intervals in \cref{fig:1man}.\end{defn}
\begin{figure}[h]
\begin{tikzpicture}[every node/.style={inner sep=0, minimum size=0.2cm, thick, circle, fill=white}, thick, x=0.5cm, y=0.5cm]
\draw (-11,0) to (-7,0);
\node at (-9,0) {$G$};
\draw (-5,0) to (-1,0);
\node at (-3,0) {$T$};
\draw (1,0) to (5,0);
\node[minimum size=0.15cm, fill=black] at (3,0) {};
\node at (2,0) {$G$};
\node at (4,0) {$T$};
\node at (3,-0.75) {$B$};
\draw (7,0) to (11,0);
\node[minimum size=0.15cm, fill=black] at (9,0) {};
\node at (8,0) {$T$};
\node at (10,0) {$G$};
\node at (9,-0.75) {$B$};
\end{tikzpicture}
\caption{Flagged 1-manifolds are disjoint unions of the above.}
\label{fig:1man}
\end{figure}
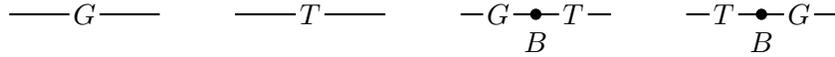

\begin{defn} An $\decX$-gate in a decorated surface is a neighborhood $U$ of the boundary with a homeomorphism $U\cong\decX\times [0,1)$ with a cylindrical decorated half-surface.  We abbreviate the cases $U=\DD_G, \DD_T, \DD_B$, by $G$-gate, $T$-gate, $B$-gate respectively.
\end{defn}

Concatenation in the $(0,1)$ direction makes any cylindrical decorated surface $\decX\times (0,1)$ into an algebra object in $\Surp$, and hence equips the category $\Zc(\decX\times (0,1))$ with the structure of a monoidal category.  Similarly, the specification of an $\decX$-gate on a decorated surface $\decS$ induces on $\Zc(\decS)$ the structure of a $\decX\times (0,1)$-module category.

A decomposition of the decorated surface $\decS$, denoted $\decS = \decS_1 \underset{\decX\times (0,1)}{\sqcup} \decS_2,$ consists of the choice of a cylindrical decorated surface embedding $\decX\times (0,1)\into\decS$, such that $\decS\smallsetminus(\decX\times (0,1))$ is a disjoint union of decorated surfaces $\decS_1$ and $\decS_2$.  In this case, $\decS_1$ and $\decS_2$ are naturally marked with $\decX$-gates.

\begin{theorem}[Excision, \cite{AFT}] Each decomposition of a decorated surface,
\[\decS = \decS_1 \underset{\decX\times (0,1)}{\sqcup} \decS_2,\]
induces a canonical equivalence of categories,
\[
\Zc(S) \simeq \Zc(\decS_1) \underset{\Zc(\decX \times (0,1))}{\boxtimes} \Zc(\decS_2).
\]
\end{theorem}

\subsection{Recollection about quantum groups}\label{sec:quagrp}
In this section we recall the braided tensor categories $\Repq(G)$ and $\Repq(T)$, and the tensor category $\Repq(B)$, consisting of `integrable' representations of the quantized universal enveloping algebras $U_q(\g)$, $U_q(\t)$ and $U_q(\b)$, and use them to define quantizations of framed configuration spaces.  We will not recall in detail explicit presentations of each enveloping algebra in this paper, referring to standard texts.  However, we recall more details the cases $\SL_2$ and $\PGL_2$, in order to fix notation for computations to come.

Given a lattice $\Lambda$, we let $\Vect_\Lambda$ denote the category of $\Lambda$-graded vector spaces, whose objects are vector spaces $V$ equipped with a decomposition $V=\oplus_{\lambda\in\Lambda}V_\lambda$, and whose morphisms are those linear maps respecting the $\Lambda$-grading.
The simple objects of this category are of the form $\C_\lambda$ (a one-dimensional vector space supported at $\lambda$), and every object is a direct sum of simples.  Clearly, we have isomorphism $\C_\lambda\ot\C_\mu\to\C_{\lambda+\mu}$.

Given, further, a complex number $q\in\C$ and an integer-valued symmetric bilinear form,
\[
\alpha\colon\Lambda\times\Lambda \to \Z,
\]
we obtain a braided tensor category structure on $\Vect_\Lambda$ by declaring the braiding on $\C_\lambda\otimes \C_\mu$ to be scalar multiplication $q^{\alpha(\lambda,\mu)}$ followed by the flip of tensor factors, and extending to arbitrary direct sums of simples.

Fix a reductive algebraic group $G$ along with a Borel subgroup $B\subset G$, and denote by $T$ the universal Cartan quotient.

\begin{defn}
\label{def:repq-t}The category $\Repq(T)$ denotes the braided tensor category $(\Vect_\Lambda,-\alpha)$, where $\Lambda$ is the weight lattice and $\alpha$ is the Cartan pairing.
\end{defn}

\begin{notation} Let $U_q(\g)$ denote the quantized universal enveloping algebra of $\g=Lie(G)$, defined over the field $\C(q)$.  Denote the standard Serre generators by $E_1,\ldots, E_r$, $F_1,\ldots F_r$, and $K_\lambda,$ for $\lambda \in \Lambda$.  Let $U_q(\b)$ (resp., $U_q(\b_+)$) denote subalgebras generated $K_\lambda$'s and by $F_i$'s (resp. $E_i$'s).  Let $U_q(\t)$ denote the subalgebra generated by the $K_\lambda$'s.
\end{notation}

Let $\Lambda_{ad}\subset \Lambda$ denote the root lattice.  In order to determine $G$ from $\g$, we need to further specify a lattice $\Lambda_G$ between them, $\Lambda_{ad} \subset \Lambda_G\subset \Lambda$.  We then define $\Repq(T)$ as the full subcategory of the category of $U_q(\t)$-modules: its representations are spanned by weight vectors $v_\mu$, $\mu \in \Lambda$, on which $K_\lambda$ acts by $q^{\alpha(\lambda,\mu)}$.  Recall that a representation of an algebra $A$ is called \defterm{locally finite} if the orbit $A\cdot v$ of every vector $v\in V$ is finite-dimensional.

\begin{defn}
The category $\Repq(B)$ is the full subcategory of locally finite $U_q(\b)$-modules whose restriction to $U_q(\t)$ lies in $\Rep_q(T)$.
\end{defn}

Recall that for generic $q$, the category $\Repq(G)$ is in fact semisimple: every object may be presented as a (possibly infinite) direct sum of simple finite-dimensional $U_q(\g)$-modules.  It follows from the classification of finite-dimensional $U_q(\g)$-modules that the restriction of an object of $\Repq(G)$ to $U_q(\b)$ lies in $\Repq(B)$.  

The universal $\Rc$-matrix is an invertible element $\Rc$ lying in a completion of $U_q(\g)\otimes U_q(\g)$,  satisfying the quantum Yang-Baxter equation. In Sweedler notation, we write $\Rc = \Rc_{(1)} \otimes  \Rc_{(2)}$. Then the braiding is given by:
\[ \sigma_G(v\otimes w) = \hr{\Rc_{(2)} w}\otimes  \hr{\Rc_{(1)} v} \]

\begin{defn}
The category $\Repq(G)$ is the category of locally finite $U_q(\g)$-modules, all of whose weights lie in $\Lambda_G$, with its braided tensor structure given via the universal $\Rc$-matrix.
\end{defn}

Recall that in Dynkin type $A_1$m we have $\Lambda=\Z$, and $\Lambda_{ad}=2\Z$, hence the only two braided tensor categories we consider are $\Repq(\SL_2)$ (with $\Lambda_{\SL_2}=\Z$), and $\Repq(\PGL_2)$ (with $\Lambda_{\PGL_2}=2\Z$).  Let us recall these in detail now.

The quantum group $U_q(\sl_2)$ is a quasi-triangular Hopf algebra over $\C(q)$ with generators $\hc{E, F, K^{\pm1}}$ subject to relations
$$
KE =q^2 EK, \qquad KF = q^{-2}FK, \qquad \hs{E,F} = \frac{K-K^{-1}}{q-q^{-1}}.
$$
It carries the comultiplication
\beq
\label{coproduct-def}
\Delta(E) = E\otimes 1+K\otimes E, \qquad \Delta(F) = F\otimes K^{-1}+1\otimes F,  \qquad \Delta(K) = K \otimes K,
\eeq
the antipode
$$
S(E)=-K^{-1}E, \qquad S(F)=-FK, \qquad S(K) = K^{-1},
$$
and the counit
$$
\epsilon(K)=1, \qquad \epsilon(E) = \epsilon(F)=0.
$$
The universal $\Rc$-matrix of the quantum group can be written as
\begin{align}
\label{Uqg-rmat}
\Rc^G = \Psi_q\hr{-(q-q^{-1})^2 E \otimes F} q^{-\frac{1}{2} H \otimes H},
\end{align}
where $\Psi_q(x)$ is the quantum dilogarithm
$$
\Psi_q(x) = \prod_{a=1}^\infty \frac{1}{1+q^{2a-1}x}.
$$
The latter is a close relative with the q-exponential, namely
$$
\Psi_q(x) = \exp_{q^2}\hr{\frac{x}{q-q^{-1}}}
$$
where
$$
\exp_q(x) = \sum_{n=0}^{\infty} \frac{x^n}{[n]_q!} \qquad\text{with}\qquad [n]_q! = \prod_{a=1}^n \frac{q^a-1}{q-1}.
$$
The category $\Repq(SL_2)$ and its subcategory $\Rep_q(PGL_2)$ are ribbon categories, with  the ribbon element $\nu_G$ acting in the $n$--dimensional irreducible representation $V_n$ of $U_q(\mathfrak{sl}_2)$ by the scalar $q^{\frac{1}{2}n^2+1}$.

The quantum group $U_q(\sl_2)$ has Borel subalgebras $U_q(\b_\pm)$ generated by $\hc{E,K^{\pm1}}$ and $\hc{F,K^{\pm1}}$ respectively, their nilpotent subalgebras $U_q(\n_\pm)$ generated respectively by the elements $E$ and $F$, and the torus $U_q(\t) = U_q(\b_+) \cap U_q(\b_-)$ generated by the elements $K^{\pm1}$. Note that $\Rc$-matrix that determines the braiding on $\Repq(T)$ from \cref{def:repq-t} may be expressed as
$$
\Rc^T = q^{-\frac{1}{2}H\otimes H}.
$$
The corresponding ribbon element $\nu_T$ acts on the 1-dimensional simple $\mathbb{C}_n, n\in\mathbb{Z}$ by the scalar $q^{\frac{1}{2}n^2}$.

Let us describe the different categories which arise in these cases completely explicitly, for generic $q$.  The category $\Repq(\SL_2)$ consists of arbitrary direct sums of simple representations $V(k)$ with highest weight $k$, $\Rep_q(B)$ consists of direct sums of indecomposable modules $X(k,l)$, with $k\equiv l\,\textrm{mod}\, 2$, where $X(k,l)$ has a basis $\{v_k, v_{k+2},\ldots v_l\}$, each $v_i$ a weight vector of weight $i$, with $Fv_i=v_{i-2}$.  In this case $\Rep_q(T)$ is the category of $\Z$-graded vector spaces with the pairing $\langle m, n \rangle=-\frac{mn}{2}$.  All three corresponding categories are the same for $\PGL_2$ in place of $\SL_2$, except we add a requirement that all $T$-weights (in particular $k$ and $l$ above) are even, i.e. in $\Lambda_{ad}$.

\subsection{The $\Disp$-algebra $\Repq(\TBG)$ and $\Zc(\decS)$}

The inclusion $i \colon U_q(\b)\to U_q(\g)$ and the projection $\pi \colon U_q(\b)\to U_q(\t)$ (obtained by quotienting by $F_i=0)$ induce pullback functors,
\[i^* \colon \Repq(G)\to\Repq(B), \qquad \pi^* \colon \Repq(T)\to \Repq(B).\]

These functors admit the important additional structure of a central lift (see \cref{sec:quagrp}), making $\Repq(B)$ into a \defterm{$(\Repq(G),\Repq(T))$-central tensor category}. Indeed, since the $\Rc$-matrix~\eqref{Uqg-rmat} is a sum of elements of $U_q(\b_+)\otimes U_q(\b)$, it has a well-defined action on any tensor product $i^*V\otimes W$ for objects $V,W$ of $\Repq(G)$ and $\Repq(B)$ respectively, and composing this action with the flip of tensor factors yields a morphism in $\Repq(B)$. Moreover, given weight vectors $v_\lambda\in \pi^*\chi$ and $w_\mu\in W$, the form of the coproduct~\cref{coproduct-def} along with the fact that $\pi(F_i)=0$ imply that
\begin{align*}
    F_i\cdot\hr{v_\lambda\otimes w_\mu} = v_{\lambda}\otimes F_iw_\mu,
\end{align*}
where $F_iw_\mu$ has weight $\mu-\alpha_i$. Hence
\begin{align*}
    \overline{\sigma}_T\circ F_i\hr{v_\lambda\otimes w_\mu} &= q^{(\lambda,\mu-\alpha_i)} F_iw_\mu\otimes v_{\lambda}\\
    &=q^{(\lambda,\mu)}F_iw_\mu\otimes K_i^{-1}v_{\lambda}\\
    &=F_i\cdot \overline{\sigma}_T(v_\lambda\otimes w_\mu),
\end{align*}
and we see that $\overline{\sigma}_T$ also defines a morphism in $\Repq(B)$. 

\begin{defn} We define a central structure,
\[ i^*\ot\pi^* \colon \Rep_q(G)\boxtimes \bop{\Rep_q(T)}\longrightarrow Z(\Rep_q(B))\]
\[V\boxtimes \chi \longmapsto i^*(V)\ot \pi^*(\chi)\]
via the half-braiding,
\[
\sigma_{(V,\chi),W} \colon (i^*(V) \otimes \pi^*(\chi)) \otimes W \xrightarrow{\sigma_{(123)} \circ \Rc^G_{13} \circ (\Rc^T_{32})^{-1}} W \otimes (i^*(V) \otimes \pi^*(\chi)).
\]
\end{defn}

\begin{defn}  The \defterm{parabolic induction $\Disp$-algebra $\Repq(\TBG)$} is determined via \cref{prop-Pr} by the assignments:
\[\DD_G\mapsto \Repq(G),\quad\DD_T\mapsto \Repq(T), \quad \DD_B\mapsto \Repq(B),\]
with their braided tensor and central structures defined above.
\end{defn}

We now come to the main definition of the paper, which as been stated already in the introduction.

\begin{defn}
The \defterm{quantum decorated character stack} $\Zc(\decS)$ is the stratified factorization homology,
\[ \Zc(\decS) := \int_{\decS} \Repq(\TBG),\]
of $\decS$ with coefficients in the parabolic induction algebra $\Repq(\TBG)$.
\end{defn}

\section{Monadic reconstruction of module categories}\label{sec:monadic}
Monadic reconstruction for module categories is a center-piece in the higher representation theory of tensor categories.  It was developed the furthest in the fusion, and in the finite non-semisimple, settings; a comprehensive reference is \cite{EGNO}.  Monadic reconstruction was developed also in the locally presentable framework in \cite{BBJ18a,BBJ18b,BJS}.  We recall those constructions here, with an additional wrinkle: in contrast to the usual situation, generators we encounter will not be projective, even in the relative sense; this reflects that the framed moduli spaces at the classical level are only quasi-affine rather than affine.  For this reason, in place of the equivalences of categories featuring in usual monadic reconstruction we will instead obtain open embeddings of categories, and we will compute their orthogonal complements.

\subsection{Adjoints and monads}
To begin, we recall some standard facts and definitions in enriched category theory.

\begin{defn}[Sketch] A $\K$-linear category $\cC$ is \defterm{locally presentable} if the colimit $\colim(D)$ of any diagram $D$ in $\cC$ exists as an object of $\cC$, and if $\cC$ is generated under $\kappa$-filtered colimits by a small category of $\kappa$-compact objects, for some cardinal $\kappa$.\end{defn}

\begin{remark}
It will suffice to take $\kappa=\aleph_0$, the countable cardinal; such categories are called \defterm{locally finitely presentable}.
\end{remark}

\begin{defn} Let $\Pr$ denote the (2,1)-category of locally presentable $\mathbf k$-linear categories with colimit-preserving functors and natural isomorphisms.\end{defn}

\begin{remark}
Examples of locally finitely presentable categories include the category of $A$-modules for an arbitrary associative $\K$-algebra, the category of representations of an algebraic group, and the category of quasi-coherent sheaves on algebraic variety, or more generally on an Artin stack.  All examples we will consider will turn out to one of these types of categories, or its $q$-deformation.
\end{remark}

\begin{theorem}[Special adjoint functor theorem] Any colimit-preserving functor $L\colon\cC\to \cD$ between locally finitely presentable categories has a right adjoint $R$, however $R$ is not itself necessarily colimit-preserving.
\end{theorem}

We will refer below to the \defterm{adjoint pair} $(L,R)$, and to the \defterm{unit} $\eta\colon\id_\cC\to LR$ and \defterm{counit} $\epsilon\colon RL\to\id_\cD$ of the adjunction.

\begin{defn} The functor $G\colon \cD\to\cC$ is \defterm{conservative} if it reflects isomorphisms, i.e. for any morphism $f$ in $\cD$, $G(f)$ is an isomorphism if, and only if, $f$ is an isomorphism.\end{defn}

\begin{prop}  Suppose that $\cC$ and $\cD$ are abelian locally presentable categories and $G\colon\cD\to\cC$ is colimit-preserving.  Then $G$ is conservative if, and only if, $G(X)=0$ implies $X=0$ for any object $X$.
\end{prop}

\begin{defn} The functor $L\colon\cC\to\cD$ is \defterm{dominant} if $\cD$ is generated under colimits by the image of $L$.\end{defn}

\begin{prop}
The functor $L\colon\cC\to\cD$ is dominant if, and only if, its right adjoint $R\colon\cD\to\cC$ is conservative.
\end{prop}

\begin{defn}  Let $L\colon\cC\to\cD$ be a colimit preserving functor, and let $R$ denote its right adjoint.  The \defterm{monad for the adjunction} is the composite functor $A=RL\colon \cC\to \cC$, equipped with the structure of a unital algebra in $\End(\cC)$, with unit given by the unit $\eta\colon\Id_\cC\to RL$ defining the adjunction, and with multiplication,
\[AA = RLRL\xrightarrow{\id_{R}\circ\epsilon\circ\id_{L}} RL=A,\]
obtained from the counit of the adjunction.
\end{defn}

\begin{defn} An \defterm{$A$-module} in $\cC$ is a pair $(c,\psi)$ of an object $c\in \cC$ and a morphism $\psi\colon Ac\to c$, satisfying an evident associativity axiom which we omit.   The \defterm{Eilenberg-Moore} category $A\mod_\cC$ is the category of $A$-modules in $\cC$, with the evident notion of morphisms.
\end{defn}

\begin{defn} The \defterm{comparison functor} $\widetilde{R}\colon\cD\to A\mod_\cC$ sends $d\in \cD$ to the object $R(d)$, equipped with an $A$-module structure via
\[AR(d) = RLR(d)\xrightarrow{\id_{R}\circ\epsilon} d.\]
\end{defn}

\begin{theorem}[Barr-Beck reconstruction] 
\label{BB-original}
If the functor $R$ is conservative and colimit preserving, then $\widetilde{R}$ defines an equivalence of categories.
\[ \cD \simeq A\mod_\cC.\] 
\end{theorem}

We will require a variant of Barr-Beck reconstruction in the abelian category setting, involving a weaker condition and a weaker conclusion.  We were unable to find a reference for the claim, although it uses standard constructions in the theory of abelian categories.

\begin{defn} We say that a fully faithful embedding $i\colon\cD\to\cE$ is \defterm{reflective} (resp. \defterm{open}), and that $\cD$ is a \defterm{reflective subcategory} (resp. \defterm{open subcategory}), if $i$ admits a left adjoint (resp. exact left adjoint).
\end{defn}

\begin{lemma}[\cite{Barr-Wells} p. 111, Corollary 7 and Theorem 9]\label{lem-faithfulcounit} The comparison functor $\widetilde{R}$ is faithful if, and only if, the counit of $(L,R)$ is an epimorphism, and is fully faithful if, and only if, the counit of the adjunction is a regular epimorphism.
\end{lemma}

\begin{theorem} 
\label{weak-reconstruction}
Suppose that $\cC$ and $\cD$ are abelian categories, and that $R$ is conservative (but not necessarily colimit-preserving).  Then $\widetilde{R}$ defines a reflective embedding of $\cD$ into $A\mod_\cC$.
\end{theorem}
\begin{proof}
By the previous discussion, we need to show that the counit of $(L,R)$ is a regular epimorphism and that $\widetilde{R}$ admits a left adjoint. Let us fix the forgetful functor $\For\colon A\mod_\cC\to\cC$.  The functor $R$ factors as $\For\circ\widetilde{R}$.  Since $R$ is limit-preserving and conservative, it is faithful, hence so is $\widetilde{R}$.  By \cref{lem-faithfulcounit}, the counit for the adjunction is an epimorphism.  In an abelian category every epimorphism is regular, hence again by \cref{lem-faithfulcounit}, $\widetilde{R}$ is in fact fully faithful.  We also have that $\widetilde{R}$ is limit preserving--because $R$ is limit-preserving and limits in $A\mod_\cC$ are computed in $\cC$--hence $\widetilde{R}$ admits a left adjoint $\widetilde{L}$.
\end{proof}

In summary, we have the following commutative diagram, in which each inner arrow is the right adjoint to each outer arrow, and in which the inner triangle and outer triangle each are commuting triangles.
\[
\begin{tikzpicture}[every node/.style={inner sep=0, minimum size=0.5cm, thick, draw=none}, x=0.75cm, y=0.75cm]

\node (1) at (0,0) {$\cD$};
\node (2) at (6,0) {$\cC$};
\node (3) at (3,3) {$A\mod_\cC$};

\draw [->] (1) to [bend left = 10] node[midway,fill=white] {$R$} (2);
\draw [->] (2) to [bend left = 10] node[midway,fill=white] {$L$} (1);
\draw [->] (1) to [bend right = 15] node[midway,fill=white] {$\widetilde{R}$} (3);
\draw [->] (3) to [bend right = 15] node[midway,fill=white] {$\widetilde{L}$} (1);
\draw [->] (3) to [bend right = 15] node[midway,fill=white] {$\For$}(2);
\draw [->] (2) to [bend right = 15] node[midway,fill=white] {$\operatorname{Free}$} (3);

\end{tikzpicture}
\]

\begin{remark} A formula for the left adjoint $\widetilde{L}$ may be given as follows.  Any $M\in A\mod_\cC$ may be presented as $\coker(\Free(V)\overset{f}{\longra}\Free(W))$, for objects $V$, $W$ of $\cC$.  Since $L$ is colimit preserving, we have:
\begin{equation}\label{eqn:LofM}
\widetilde{L}(M) = \coker \left((\widetilde{L}\circ\Free)(V)\overset{\widetilde{L}(f)}{\longra} (\widetilde{L}\circ\Free)(W)\right) = \coker\left(L(V)\overset{\widetilde{f}}{\longra} L(W)\right),
\end{equation}
where $\widetilde{f}$ is the determined by $f$ by the following natural maps,
\begin{align}\label{eqn:tildef}
\Hom_{A-mod_\cC}(RL(V),RL(W)) \overset{\sim}{\longra} &\Hom_{\cC}(V,RL(W))\overset{L}{\longra}\\
&\overset{L}{\longra}\Hom_{\cD}(L(V),LRL(W))\overset{\epsilon}{\longra} \Hom_{\cD}(L(V),L(W)) \nonumber
\end{align}
\end{remark}

\begin{defn}
Let $\cC$ be an abelian category, and let $\cD$ be a full subcategory.  The left orthgonal $^\perp\cD$ is the full subcategory of $\cC$ consisting of objects $c$ such that $\Hom(c,d)=0$, for all $d\in\cD$.  Similarly, the right orthogonal $\cD^\perp$ is the full subcategory of $\cC$ consisting of objects $c$ such that $\Hom(d,c)=0$, for all $d\in\cD$.
\end{defn}

\begin{prop}\label{prop:kernel}
Suppose that $R\colon\cD\to\cE$ defines a reflective embedding, with left adjoint $L$.  Then we have an identification,
\[R(\cD) = (\ker(L))^\perp,\]
between the essential image of $R$ and the right orthogonal to the kernel of $L$.
\end{prop}
\begin{proof}
Let $X\in \cE$.  By the Yoneda embedding, $L(X)=0$ if, and only if, $\Hom(L(X),Y)=0$ for all $Y\in\cD$, which is if, and only if, $\Hom(X,R(Y))=0$ for all $Y\in \cD$.  The first condition is that $X\in\ker(L)$, while the second condition is that $X\in \,^\perp R(\cD)$.  Hence $\ker L=\,^\perp R(\cD)$, hence $(\ker L)^\perp=(^\perp R(\cD))^\perp=R(\cD)$.
\end{proof}

\subsection{Internal homomorphisms}  The preceding discussion of monadicity is especially useful in analyzing rigid tensor categories and their module categories.  We recall these applications here.

Let $\cA$ be a rigid tensor category, and $\cM$ its module category, and fix an object $m\in M$.  Then the functor $act_m \colon \cA\to \cM$ sending $a \mapsto a\otimes m$ is colimit preserving.  

\begin{defn} The \defterm{internal homomorphisms from $m$ to $n$} is the object,
$$
\iHom_\cA(m,n)=act_m^R(n)\in\cA.
$$
In what follows, we will often write $\iHom(m,n)$ instead of $\iHom_\cA(m,n)$, if there is no ambiguity in regards to the tensor category.
\end{defn}

\begin{defn} The \defterm{evaluation morphism}
$$
\mathrm{ev}_{m,n}\in \Hom(\iHom(m,n) \otimes m,n)
$$
is the image of the identity morphism, $\id_{\iHom(m,n)}$, under the canonical isomorphism
$$
\Hom(\iHom(m,n)\otimes m,n) \simeq \End(\iHom(m,n)).
$$

\end{defn}

\begin{defn} The \defterm{composition law}
$$
\iHom(n,o) \otimes \iHom(m,n) \longrightarrow \iHom(m,o)
$$
is the image of the morphism
$$
\mathrm{ev}_{n,o}\circ(\id_{\iHom(n,o)}\otimes ~\mathrm{ev}_{m,n}) \in \Hom(\iHom(n,o)\otimes\iHom(m,n)\otimes m,o)
$$
under the canonical isomorphism
$$
\Hom(\iHom(n,o)\otimes\iHom(m,n)\otimes m,o) \simeq \Hom(\iHom(n,o)\otimes\iHom(m,n),\iHom(m,o))
$$
\end{defn}

\begin{defn}
\label{defn-iEnd}
The \defterm{internal endomorphism algebra of $m$} is
$$
\iEnd(m)=\iHom(m,m)\in\cA,
$$
with the algebra structure given by the composition law.
\end{defn}

\begin{remark} Let $A$ denote the monad associated to the adjoint pair $(\act_m, \iHom(m,-))$ in the preceding section.  Then we have a natural isomorphism $\iEnd(m)\cong A(\one_\cA)$, and the algebra structure on $\iEnd(m)$ is the one induced from the monadic structure on $A$.  Hence the calculus of internal Hom's is a special case of monadic reconstruction, so we are able to produce genuine algebra objects in $\cA$, rather than in $\End(\cA)$.  This is due to the fact that $\iHom(m,-)$ and $\act_m$, and hence their composition are $\cA$-module functors, and we have the usual equivalence $\End_\cA(\cA)=\cA^{op}$.
\end{remark}

\begin{defn} We say that $m$ is an $\cA$-generator of $M$ if $\iHom(m,-)$ is conservative, and we say that $m$ is $\cA$-compact-projective if $\iHom(m,-)$ is colimit preserving. \end{defn}
We have the following translation of \cref{BB-original} and \cref{weak-reconstruction} to the case of module categories.

\begin{theorem}[Barr-Beck reconstruction for module categories \cite{Ostrik},\cite{BBJ18a}]  Fix $\cA$ a rigid tensor category, $\cM$ its abelian module category, $m$ an object of $\cM$.
\begin{enumerate}
\item  Suppose that $m$ is an $\cA$-generator of $\cM$.  Then we have a reflective embedding,
\[\cM\hookrightarrow \iEnd(m)\mod_\cA,\]
realizing $\cM$ as the reflective subcategory of $\iEnd(m)\mod_\cA$.
\item Suppose, moreover, that $m$ is an $\cA$-compact-projective generator.  Then we have an equivalence,
\[\cM\simeq \iEnd(m)-mod_\cA.\]
\end{enumerate}
\end{theorem}

\begin{remark}
In all examples we will consider in the paper, the required left adjoint is clearly exact (e.g. if $\cA=\Repq(G)$ or $\Repq(T)$, so that we obtain an open embedding.
\end{remark}
\begin{prop}\label{insertgate} Suppose $\cC$ and $\cD$ are rigid tensor categories, $F\colon\cC\to\cD$ is a tensor functor, and suppose that $d \in \cD$ is an $\cC$-projective generator, so that $\cD\simeq \iEnd(d)\mod_{\cC}$.  Suppose further that $\cM$ is a $\cD$-module category, with a $\cD$-projective generator, $m$.  Then we have
\[\cD\simeq \iEnd(d)\mod_\cC ,\qquad \cM\simeq \iEnd(m)\mod_\cD,\]
and we obtain a homomorphism of algebras in $\cC$:
\[\iEnd_\cC(d)\to \iEnd_\cC(m),\]
such that the action of $\cD$ on $\cM$ is given by relative tensor products over $\iEnd_\cC(d)$.
\end{prop}

\subsection{Monadic reconstruction: classical examples}\label{sec:barrbeckwarmup}
Let us recall some very classical and well-known facts in geometric representation theory of flag varieties, recast in the categorical framework of noncommutative algebraic geometry and treated by monadic reconstruction.

\begin{example}\label{ex:aff} Let us start with a simple, almost trivial, application of the Barr-Beck theorem. Let $R \rightarrow S$ be commutative algebras. Let $X = \Spec(R)$, $Y=\Spec(S)$ so that we have the corresponding map of varieties $\pi\colon Y \rightarrow X$. We have $\QC(X) = R\mod$ and $\QC(Y) = S\mod$.

There is a natural functor $L := \pi^*\colon  \QC(X) \rightarrow \QC(Y)$ given by $M \rightarrow M \otimes_R S$. The right adjoint is $R := \pi_*\colon  \QC(Y) \rightarrow \QC(X)$. Note that $\QC(Y)$ is a module category for $\QC(X)$. Barr-Beck reconstruction realizes $\QC(Y)$ as the category of objects of $\QC(X)$ equipped with an action of the monad $RL$.

The monad for the adjunction is $- \otimes_R S$. An algebra for this monad is exactly an $R$-module which has a compatible $S$-module structure. Clearly such an object in $\QC(X)$ corresponds to an object in $\QC(Y)$.
\end{example}

The examples we consider throughout this paper build on these examples in three distinct directions.
\begin{enumerate}
\item $X$ and $Y$ will be replaced by stacks.  Hence we will work internally to a monoidal category which carries the symmetries of the stack.
\item The map $Y \rightarrow X$ will not be affine, but quasi-affine.  Hence applying Barr-Beck will only give open embeddings of categories rather than equivalences.
\item We will be interested in quantizations; hence the symmetric monoidal categories expressing classical symmetries will be replaced systematically by their braided monoidal $q$-deformations.
\end{enumerate}
In the remainder of this section, we will illustrate the first two points with classical examples. In the next section, we treat the quantum case.

\begin{example} Let $X=\bullet/(G \times G)$ and let $Y=\bullet/G$ be the stack quotients of the one-point variety, by the trivial group action -- these are sometimes called \defterm{classifying stacks}, and denoted $B(G\times G)$, $BG$.  By definition, we have
\[\QC(X)=\Rep(G) \boxtimes \Rep(G), \qquad \QC(Y)=\Rep(G).\]
The group homomorphism $g \rightarrow (g, g^{-1})$ induces a natural map $\pi\colon Y \rightarrow X$. Pullback along $\pi^*$ defines a $\Rep(G) \boxtimes \Rep(G)$-module structure on $\Rep(G)$, and we can apply Barr-Beck techniques to describe it. Let $\one_G$ be the trivial $G$-module. The functor $L=\pi^*$ is given by $\act_{\one_G}$, or
$$(V \boxtimes W)  \longmapsto V \otimes W.$$
The right adjoint can be computed, $$R=\pi_* V  \longmapsto \bigoplus_{\la \in P_+} (V \otimes V_\la^*) \boxtimes V_\la.$$
One then computes an isomorphism of algebras,
$$
\iEnd_{G\times G}(\one_G) = \bigoplus_{\la \in P_+} V_\la^* \boxtimes V_\la \cong \Oc(G).
$$
Hence, Barr-Beck asserts that $\Rep(G)$ is equivalent to the category $\Oc(G)\mod_{G \times G}$. In other words, $\Oc(G)$ is an algebra object in $\Rep(G) \boxtimes \Rep(G)$, and $\Rep(G)$ can be realized as $\Rep(G) \boxtimes \Rep(G)$-modules which have the structure of an algebra over the $\Rep(G) \boxtimes \Rep(G)$-algebra $\Oc(G)$.  The equivalence in this case expresses the obvious isomorphism of stacks, $G / (G\times G) \simeq \bullet/G$.

Note that the map $Y \rightarrow X$ is a fibration with fiber $G$. Barr-Beck for the map of varieties $G \rightarrow \bullet$ is easy--this is just a map of affine varieties as in the previous example. Thus we are really just revisiting \cref{ex:aff} in a family over the stack $X=\bullet/(G \times G)$.
\end{example}

\begin{example}\label{ex:GmodN} Let $Y=N\backslash G$ and let $X$ be a point. The left adjoint $L = \pi^*$ takes a vector space $V$ to the trivial vector bundle over $Y$ with fiber $V$. The right adjoint $R= \pi_*$ is the functor of global sections. Barr-Beck does not allow us to reconstruct $\QC(Y)$ inside $\QC(X)$, because $\QC(Y)$ does not have a compact-projective generator. The structure sheaf is a generator, but the global sections functor $\Hom( \Oc , -)$ does not preserve colimits (it is not right exact). Therefore, we have an open embedding 
$$\QC(N\backslash G) \hookrightarrow \Oc(N\backslash G)\mod.$$
Geometrically, $N\backslash G$ is quasi-affine, so it sits as an open set in $\overline{N\backslash G} = \Spec(\Oc(N\backslash G))$, and the orthogonal complement to $\QC(N\backslash G)$ in $\QC(\overline{N\backslash G})$ is the Serre subcategory of torsion sheaves supported on the complementary closed subvariety. 
\end{example}

\begin{example}[Classical parabolic induction]
\label{ex:para}
Let $X=\bullet/(G \times T)$ and let $Y=\bullet/B$. The map $B \rightarrow G \times T$ induces a natural map $\pi\colon Y \rightarrow X$.  Recall that $\pi^*$ equips $\QC(Y)=\Rep(B)$ with the structure of a $\QC(X)=\Rep(G\times T)$-module category.  Let $\one_B$ denote the trivial representation of $B$. It is a $\Rep(G\times T)$-generator, but not $\Rep(G\times T)$-compact projective.
Then we have the module functor,
$$L=\act_{\one_B} \colon \Rep(G\times T) \longrightarrow \Rep(B), \qquad V \boxtimes \chi \longmapsto i^*V \otimes \pi^*\chi.$$

We then have
$$
\iEnd_{G \times T}(\one_B) \simeq \bigoplus_{\la \in P_+} V_\la \boxtimes\mathbb{C}_{-\la} \simeq \Oc(N\backslash G),
$$
with the latter isomorphism being one of algebras in $\Rep(G\times {T})$.
Thus $\Rep(B)$ can be realized as an open subcategory of $G\times T$-equivariant $\Oc(N\backslash G)$-modules. By the classical $\mathrm{Proj}$-construction applied to the $T$-graded ring $\Oc(N\backslash G)$, this open subcategory is further equivalent to the category of $G$-equivariant sheaves on the flag variety $B\backslash G$.   Thus the monadic comparison functor is nothing but the classical parabolic induction equivalence
\[
\Rep(B)\simeq \QC_G(B\backslash G), \qquad M \longmapsto M\times_B G,
\]
sending a $B$-module $M$ to the $G$--equivariant vector bundle $M\times_B G$ on $B\backslash G$. Note that the left adjoint to this functor is given by taking the fiber $\mathcal{F}_{e}$ of a $G$--equivariant sheaf $\mathcal{F}$ on $B\backslash G$ over the identity element $e\in G$; this fiber carries a $B$-action since the identity coset $Be$ is a fixed point for the action of $B$ on $B\backslash G$ by right translation.
\end{example}

\subsection{Monadic reconstruction: $\Repq(G)$ and $\Oq(G)$}

As we shall now explain, the monadic approach to the classical parabolic induction construction carries over equally well to the quantum setting. 
\begin{notation}
In what follows we set
$$
\iHom_{H}(V,W) := \iHom_{\Rep_q(H)}(V,W)
$$
for $H$ being $G$, $B$, or $T$. We do the same for $\Rep_q(H)^\op$ and $\rev{\Rep_q(H)}$, so that e.g.
$$
\iEnd_{G \boxtimes \rev{G}}(\one_G) = \iEnd_{\Rep_q(G) \boxtimes \rev{\Rep_q(G)}}(\one_G).
$$
\end{notation}

\begin{example}
\label{ex-Oq}

Consider the rigid tensor category $\cA = \Rep_q(G)^{\op} \boxtimes \Rep_q(G)$. Then $\cM = \Rep_q(G)$ is a module category for $\cA$ with the module structure given by
$$
\cA \otimes \cM \longrightarrow \cM, \qquad (V \boxtimes W) \otimes U \longmapsto V \otimes U \otimes W.
$$
Writing $\one_G \in \cM$ for the unit $U_q(\g)$-module, we obtain the functor
$$
\act_{\one_G} \colon \cA \longrightarrow \cM, \qquad V \boxtimes W \longmapsto V \otimes W.
$$
Note that
\begin{align*}
\Hom_{G^{\op}\times G}\hr{V_\la\boxtimes V_\mu,\iEnd_{G^{\op}\times G}(\one_G)} & =\Hom_{G^{\op}\times G}\hr{V_\la\boxtimes V_\mu,\act^R_{\one_G}(\one_G)} \\[4pt]
&\simeq \Hom_G\hr{V_\la\otimes V_\mu,\one_G}.
\end{align*}
Since the category $\Rep_q(G)$ is semisimple, and
$$
\Hom_G\hr{V_\la\otimes V_\mu,\one_G}=
\begin{cases}
\mathbb{C} &\text{if} \; V_\la\simeq V_\mu^*, \\
0 &\text{otherwise},
\end{cases}
$$
it follows that
$$
\iEnd_{G^{\op}\times G}(\one_G) \simeq \bigoplus_{V \in \mathrm{Irr(\cM)}} V^* \boxtimes V = \bigoplus_{\la \in \Lambda_+} V^*_\la \boxtimes V_\la,
$$
where the first direct sum is taken over the isomorphism classes of simple objects in $\cM$. The algebra structure on $\iEnd_{G\times G^{\op}}(\one_G)$ is given by
$$
(\xi \boxtimes v) \otimes (\zeta \boxtimes w) = (\xi \otimes^{\op} \zeta) \boxtimes (v \otimes w) = (\zeta\otimes \xi) \boxtimes (v \otimes w).
$$
It follows that $\iEnd_{G^{\op}\times G}(\one_G)$ is isomorphic as an algebra in $\Ac$ to the Faddeev--Reshetikhin--Takhtadjan quantum coordinate ring $\Oq(G)$, the Hopf algebra of matrix coefficients of finite-dimensional type I representations of $U_q(\g)$. The algebra $\Oq(G)$ is a $q$--deformation of the Poisson algebra of functions on the simple Lie group $G$ equipped with its standard Poisson-Lie structure  $\pi=r^L-r^R$ given by the difference of the left and right translates of the classical $r$--matrix, see~\cite{STS94}.
\end{example}

\begin{remark}
\label{rem:minors}
Let $l_\lambda \subset V_\lambda$ and $l_{-\lambda}\subset V_{\lambda}^*$ be the highest and lowest weight lines respectively, and consider the \defterm{principal minor} $f_{\lambda}=v^*_{-\lambda}\boxtimes v_\lambda\in \Oq(G)$, where $v^*_{-\lambda},v_\lambda$ are any elements of $l_{-\lambda},l_{\lambda}$ normalized such that $\ha{v^*_{-\la},v_{\la}}=1$. Then since $V_{\lambda+\mu}$ appears with multiplicity one in $V_\lambda\otimes V_\mu$ and $v_\lambda\otimes v_\mu$ spans the corresponding highest weight line, we have $f_{\lambda}f_\mu = f_{\lambda+\mu}$. In particular, the subalgebra of $\Oq(G)$ generated by the principal minors $\{f_\lambda\}_{\lambda\in \Lambda_+}$ is commutative, and isomorphic to the polynomial ring $\mathbb{C}(q)[\Lambda_+]$. 
\end{remark}

\begin{example}
\label{ex-Oq-twisted}
Let us now consider the rigid tensor category $\Ac = \Rep_q(G) \boxtimes {\Rep_q(G)}$. Then $\Mc = \Rep_q(G)$ again has the structure of an $\Ac$--module category, with the module structure now being given by
$$
\Ac\boxtimes\Mc\longrightarrow\Ac, \qquad \hr{V\boxtimes W}\otimes U \longmapsto V\otimes W\otimes U.
$$
The key difference in this case is the nontrivial structure map coming from the braiding in $\Rep_q(G)$: we have the isomorphism
\beq
\label{associator}
\sigma_{W_1,V_2} \colon V_1\otimes W_1\otimes V_2\otimes W_2\otimes U\longrightarrow V_1\otimes V_2\otimes W_1\otimes W_2\otimes U.
\eeq
The same argument used in the previous example shows that as an object of $\Ac$, we again have 
$$
\iEnd_{G^2}(\one_G) \simeq \bigoplus_{\la \in \Lambda_+} V^*_\la \boxtimes V_\la,
$$
but in view of~\eqref{associator} the algebra structure is now given by
\begin{align*}
( \xi\boxtimes v) \otimes (\zeta \boxtimes w) &= \sigma_{V^*,W^*}(\xi\otimes\zeta) \boxtimes (v\otimes w)\\
&=(\Rc_2\cdot\zeta\otimes\Rc_1\cdot\xi) \boxtimes (v\otimes w).
\end{align*}
The  algebra $\iEnd_{G^2}(\one_G)$ thus coincides with $\Oq(G_+)$, the $q$--deformation of the Poisson algebra of functions on the simple Lie group $G$ equipped with the Poisson structure $\pi_+=r^L+r^R$ given by the average of the left and right translates of the classical $r$--matrix, see again~\cite{STS94}.
\end{example}

\begin{remark}
\label{ribbon-rem}
Both algebras $\Oq(G)$ and $\Oq(G_+)$ have subalgebras 
$\Oq(N\backslash G),\Oq(N\backslash G_+)$ given by
$$
\bigoplus_{\la \in P_+}  l_{-\lambda} \boxtimes V_{\lambda}\subset\bigoplus_{\la \in P_+}  V^*_{\lambda} \boxtimes V_{\lambda},
$$
where $l_{-\la}$ is the lowest weight line in $V_{\la}^*$. However, while multiplication rule in $\Oq(N\backslash G)$ is simply given by
$$
(\xi_{-\lambda} \boxtimes v) \otimes (\zeta_{-\mu} \boxtimes w) 
= (\zeta_{-\mu}\otimes \xi_{-\lambda}) \boxtimes (v \otimes w),
$$
in $\Oq(N\backslash G_+)$ we have
$$
(\xi_{-\lambda} \boxtimes v) \otimes (\zeta_{-\mu} \boxtimes w) 
= q^{-(\lambda,\mu)}(\zeta_{-\mu}\otimes \xi_{-\lambda}) \boxtimes (v \otimes w),
$$
owing to the fact that the $\Rc$--matrix acts on the tensor product of lowest-weight vectors by $\Rc_G(\xi_{-\lambda}\otimes\zeta_{-\mu}) = q^{-(\lambda,\mu)}\xi_{-\lambda}\otimes\zeta_{-\mu}$. Nonetheless, the half--ribbon element $\nu^{\frac{1}{2}}_T$ from $\Repq(T)$ induces an isomorphism of algebras in $\Repq(T\times G)$
$$
\nu^{\frac{1}{2}}_T\boxtimes\mathrm{id}\colon\Oq(N\backslash G_+)\rightarrow \Oq(N\backslash G), \qquad \xi_{-\lambda}\otimes v \longmapsto q^{\frac{1}{2}(\lambda,\lambda)}\xi_{-\lambda}\otimes v.
$$
\end{remark}

\begin{example}
\label{ex:gate-opening}
The $\Repq(G^2)$--module category in the \cref{ex-Oq-twisted} can be understood as the pullback of a $\Repq(G)$--module structure by the tensor functor $m_G\colon\Repq(G)^{\boxtimes2}\rightarrow\Repq{G}$, with the tensor structure being given by the braiding in $\Repq(G)$. Such pullbacks will arise often for us as we open additional gates within $G$- and $T$-regions. Of particular importance is the case of a $\Repq(\overline{T})$--module category of the form $\Mc=A\mod_T$, regarded as a $\Repq(\overline{T}^2)$--module category via $m_{\overline{T}}$ with tensor structure coming from the braiding in $\Repq(\overline{T})$. By \cref{insertgate}, we have an equivalence $\Mc\simeq \widetilde{A}\mod_{T^2}$, where
$$
\widetilde{A} = \iEnd_{T^2}(A).
$$
As an object of $\Repq(T^2)$, $\widetilde{A}$ may be determined by a similar argument to that made in \cref{ex-Oq}. Writing $A = \bigoplus_{\lambda\in\Lambda}A^\lambda\otimes\mathbb{C}_{\lambda}$, we find that 
$$
\Hom_{T^2}\hr{\mathbb{C}_\lambda\boxtimes \mathbb{C}_\mu,\widetilde{A}} \simeq\Hom_{A\mod_T}\hr{\mathbb{C}_{\lambda+\mu}\otimes A,A} \simeq A^{\lambda+\mu}.
$$
It follows that we have
\beq
\label{gate-induction}
\widetilde{A} = \bigoplus_{\lambda,\mu\in\Lambda}A^{\lambda+\mu}\otimes\mathbb{C}_\lambda\boxtimes\mathbb{C}_\mu,
\eeq
and recalling~\eqref{associator} we see that the algebra structure reads
\begin{align}
\nonumber
\label{induction-mult-rule}
\widetilde A^{\la,\mu} \ot \widetilde A^{\nu,\rho} &\longra \widetilde A^{\la+\nu,\mu+\rho}\\
\nonumber(a_1 \ot \chi_\la\boxtimes\chi_\mu) \cdot (a_2 \ot \chi_\nu\boxtimes\chi_\rho)&=(a_1a_2)\boxtimes \hr{m_{\overline{T}}^{\otimes2}\circ\overline{\sigma}^{-1}_{\mathbb{C}_\nu\ot\mathbb{C}_\mu}}\hr{\chi_\lambda\ot\chi_\mu\ot\chi_\nu\ot\chi_\rho}\\ 
&= q^{-(\mu, \nu)} a_1a_2 \bt (\chi_{\lambda+\nu} \ot \chi_{\mu+\rho}).
\end{align}
\end{example}

\subsection{Monadic reconstruction: $\Repq(B)$ and $\Oq(N\backslash G)$}
We are now ready to give a monadic description of the category $\Repq(B)$ in terms of the quantized flag variety, by analogy with the classical parabolic induction equivalence described in \cref{ex:para}. The first step in this direction is the computation outlined in the following example. 
\begin{example}\label{ex:OqGN}Recall that $\Repq(B)$ is a module category for $\Repq(G\times \rev T)$.
We write $\one_B$ for the monoidal unit in $\Repq(B)$, and consider the functor
$$
\act_{\one_B} \colon \Repq(G\times \rev T) \longrightarrow \Repq(B), \qquad V \boxtimes \chi \longmapsto i^*V \otimes \pi^*\chi.
$$
Its right adjoint is given by 
\begin{align*}\act^R_{\one_B} \colon \Repq(B) &\longrightarrow \Repq(G\times \rev T), \\
M &\longmapsto  \bigoplus_{\substack{\la \in \Lambda_+ \\ \mu\in \Lambda}} \Hom_{B}\hr{V_\lambda\otimes\mathbb{C}_\mu,M}\otimes V_\lambda\boxtimes\mathbb{C}_\mu,
\end{align*}
which implies
$$
\act^R_{\one_B}(M) \simeq \hr{M\otimes \Fq(G\times T)}^{U_q(\mathfrak{b})},
$$
where we have set
$$
\Fq(G\times T) = \bigoplus_{\substack{\la \in \Lambda_+ \\ \mu\in \Lambda}}\hr{\mathbb{C}_{\mu}^*\ot V_{\lambda}^*}\otimes V_\lambda\boxtimes\mathbb{C}_\mu,
$$
with $U_q(\mathfrak{b})$ acting diagonally in the $\mathbb{C}_{-\mu}\ot V_{\lambda}^*$ tensor factor, and $U_q(\mathfrak{g})\boxtimes U_q(\mathfrak{t})$ acting in $V_{\lambda}\boxtimes\mathbb{C}_{\mu}$.
Let us now turn our attention to the braided-commutative algebra $\iEnd_{G \times \rev{T}}(\one_B).$ Setting $M=\one_B$ in the formula above, we obtain an isomorphism in $\Repq(G\times \rev{T})$
\begin{align*}
\iEnd_{G \times \rev{T}}(\one_B) \simeq \Oq(G\times T)^{U_q(\mathfrak{b})} = \bigoplus_{\la \in P_+} (\mathbb{C}^*_{-\lambda}\otimes l_{-\lambda})\otimes V_\la \boxtimes \C_{-\la}
\end{align*}
where $l_{-\la}$ is the lowest weight line in $V_{\la}^*$. 
On the other hand, recall from \cref{ribbon-rem} the subalgebra 
$$
\Fq(N\backslash G_+)=\bigoplus_{\la \in P_+}  l_{-\lambda} \boxtimes V_{\lambda}\subset \Fq( G_+),
$$
which we regard as an algebra object of $\Rep_q(G)\boxtimes\Rep_q(\overline{T})$ with acting $U_q(\mathfrak{g})$ in the second tensor factor $V_{\lambda}$, and $U_q(\mathfrak{t})$ acting in $l_{-\lambda}$.
Then we have an isomorphism of algebras in $\Rep_q(G)\boxtimes\Rep_q(\overline{T})$
$$
\Oq(N\backslash G_+) \longrightarrow \iEnd_{G \times \rev{T}}(\one_B),  \qquad
v^*_{-\lambda}\otimes v \longmapsto (\chi^\vee_\lambda\otimes v^*_{-\lambda})\otimes v\boxtimes \chi_{-\lambda},
$$
where $\chi^\vee_\lambda,\chi_{-\lambda}$ are any elements of $\mathbb{C}^*_{-\lambda},\mathbb{C}_{-\lambda}$ respectively normalized so that $\ha{\chi^\vee_\lambda,\chi_{-\lambda}}=1$.
Note that by \cref{ribbon-rem}, we may equivalently identify  $\iEnd_{G \times \rev{T}}(\one_B)$ with $\Oq(N\backslash G)$ as an algebra in $\Repq(G\times \overline{T})$, and in the sequel we will always choose to do so in order to avoid extraneous powers of $q$ appearing in various explicit formulas.
\begin{remark} A similar computation will be used to calculate $\widetilde{L}$ in \cref{thm:D1reconst}.
\end{remark}
\end{example}

\begin{prop}
\label{prop:conserv}
The unit object $\one_B\in \Repq(B)$ is a $\Rep_q(G\times \rev T)$-generator, in other words the functor $\iHom_{G\times \rev{T}}(\one_B,-) = \act_{\one_B}^R$ is conservative.
\end{prop}

\begin{proof}
It is equivalent to show in that case that $\act_{\one_B}$ is dominant, i.e. that for any non-zero $U_q(\gb)$-module $M$, there exists some $V\in\Rep_q(G)$ and $\chi\in\Rep_q(T)$ such that $\Hom(i^*V\ot\pi^*\chi,M) \neq 0$.  Clearly we may further assume (by passing to a submodule) that $M$ is cyclic on some weight vector $m_\nu$, so that being finite-dimensional, it is a quotient of $U_q(\gb)\ot_{U_q(\gt)} \mathbb{C}_\nu$, for some $T$-weight $\nu$.  We may choose $V=V_\lambda$ for $\lambda$ suitably large that the annihilator of $V$, intersected with the subalgebra $U_q(\gn)$, contains the annihilator in $U_q(\gn)$ of $M$.  Then there is a unique $U_q(\gn)$-module map from $V_\lambda\to M$ sending the highest weight vector $v_{\lambda}$ to $m_\nu$ and the only obstruction to this being a $U_q(\gb)$-map is that $\lambda\neq \nu$.  Hence we obtain a non-zero map $i^*V_\lambda\ot\pi^*\mathbb{C}_{\nu-\lambda}\to M$.
\end{proof}

\begin{defn}\label{def:torsion}
An $\Oq(N\backslash G)$-module $M$ is \defterm{torsion} if, for every $m\in M$, there exists a $\lambda$ such that $\mu\geq \lambda$ implies that $V_\mu m=0$.  We denote the collection of torsion modules by $\Torsion$.  A module $M$ is \defterm{torsion-free} if it does not contain any non-zero torsion submodules.
\end{defn}

We note that the sum $\Torsion(M)$ of all torsion submodules of any module $M$ is again a submodule of $M$, and that the quotient of $M$ by this submodule is torsion free.  Moreover, $\Torsion$ is a Serre subcategory, meaning that it is closed under formation of long exact sequences, and the quotient category is again abelian. 

\begin{lemma}\label{lem:invideals}  We have the following characterization of $U_q(\g)\otimes U_q(\mathfrak{t})$-invariant ideals in $\Oq(N\backslash G)$.
\begin{enumerate}
\item For a dominant weight $\la$, the subspace $I_\la = \oplus_{\mu \geq \la} (V_\mu\boxtimes \mathbb{C}_{-\mu})$ is an invariant ideal.
\item Every invariant ideal in $I$ is a finite sum $I=\sum_i I_{\la_i}$ of such ideals.
\end{enumerate}
\end{lemma}
\begin{proof}
Claim (1) is clear, since multiplication in $\Oq(N\backslash G)$ only increases weights in the ordering.  For Claim (2), we note that every $U_q(\g)\otimes U_q(\mathfrak{t})$-subspace of $\Oq(N\backslash G)$ is a direct sum of its full isotypic components, because $\Oq(N\backslash G)$ is multiplicity-free.  Moreover, we recall that the multiplication in $\Oq(N\backslash G)$ is such that $m(V_\la\ot V_\mu) = V_{\la+\mu}$, so that once $V_\la$ occurs in the ideal, all its dominant translates do as well. 
\end{proof}

\begin{cor}\label{cor:no-ann}
An object $M$ of $\Oq(N\backslash G)\mod_{G\times\rev{T}}$ is torsion-free if, and only if all its compact subobjects have zero annihilator ideals.
\end{cor}
\begin{proof}
If we suppose $M$ is torsion, clearly all its compact subobjects have non-trivial annihilator ideals.  Conversely, as soon as the annihilator ideal of any compact subobject is non-empty, then it is of the form $I=\sum_iI_{\la_i}$ described in \cref{lem:invideals}, which implies that it is torsion, since we may take $\la=\la_i$ for any $i$ in the definition of a torsion module.
\end{proof}

\begin{theorem}\label{thm:D1reconst}
The comparison functor
\begin{align*}
    \iHom_{G\times \rev T}(\one_B,-) \colon \Repq(B) &\longra \Oq(N\backslash G)\mod_{G\times \rev T}\\
    V&\longmapsto \hr{V\otimes \Oq(G\times T)}^{U_q(\mathfrak{b})},
    \end{align*}
defines a fully faithful embedding, realizing $\Repq(B)$ as the open subcategory $\Torsion^\perp$.
\end{theorem}

\begin{proof}  
Thanks to \cref{prop:conserv}, we may apply Theorem~\ref{weak-reconstruction} to conclude that $\iHom_{G\times \rev T}(\one_B,-)$ yields a fully faithful embedding of $\Repq(B)$ as an open subcategory of $\Oq(N\backslash G)\mod_{G\times \rev T}$. So by \cref{prop:kernel}, all that remains is to identify the kernel of its left adjoint $\widetilde{L}$ with $\Torsion$. In complete analogy with the classical parabolic induction functor from \cref{ex:para}, this left adjoint is given by
$$
\widetilde{L}\colon \Oq(N\backslash G)\mod_{G\times \rev T}\longrightarrow \Repq{B}, \qquad M \longmapsto M/(\ker \eps) M,
$$
where $\ker \eps$ is the 2-sided ideal in $\Oq(N\backslash G)$ given by the kernel of the evaluation morphism $\eps \colon \Oq(N\backslash G)\rightarrow \one_B$, and the $U_q(\mathfrak{b})$-action on $M/(\ker \eps) M$ is inherited from that of $U_q(\mathfrak{g})\otimes U_q(\mathfrak{t})$ on $M$.

Now suppose that $M$ is an object of $\Torsion$.  

Recall the principal minors $\{f_{\lambda}\}_{\lambda\in\Lambda_+}$ defined in \cref{rem:minors}, and observe that $f_\lambda\in\Oq(N\backslash G)$, while $\eps(f_{\lambda})=1$ for all $ \lambda\in\Lambda_+$. So for each $m\in M$, choosing $\la$ sufficiently large that $f_{\la}\cdot m=0$, we have $m=1\cdot m = (1-f_{\la}) m \in (\ker\eps)M$, and thus $\widetilde{L}(M)=0$. 

To conclude the proof, we must show that if $M$ is torsion-free, then $M/\ker\eps M\neq0$. For this, let us first suppose that $M$ is finitely generated as an $\Oq(N\backslash G)$--module.     Given a weight $\la$, denote by $M[\la]$ the corresponding weight space  of $M$ with respect to the $U_q(\mathfrak{t})$-action. Let $\la_0$ be the highest weight such that $M[\la_0] \ne 0$. Since for any weight $\la$
$$
\ker(\eps)M[\la] \subseteq \oplus_{\mu\le\la}M[\mu],
$$
the condition $M = \ker(\eps)M$ implies $M[\la_0] \subseteq \ker(\eps)M[\la_0]$. Now let $\hc{m_1,\dots,m_n}$ be a basis of $M[\la_0]$. Then any element $m \in M[\la_0]$ can be written as
\beq
\label{eqn:tors-proof-1}
m = \sum_i h_i m_i
\eeq
for some $h_i\in\ker\eps$. Note that by comparing the  $T$-weights of both sides, we may assume that the $h_i$ are elements of the (commutative) polynomial ring $\mathbb{C}(q)[\Lambda_+]$ generated by the principal minors $f_\lambda$. We may thus make a similar argument to that used to prove the classical Nakayama lemma: expand
\begin{align}
\label{nak}
m_i = \sum_{j}f_{ij}m_j
\end{align}
and form the $n\times n$--matrix $\delta_{ij}-f_{ij}$ with entries in $\mathbb{C}(q)[\Lambda_+]$. Multiplying~\cref{nak} through by its adjugate matrix, we find that $\det(\delta_{ij}-f_{ij})$ annihilates all $m_i$, and is of the form $1+\ker\eps$. The annihilator of the $m_i$ is thus nonzero, which by Proposition~\ref{cor:no-ann} contradicts the assumption that $M$ is torsion-free.

To treat the general case, we note that $\widetilde{L}$ is exact on torsion free modules.  Hence, if $\widetilde{L}(M)=0$, then for any finitely generated submodule $M'\hookrightarrow M$ we must also have $\widetilde{L}(M')=0$. By the preceding discussion this implies $M'$ is torsion.  Hence $M$ is torsion, as required.  
\end{proof}

\subsection{Monadic reconstruction: $\Zc(\DD_n)$ and $\Oq(\Conf_n)$}

\begin{defn} We define the braided tensor category $\Repq(G\times \rev{T}^{\nsp n})$ by
$$
\Repq(G\times \rev{T}^{\nsp n}):= \Repq(G) \bt \rev{\Repq(T)}^{\nsp \bt n}.
$$
\end{defn}

\begin{defn} We endow the categories $\Zc(\DD_n)$ with an action of the monoidal category $\Repq(G\times \rev{T}^{\nsp n}),$ by choosing a single interval in $\partial S \cap S_G$, and one in each connected component of $\partial S \cap S_T$, as indicated in~\cref{fig:D5-insert}.  We let $\Dist$ denote the distinguished object of the category $\Zc(\DD_n)$. \end{defn}

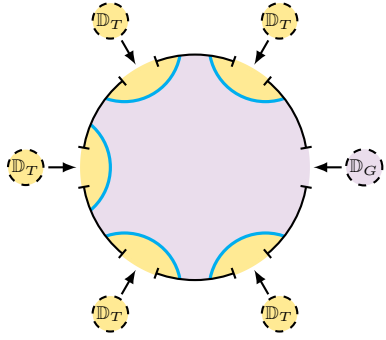
\begin{figure}[t]
\begin{tikzpicture}[every node/.style={inner sep=0.5, thick, circle}, x=0.75cm, y=0.75cm]

\def\Gcirc{(0,0) circle (2)};
\fill[Orchid!20] \Gcirc;
\def\Tone{(60:2.5) circle (1)};
\def\Ttwo{(120:2.5) circle (1)};
\def\Tthree{(180:2.5) circle (1)};
\def\Tfour{(240:2.5) circle (1)};
\def\Tfive{(300:2.5) circle (1)};
\begin{scope}
	\clip\Gcirc;
	\fill[Goldenrod!60] \Tone \Ttwo \Tthree \Tfour \Tfive;
	\draw[very thick, cyan] \Tone \Ttwo \Tthree \Tfour \Tfive;
\end{scope}
\draw[thick] \Gcirc;
\draw[ultra thick, Orchid!20, domain=-10:10] plot ({2*cos(\x)}, {2*sin(\x)});
\draw[thick, domain=1.9:2.1] plot ({\x*cos(10)}, {\x*sin(10)});
\draw[thick, domain=1.9:2.1] plot ({\x*cos(-10)}, {\x*sin(-10)});

\draw[ultra thick, Goldenrod!60, domain=50:70] plot ({2*cos(\x)}, {2*sin(\x)});
\draw[thick, domain=1.9:2.1] plot ({\x*cos(50)}, {\x*sin(50)});
\draw[thick, domain=1.9:2.1] plot ({\x*cos(70)}, {\x*sin(70)});

\draw[ultra thick, Goldenrod!60, domain=110:130] plot ({2*cos(\x)}, {2*sin(\x)});
\draw[thick, domain=1.9:2.1] plot ({\x*cos(110)}, {\x*sin(110)});
\draw[thick, domain=1.9:2.1] plot ({\x*cos(130)}, {\x*sin(130)});

\draw[ultra thick, Goldenrod!60, domain=170:190] plot ({2*cos(\x)}, {2*sin(\x)});
\draw[thick, domain=1.9:2.1] plot ({\x*cos(170)}, {\x*sin(170)});
\draw[thick, domain=1.9:2.1] plot ({\x*cos(190)}, {\x*sin(190)});

\draw[ultra thick, Goldenrod!60, domain=230:250] plot ({2*cos(\x)}, {2*sin(\x)});
\draw[thick, domain=1.9:2.1] plot ({\x*cos(230)}, {\x*sin(230)});
\draw[thick, domain=1.9:2.1] plot ({\x*cos(250)}, {\x*sin(250)});

\draw[ultra thick, Goldenrod!60, domain=290:310] plot ({2*cos(\x)}, {2*sin(\x)});
\draw[thick, domain=1.9:2.1] plot ({\x*cos(290)}, {\x*sin(290)});
\draw[thick, domain=1.9:2.1] plot ({\x*cos(310)}, {\x*sin(310)});

\node[draw, dashed, fill=Orchid!20] at (0:3) {\tiny$\DD_G$};
\node[draw, dashed, fill=Goldenrod!60] at (60:3) {\tiny$\DD_T$};
\node[draw, dashed, fill=Goldenrod!60] at (120:3) {\tiny$\DD_T$};
\node[draw, dashed, fill=Goldenrod!60] at (180:3) {\tiny$\DD_T$};
\node[draw, dashed, fill=Goldenrod!60] at (240:3) {\tiny$\DD_T$};
\node[draw, dashed, fill=Goldenrod!60] at (300:3) {\tiny$\DD_T$};

\draw[thick, ->] (0:2.6) to (0:2.1);
\draw[thick, ->] (60:2.6) to (60:2.1);
\draw[thick, ->] (120:2.6) to (120:2.1);
\draw[thick, ->] (180:2.6) to (180:2.1);
\draw[thick, ->] (240:2.6) to (240:2.1);
\draw[thick, ->] (300:2.6) to (300:2.1);

\end{tikzpicture}
\caption{A gated pentagon $\DD_5$ with a depiction of the action by disk insertions.}
\label{fig:D5-insert}
\end{figure}

\begin{defn}
For $n\geq 1$, the \defterm{quantum coordinate algebra of $n$ framed flags} is the $n$-fold braided tensor product,
\[\Oq(\Conf_n) := \Oq(N\backslash G)\otimes \cdots \otimes \Oq(N\backslash G).\]
\end{defn}

By definition, this means that as an object it is the tensor product in $\Repq(G\times \rev T^n)$, and that multiplication within each factor is given by the algebra structure of $\Oq(N\backslash G)$ and that relations between the factors are given by
\beq
\label{cross-rel}
b^{(j)} a^{(i)} = \hr{\Rc_{(2)} \rhu a^{(i)}} \hr{\Rc_{(1)} \rhu b^{(j)}},
\eeq
where $1 \le i < j \le n$ and $x^{(i)}$ denotes the element $x$ placed in the $i$-th braided tensor factor, where $\Rc$ denotes the universal $R$-matrix of $U_q(\g)$.  We note that each subalgebra $\Oq(N\backslash G)$ occupies a \emph{distinct} $\Repq(T)$ factor, while sharing the same $\Repq(G)$ factor, hence the braided tensor product involves only the universal $\Rc$-matrix for $U_q(\g)$.

\begin{theorem}\label{thm:DnOqConf}
Let $\Dist_{\DD_n}\in\Zc(\DD_n)$ denote the distinguished object.  We have an isomorphism
$$\iEnd_{G \times \rev T^{\nsp n}}(\Dist_{\DD_n}) \cong \Oq(\Conf_n).$$
\end{theorem}

\begin{proof}
The case $n=1$ has been already treated in \cref{ex:OqGN}. For $n>1$ we only need to apply excision. By induction we can write $\DD_n = \DD_1\sqcup_{\DD_G} \cdots \sqcup_{\DD_G} \DD_1$, hence we obtain an equivalence of categories,
$$
\Zc(\DD_n) \simeq \Rep_q(B) \mathop{\bt}_{\Rep_q(G)} \cdots \mathop\bt_{\Rep_q(G)} \Rep_q(B),
$$
and therefore applying \cite{BBJ18a}[Corollary 4.13], we have an isomorphism,
$$
\iEnd_{G\times \rev T^{\nsp n}}(\Dist_{\DD_n}) \simeq \iEnd_{G\times \rev T}(\one_B) \widetilde{\otimes} \cdots \widetilde{\otimes} \iEnd_{G\times \rev T}(1_B),
$$
where $\widetilde{\otimes}$ denotes the braided tensor product of algebras in $\Rep_q(G)$, as required.

\end{proof}

\begin{defn}  For $n\geq 1$, we denote by $\Torsion_n\subset \Oq(\Conf_n)\mod_{G\times T^n}$ the full Serre subcategory consisting of modules whose restriction to some $\Oq(\Conf_1)$ subfactor is torsion.
\end{defn}

\begin{prop}
\label{disk-computation}
For any $n\geq 1$, the functor $\iHom_{G\times \rev{T}^n}(\Dist_{\DD_n},-)$ induces an equivalence of categories,
\[
\Zc(\DD_n) \simeq \Torsion_n^\perp \subset \Oq(\Conf_n)\mod_{G\times \rev{T}^n}.
\]
\end{prop}

\begin{proof}
Let $\widetilde{L}_n$ denote the left adjoint.  By \cref{prop:kernel} we need to identify the kernel of $\widetilde{L}$ with $\Torsion_n$.  The case $n=1$ is precisely \cref{thm:D1reconst}.  For $n> 1$, we have have the equivalence,
\[
Z(\DD_n) \simeq \Repq(B)\underset{\Repq(G)}{\bt}\cdots\underset{\Repq(G)}{\bt}\Repq(B).
\]
Under this equivalence $\widetilde{L}$ identifies with the functor, $\widetilde{L}_1\underset{\Repq(G)}{\bt}\cdots \underset{\Repq(G)}{\bt}\widetilde{L}_1$.  The canonical functor from $\widetilde{L}_1\bt\cdots \bt\widetilde{L}_1$ to the relative tensor product is conservative, hence we have $\ker(\widetilde{L}_n) = \ker(\widetilde{L}_1)\underset{\Repq(G)}{\bt}\cdots \underset{\Repq(G)}{\bt}\ker(\widetilde{L}_1)=\Torsion_n$. 

\end{proof}

\subsection{Further computations of $\Oq(\Conf_n)$ for $G=\SL_2$} 
\label{A-sec}

For $G=SL_2$, the algebra $\Fq(G)$ is generated by matrix coefficients of the first fundamental representation, and thus can be written as
$$
\Fq(G) \simeq \C(q)\ha{a_{ij}} \qquad\text{with}\qquad i,j={1,2},
$$
where
$$
a_{11}a_{12} = qa_{12}a_{11}, \qquad a_{11}a_{21} = qa_{21}a_{11}, \qquad a_{12}a_{21}=a_{21}a_{12},
$$
and
$$
a_{11}a_{22} - qa_{12}a_{21} = 1.
$$
Spelling out the above equality, and setting $x_1 = a_{11}$ and $x_2 = a_{12}$ in each factor, we arrive at the following Proposition. 

\begin{prop}
\label{prop:gensrels}
The algebra $\Oq(\Conf_n)$ is an associative algebra over $\C(q)$ with generators 
$$
x^{(i)}_1, x^{(i)}_2, \quad 1\le i \le n,
$$ 
and relations 
$$
x^{(i)}_1 x^{(i)}_2 = q x^{(i)}_2 x^{(i)}_1 \quad \text{for all} \quad 1\le i \le n,
$$
and

\begin{align}
\nonumber x_1^{(j)} x_1^{(i)} &= q^{-\frac12} x_1^{(i)} x_1^{(j)}, \\
\nonumber x_1^{(j)} x_2^{(i)} &= q^{\frac12} x_2^{(i)} x_1^{(j)}, \\
\label{commutator-rel} x_2^{(j)} x_1^{(i)} &= q^{\frac12} \hr{x_1^{(i)} x_2^{(j)} + (q^{-1}-q) x_2^{(i)} x_1^{(j)}}, \\
\nonumber x_2^{(j)} x_2^{(i)} &= q^{-\frac12} x_2^{(i)} x_2^{(j)}.
\end{align}
for all $1\le i<j \le n$.
\end{prop}

All the above identities in \cref{prop:gensrels} follow from viewing functions on $N \backslash G$ as a $G$-representation, and using the $R$-matrix to braid the factors in $\Oq((N \backslash G)^n)$ before multiplying.

\begin{defn}
\label{defn-Z}
For any pair $1 \le i,j \le n$ such that $i \ne j$, we define elements $D_{ij}\in\Oq(\Conf_n)$ by
$$
D_{ij} = x_1^{(i)} x_2^{(j)} - q x_2^{(i)} x_1^{(j)}.
$$
\end{defn}

\begin{lemma}
\label{delta-automorphy}
For $1 \le i < j \le n$ we have
$$
D_{ji} = -q^{\frac32} D_{ij}.
$$
\end{lemma}

\begin{proof}
The proof follows immediately from relation~\eqref{commutator-rel} in \cref{prop:gensrels}:
\begin{align*}
D_{ji} &= x_1^{(j)} x_2^{(i)} - q x_2^{(j)} x_1^{(i)} \\
&= q^{\frac12} x_2^{(i)} x_1^{(j)} -q^{\frac32}\hr{x_1^{(i)}x_{2}^{(j)}+(q^{-1}-q)x^{(i)}_2x_1^{(j)}} \\
&= -q^{\frac32} D_{ij}.
\end{align*}
\end{proof}

\begin{prop}
\label{delta-x-rels}
Elements $D_{ij}$ are invariant with respect to the (left coregular) $U_q(\sl_2)$-action, that is $D_{ij} \in \Oq(\Conf_n)^{U_q(\sl_2)}$. Moreover, for all $1 \le i < j \le n$,  and $a=1,2$, they satisfy commutation relations
$$
D_{ij} x_a^{(k)} = 
\begin{cases}
q^{\frac12} x_a^{(i)} D_{ij}, &\text{if} \;\; k=i, \\
q^{-\frac12} x_a^{(j)} D_{ij}, &\text{if} \;\; k=j, \\
x_a^{(k)} D_{ij}, &\text{if} \;\; k<i \;\; \text{or} \;\; k>j,
\end{cases}
$$
as well as
\begin{align*}
D_{ij} x_a^{(k)} &= q^{-\frac{1}{2}} x_a^{(j)} D_{ik} + q^{\frac{1}{2}} x_a^{(i)} D_{kj}, \\
x_a^{(k)} D_{ij} &= q^{\frac{1}{2}} D_{ik} x_a^{(j)} + q^{-\frac{1}{2}} D_{kj} x_a^{(i)}
\end{align*}
if $i < k < j$.
\end{prop}

\begin{proof}
Note that we have
\begin{align*}
E \rhu x_1 &= 0, & K \rhu x_1 &= qx_1, & F \rhu x_1 &= x_2, \\
E \rhu x_2 &= x_1, & K \rhu x_2 &= q^{-1}x_2, & F \rhu x_2 &= 0.
\end{align*}
The proof of the invariance of degree 2 element $D_{ij}$ then follows from a direct calculation using the module algebra structure. For example, we have
\begin{align*}
E \rhu D_{ij}
&= \hr{(E \rhu x_1^{(i)}) x_2^{(j)} + (K \rhu x_1^{(i)}) (E \rhu x_2^{(j)})} \\
&- q\hr{(E \rhu x_2^{(i)}) x_1^{(j)} + (K \rhu x_2^{(i)}) (E \rhu x_1^{(j)})} \\
&= qx_1^{(i)}x_1^{(j)} - qx_1^{(i)}x_1^{(j)} = 0.
\end{align*}

As for the commutation relations, we shall only prove the top equality here, with the other cases being proved in a similar fashion. Substituting $a=1,2$ we obtain
\begin{align*}
D_{ij} x_2^{(i)} 
&= x_1^{(i)} x_2^{(j)} x_2^{(i)} - q x_2^{(i)} x_1^{(j)} x_2^{(i)} \\
&= q^{\frac12}x_2^{(i)} x_1^{(i)} x_2^{(j)}  - q^{\frac32} x_2^{(i)}x_2^{(i)} x_1^{(j)} \\
&= q^{\frac{1}{2}} x_2^{(i)} D_{ij}
\end{align*}
and
\begin{align*}
D_{ij} x_1^{(i)}
&= x_1^{(i)} x_2^{(j)} x_1^{(i)} - q x_2^{(i)} x_1^{(j)} x_1^{(i)} \\
&= q^{\frac12}x_1^{(i)}\hr{x_1^{(i)} x_2^{(j)} + (q^{-1}-q) x_2^{(i)}x_1^{(j)}} - q^{-\frac12}x_1^{(i)} x_2^{(i)} x_1^{(j)}\\
&= q^{\frac{1}{2}} x_2^{(i)} \hr{x_1^{(i)} x_2^{(j)} -) x_2^{(i)}x_1^{(j)}}\\
&= q^{\frac{1}{2}} x_2^{(i)} D_{ij}.
\end{align*}
\end{proof}

As an immediate corollary of \cref{delta-x-rels}, we obtain the following commutation relations between the $D_{ij}$.

\begin{cor}
\label{cor-deltas}
For $i<k<j$, we have 
\begin{align*}
D_{ik} D_{kj} &= q^{-\frac{1}{2}} D_{kj} D_{ik}, \\
D_{kj} D_{ij} &= q^{-\frac{1}{2}} D_{ij} D_{kj}, \\
D_{ij} D_{ik} &= q^{-\frac{1}{2}} D_{ik} D_{ij}.
\end{align*}
\end{cor}

We also obtain the following exchange relation:

\begin{cor}
\label{cor-mut}
For $i<k<j<l$, we have 
$$
D_{ij} D_{kl} = q^{-\frac{1}{2}} D_{ik} D_{jl} + q^{\frac{1}{2}} D_{kj} D_{il}.
$$
\end{cor}

Next, we will show that certain sets of functions $D_{ij}$ are multiplicative Ore sets, which will allow us to localize the ring $\Oq(\Conf_n)$.

We will say that an $n$-gon $\DD_n$ is \emph{based} if its vertices are labelled 1 to $n$ in a counterclockwise order. For any $1 \le k < n$, a \emph{step $k$ diagonal} $e_{ij}$ in a based $n$-gon $\DD_n$ is a segment connecting vertices $i$ and $j$ with $j-i=k$. In particular, all step 1 diagonals are sides of $\DD_n$, as well as the only step $(n-1)$ diagonal $e_{1n}$. In what follows, diagonals of step 2 will be referred to as \emph{short diagonals}. We say that diagonals $e_{ij}$ and $e_{kl}$ are \emph{crossing} if either of the following two inequalities holds:
$$
i < k < j < l \qquad\text{or}\qquad k < i < l < j.
$$
A \emph{triangulation} of $\DD_n$ is a maximal collection of its non-crossing diagonals. Note that every triangulation contains all sides of $\DD_n$ and at least one short diagonal.  We will call any collection, not necessarily maximal, of non-crossing diagonals of $\DD_n$ its \emph{partial triangulation.}

To an algebra $\Oq(\Conf_n)$ let us associate an $n$-gon $\DD_n$, whose vertices represent tensor factors of $\Oq(\Conf_n)$ and are labelled accordingly. A partial triangulation $\tri$ of $\DD_n$ gives rise to the multiplicative set $S_{\tri} \subset \Oq(\Conf_n)$ generated by elements $D_{ij}$ such that $e_{ij} \in \tri$. Note that by \cref{cor-deltas}, all elements of $S_{\tri}$ skew-commute.

\begin{lemma}
\label{lem:Ore}
For any element $r \in \Oq(\Conf_n)$ and any triangulation $\tri$ of $\DD_n$, there exist elements $s_r, s'_r \in S_{\tri}$ such that both products
$$
s_rr \qquad\text{and}\qquad rs'_r
$$
can be expressed as sums of monomials in $x_a^{(1)}$, $x_a^{(n)}$, and $D_{ij} \in S_{\tri}$.
\end{lemma}

\begin{proof}
We shall prove the Lemma by induction on $n$. For $n=2$ the statement is immediate. Note that by~\cref{prop:gensrels}, any element $r \in \Oq(\Conf_n)$ can be uniquely written as a sum
$$
r = \sum_{\ell=1}^m c_\ell \hr{x_{1}^{(1)}}^{\alpha_{\ell,1}} \hr{x_{2}^{(1)}}^{\beta_{\ell,1}} \dots \hr{x_{1}^{(n)}}^{\alpha_{\ell,n}} \hr{x_{2}^{(n)}}^{\beta_{\ell,n}},
$$
where $c_\ell \in \C(q)$ and $m$ is the number of monomials in the expression. Let us set
$$
d_i = \max_\ell\hr{\alpha_{\ell,i} + \beta_{\ell,i}}.
$$
Now, consider a triangulation $\tri$ of $\DD_n$ and any of its short diagonals $e = e_{i-1,i+1} \in \tri$. By \cref{delta-x-rels} and \cref{cor-deltas}, we have relations
\begin{align*}
D_e x_a^{(i)} &= q^{-\frac{1}{2}} x_a^{(i+1)} D_{i-1,i} + q^{\frac{1}{2}} x_a^{(i-1)} D_{i,i+1}, \\
x_a^{(i)} D_e &= q^{\frac{1}{2}} D_{i-1,i} x_a^{(i+1)} + q^{-\frac{1}{2}} D_{i,i+1} x_a^{(i-1)}
\end{align*}
and know that $D_e$ skew-commutes with all monomials in the right hand side of the above equalities. This in turn implies that the products
$$
D_e^{d_i} \cdot r \qquad\text{and}\qquad r \cdot D_e^{d_i}
$$
can be expressed as sums of monomials in elements $D_{i-1,i}$, $D_{i,i+1}$, and $x_a^{(j)}$ with $j \ne i$.

Now, let $\DD_{n-1}$ be the $(n-1)$-gon obtained from $\DD_n$ by erasing the sides $e_{i-1,i}$ and $e_{i,i+1}$. Choosing a short diagonal $e' = e_{j-1,j+1}$ of $\DD_{n-1}$ and repeating the above reasoning, we derive that the products
$$
D_{e'}^{d'_j} D_e^{d_i} \cdot r
\qquad\text{and}\qquad
r \cdot D_e^{d_i} D_{e'}^{d'_j}
$$
can be expressed as sums of monomials in elements $D_{i-1,i}$, $D_{i,i+1}$, $D_{j-1,j}$, $D_{j,j+1}$, and $x_a^{(k)}$ with $k \ne i,j$. The rest of the proof is done by induction on the number $n$ of sides in $\DD_n$.
\end{proof}

\begin{cor}
\label{cor:Ore}
For any triangulation $\tri$ of $\DD_n$, the collection $S_{\tri}$ forms a multiplicative (left and right) Ore set in the ring $R = \Oq(\Conf_n)$: for all $r \in \Oq(\Conf_n)$ and $s\in S_\tri$ one has
$$
S_\tri r\cap Rs \neq \emptyset \qquad\text{and}\qquad rS_\tri\cap sR \neq \emptyset.
$$
\end{cor}

\begin{proof}
By~\cref{lem:Ore}, for any $r \in R$ there exist elements $s_r, s'_r \in S_\tri$ so that the products $s_rr$ and $r s'_r$ can be expressed as sums of monomials in $x_a^{(1)}$, $x_a^{(n)}$, and $D_{ij} \in S_\tri$. Since elements of~$S_\tri$ skew-commute with each other, as well as with $x_a^{(1)}$ and $x_a^{(n)}$, we conclude that
$$
s s_rr = r's \qquad\text{and}\qquad r s'_r s = sr''
$$
for some $r', r'' \in R$.
\end{proof}

\begin{defn}
Let $\tri$ be a partial triangulation of $\DD_n$. We will denote by $\OqConftri{n}$ the Ore localization of $\Oq(\Conf_n)$ at the multiplicative Ore set $S_\tri$.
In particular, we will write $\OqConfe{n}{e_{ij}}$ for $\Oq(\Conf_n)[D_{ij}^{-1}]$. 
\end{defn}

\begin{lemma}
\label{Oq-morphism-lemma}
For $n \ge 2$ and $1 \le i < j \le n$, there is an injective map of $U_q(\g)$-module algebras
$$
\eta_{ij} \colon \Oq(G) \longra \OqConfe{n}{e_{ij}}
$$
defined by
\begin{align*}
&a_{11} \longmapsto x_1^{(i)}, \qquad\qquad a_{21} \longmapsto x_1^{(j)} D_{ij}^{-1}, \\
&a_{12} \longmapsto x_2^{(i)}, \qquad\qquad a_{22} \longmapsto x_2^{(j)} D_{ij}^{-1}.
\end{align*}
\end{lemma}

\begin{proof}
The statement of the Lemma is easily verified with the help of relations established in \cref{prop:gensrels} and \cref{delta-x-rels}.
\end{proof}

\subsection{Cluster charts on $\Zc(\DD_n)$ for $G=\SL_2$}\label{sec:disk_charts}
We have the algebra $\Oq(\Conf_n)$ along with its localization $\OqConftri{n}$. We would like to relate the localization with $\Zc(\DD_n)$. By construction the two inclusions indicated by solid arrows below define open subcategories: we collect the embedding functors and their left adjoints in the diagram below.

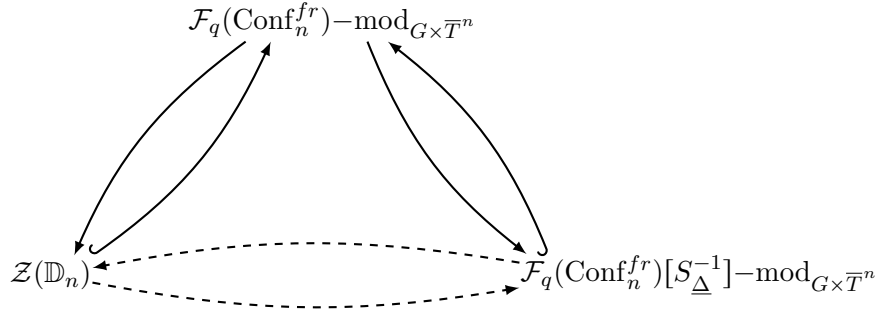
\[
\begin{tikzpicture}[every node/.style={inner sep=0, minimum size=0.5cm, thick, draw=none}, x=0.75cm, y=0.75cm]

\node (1) at (0,0) {$\Zc(\DD_n)$};
\node (2) at (11.5,0) {$\OqConftri{n}\mod_{G\times \rev T^{\nsp n}}$};
\node (3) at (5,4.5) {$\Oq(\Conf_n)\mod_{G\times \rev T^{\nsp n}}$};

\draw [->, thick, dashed] (1.-10) to [bend right = 10] (2.-176);
\draw [->, thick, dashed] (2.176) to [bend right = 10] (1.10);
\draw [->, thick, right hook-latex] (1.25) to [bend right = 15] (3.-162);
\draw [->, thick] (3.-167) to [bend right = 15] (1.40);
\draw [->, thick, right hook-latex] (2.172) to [bend right = 15] (3.-20);
\draw [->, thick] (3.-30) to [bend right = 15] (2.173);

\end{tikzpicture}
\]
We define the dotted arrows as the evident compositions.  In this section we will show that the dotted arrows realize  $\OqConftri{n}\mod_{G\times T^n}$ as an open subcategory of $\Zc(\DD_n)$. Having already identified $\Zc(\DD_n)$ as the subcategory of torsion-free modules, it remains only to show that localizing to $\OqConftri{n}$ kills all torsion modules.

\begin{theorem}
\label{triangle-open-subcat}
The dotted arrows realize $\OqConftri{n}\mod_{G\times T^n}$ as an open subcategory of $\Zc(\DD_n)$.
\end{theorem}

\begin{proof}
Let us identify the weight lattice of $\SL_2$ with $\mathbb{Z}$ via $\omega_1\mapsto 1, \alpha_1\mapsto 2$. As an algebra object in $\Rep_q(G\times \overline{T}^n)$, $\Oq(\Conf_n)$ is graded by the torus $T^n$; let us write $\Oq(\Conf_n)^{(i;k)}$ for the set of elements of weight $\geq k$ with respect to the $i$th $T$-factor, and $\Oq(\Conf_n)^{k\epsilon_{i}}$ for the subset of elements whose weight with respect to $T^n$ is the vector $k\epsilon_i$ with a $k$ in position $i$ and zeros elsewhere. Then it is evident from the homogeneity of the relations~\eqref{commutator-rel} that we have
	\begin{equation}\label{eqn:homog-rmk}
	\Oq(\Conf_n)^{(i;k)} = \Oq(\Conf_n)\Oq(\Conf_n)^{k\epsilon_i} . 
	\end{equation}

	Suppose some $\Oq(\Conf_n)$-module $M$ admits a non-zero morphism from a torsion module.  Then there must exist an element $m\in M$, an index $1\le k\le n$, and a natural number $N$ such that 
	\[
	r\ge N\implies\Oq(\Conf_n)^{r\epsilon_k}m=0.
	\]
	Now any $Z_{ij}$ such that $k\in\{i,j\}$ is an element of $\Oq(\Conf_n)^{k;1}$. So by \cref{eqn:homog-rmk}, we have $Z_{ij}^r\in \Oq(\Conf_n)^{k;r} =\Oq(\Conf_n)\Oq(\Conf_n)^{r\epsilon_k} $. Therefore $\left(Z_{ij}\right)^rm=0$, which shows that the element $Z_{ij}$ does not act invertibly on such a module $M$.  Thus we have $\OqConftri{n}\mod_{G\times T^n}\subset\Torsion^{\perp}$. 
\end{proof}

\begin{theorem}
\label{thm-inv}
For any partial triangulation $\tri \ne \emptyset$ of a based $n$-gon associated to $\Oq(\Conf_n)$, the functor of taking $U_q(g)$-invariants,
$$
\Gc \colon \Oq(\Conf_n)\mod_{G\times \rev T^{\nsp n}} \longra \Oq(\Conf_n)^{U_q(\g)}\mod_{\rev T^{\nsp n}},
$$
restricts to an equivalence of categories,
$$
\Gc' \colon \OqConftri{n}\mod_{G\times \rev T^{\nsp n}} \longra \OqConftri{n}^{U_q(\g)}\mod_{\rev T^{\nsp n}}.
$$
\end{theorem}

\begin{proof}
Because the category $\Repq(G)$ is semisimple, the functor $\Gc'$ of taking $U_q(\g)$-invariants is exact. Thus to prove that it is an equivalence, it suffices to show that it is also conservative. Since the partial triangulation $\tri$ is not empty, there is an edge $e_{ij} \in \tri$ and hence an injective homomorpshim $\eta_{ij} \colon \Oq(G) \rightarrow \OqConftri{n}$ by \cref{Oq-morphism-lemma}. This in turn allows us to include the functor $\Gc'$ into the commutative diagram
$$
\xymatrix{
\OqConftri{n}\mod_{G\times \rev T^{\nsp n}} \ar[r]^{\Gc'} \ar[d]^{\Fc''} & \OqConftri{n}^{U_q(\g)}\mod_{\rev T^{\nsp n}} \ar[d]^{\Fc'} \\
\Oq(G)\mod_{G\times \rev T^{\nsp n}} \ar[r]^{\Gc''} &
\rev \Repq(T^{\nsp n})}
$$
where $\Fc'$ and $\Fc''$ are forgetful functors, and $\Gc''$ is again the functor of taking $U_q(\g)$-invariants. Now, it is enough to show that the composition $\Gc''\circ\Fc''$ is conservative. Being a forgetful functor, $\Fc''$ is obviously conservative, while $\Gc''$ is in fact an equivalence with inverse given by $V \mapsto \Oq(G)\otimes V$.
\end{proof}

\section{Charts and flips on $\Zc(\decS)$ via excision}

In this section we will compute with an arbitrary simple decorated surface $\decS$.  We will describe a special family of open subcategories of the category $\Zc(\decS)$, which we regard as open subvarieties -- charts -- on the quantum decorated character stack of $\decS$.  We have one such chart for each isotopy class of triangulation of the  $\decS$, and each chart comes equipped with a distinguished compact projective generator $\Dist_\tri$, and an identification as the category of modules for a quantum torus, resembling the notion of a cluster chart.

The basic mechanism to obtain these charts is to start from the monadic description of $\Zc(\DD_3)$ obtained in the previous section, and then apply excision with respect to gluing triangles along a common digon $\DD_2$.  The charts so constructed admit pairwise transition maps resembling those in the quantum $\Ac$-variety variety of Fock--Goncharov -- we call the whole structure a \emph{quantum $\widehat{\Ac}$-variety}, and we compare it to the quantum cluster $\Ac$-variety of Fock--Goncharov.  Taking $T$-invariants yields a version of the quantum cluster ensemble map, giving a reconstruction of cluster $\Xc$-varieties.

\subsection{Opening $T$-gates}

If $\decS$ is a decorated surface with a collection $\Gc$ of gates, recall that the category $\Zc(\decS)$ inherits the structure of a module category for the braided tensor category $\Repq(\Gc) = \Repq(G^m \boxtimes \rev T^n)$, where $m$ and $n$ denote the number of $G$- and $T$-gates in $\Gc$. Given some object $X\in\decS$ -- typically in examples $X$ will be some $\Dist_\tri$ -- different choices of $\Gc$ clearly produce distinct internal endomorphism algebras $\iEnd_\Gc(X)$. 

However, the category $\iEnd_\Gc(X)\mod_\Gc$ of $\iEnd_\Gc(X)$-modules in $\Repq(\Gc)$ does not change when we open gates.

\begin{lemma}\label{lem:opening}
Let $\Gc$ be some configuration of $T$-gates, and let $\Gc'$ be obtained from $\Gc$ by replacing any $T$-gate with a pair of adjacent gates. Then the functor of taking invariants at either new gate yields an equivalence of categories:
$$
\iEnd_\Gc(\Dist_\decS)\mod_{\Gc} \simeq \iEnd_{\Gc'}(\Dist_\decS)\mod_{\Gc'}.
$$
\end{lemma}

\begin{proof}
Indeed, in this case the $\Gc'$-module category $\iEnd_\Gc'(\Dist_\decS)\mod_{\Repq(\Gc')}$ is simply the pullback of $\iEnd_{\Gc}(\Dist_\decS)\mod_{\Repq(\Gc)}$ along the tensor functor $\Repq(\overline{T})\boxtimes \Repq(\overline{T})\to\Repq(\overline{T})$.
\end{proof}

The specification of $T$-gates, as well as the presentation for the resulting algebras of internal endomorphisms will be encoded combinatorially in the data of a \defterm{fencing graph} associated to the triangulation.

\begin{defn}
\label{def:long-short-edges}
Let $\tri$ be a triangulation of a simple decorated surface. Then the boundary of the $G$-region of every digon or triangle in $\tri$ consists of interlacing \defterm{short} and \defterm{long}, where the former lie in $\decS_T$, the latter in $\decS_G$.
\end{defn}

Thus each triangle or digon gives rise to a hexagon or rectangle respectively with alternating long and short edges.

\begin{defn}
\label{def:fencing-graph}
The \defterm{fencing graph} $\Gamma_\tri$ has one vertex for each corner of a hexagon appearing in $\tri$.  The edge set consists of a directed arrow connecting opposite vertices of each short edge of either a hexagon or a rectangle, with the direction being compatible with the clockwise orientation within $T$-region, and a bivalent edge connecting the opposite pair of vertices of each long edge.  
\end{defn}

An example of a fencing graph for the punctured triangle $\DD_3^\circ$ is illustrated in~\cref{fig:peace}. Clearly, there are $6h$ vertices in $\Gamma_\tri$, where $h$ is the number of triangles in $\tri$.  We associate to the fencing graph $\Gamma_{\tri}$ a collection of $T$-gates in $\decS$ by declaring that each vertex of $\Gamma_\tri$ lies in a unique $T$-gate. Thus a notched triangulation $\trib$ determines in each $T$-region a \defterm{distinguished vertex} of the fencing graph $\Gamma_{\tri}$, for convenience we assume that as we traverse the boundary of any $T$-disk in a counterclockwise direction, we encounter the distinguished vertex first. It also induces a total ordering on the edge-ends incident to vertices at each $T$-region: we number edge-ends as we traverse the $T$-region clockwise, starting with the one incident to the distinguished vertex.

This in turn determines a \defterm{notched ideal triangulation}, where we identif all vertices in each $T$-region of the fencing graph, collapse parallel long edges, and collapse all short edges but one whose head is the distinguished vertex. The remaining short edge becomes a loop, which we call a \defterm{notch}; this defines a total ordering on the edge-ends of the notched ideal triangulation, see~\cref{fig:peace}.

We will be considering the following configurations of gates associated to a notched triangulation $\trib$:
\begin{itemize}
    \item $\Tc_\Gamma$ has a single gate located at each vertex of the fencing graph.
    \item $\Tc_P$ contains only the distinguished gate at each puncture.
    \item $\Tc_{P^c}$ contains only the non-distguished gates at each puncture.
    \item $\Tc_M$ contains only the gates at each marked point.
\end{itemize}

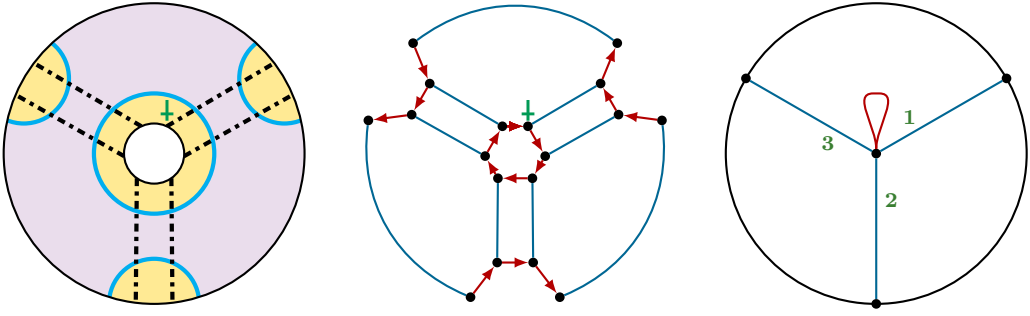
\begin{figure}[t]
\begin{tikzpicture}[every node/.style={inner sep=0.5, thick, circle, minimum size=0.1cm}, x=0.5cm, y=0.5cm, thick, scale={0.8}]

\begin{scope}[shift={(-12,0)}]
\def\Greg{(0,0) circle (5)};

\def\Tone{(-90:5) circle (1.5)};
\def\Ttwo{(30:5) circle (1.5)};
\def\Tthree{(150:5) circle (1.5)};
\def\Tcirc{(0,0) circle (2)}

\fill[Orchid!20] \Greg;
\begin{scope}
	\clip\Greg;
	\fill[Goldenrod!60] \Tone \Ttwo \Tthree \Tcirc;
	\draw[ultra thick, cyan] \Tone \Ttwo \Tthree \Tcirc;
	\draw[ultra thick, dashdotted] (-5:1) to (23:5);
	\draw[ultra thick, dashdotted] (65:1) to (37:5);
	\draw[ultra thick, dashdotted] (-55:1) to (-83:5);
	\draw[ultra thick, dashdotted] (-125:1) to (-97:5);
	\draw[ultra thick, dashdotted] (-175:1) to (157:5);
	\draw[ultra thick, dashdotted] (115:1) to (143:5);
\end{scope}

\draw[fill=white] (0,0) circle (1);

\draw \Greg;

\node[ForestGreen, shift={(65:1)}] at (-0.1,0.3) {$\rotatebox[origin=c]{180}{\Cross}$};

\end{scope}

\begin{scope}
\path[name path = Greg] (0,0) circle (5);

\path[name path = T1] (150:5) circle (1.5);
\path[name path = T2] (-90:5) circle (1.5);
\path[name path = T3] (30:5) circle (1.5);

\path[name path = Tcirc] (0,0) circle (1);

\path[name path = l1] (115:1) to (143:5);
\path[name path = l2] (-175:1) to (157:5);
\path[name path = l3] (-125:1) to (-97:5);
\path[name path = l4] (-55:1) to (-83:5);
\path[name path = l5] (-5:1) to (23:5);
\path[name path = l6] (65:1) to (37:5);

\path[name intersections={of=Greg and T1, by={A1,A2}}];
\path[name intersections={of=Greg and T2, by={A3,A4}}];
\path[name intersections={of=Greg and T3, by={A6,A5}}];

\path[name intersections={of=l1 and T1, by=P1}];
\path[name intersections={of=l2 and T1, by=P2}];
\path[name intersections={of=l3 and T2, by=P3}];
\path[name intersections={of=l4 and T2, by=P4}];
\path[name intersections={of=l5 and T3, by=P5}];
\path[name intersections={of=l6 and T3, by=P6}];

\path[name intersections={of=l1 and Tcirc, by=B1}];
\path[name intersections={of=l2 and Tcirc, by=B2}];
\path[name intersections={of=l3 and Tcirc, by=B3}];
\path[name intersections={of=l4 and Tcirc, by=B4}];
\path[name intersections={of=l5 and Tcirc, by=B5}];
\path[name intersections={of=l6 and Tcirc, by=B6}];

\node[ForestGreen, shift={(65:1)}] at (-0.1,0.3) {$\rotatebox[origin=c]{180}{\Cross}$};

\foreach \i in {1,...,6}
{
	\node[draw, fill] (a\i) at (A\i) {};
	\node[draw, fill] (p\i) at (P\i) {};
	\node[draw, fill] (b\i) at (B\i) {};
	\draw[-, MidnightBlue] (b\i) to (p\i);
}

\foreach \i in {1,...,5}
{
	\pgfmathparse{int(\i+1)};
	\edef\j{\pgfmathresult};
	\draw[->, red!70!black] (b\j) to (b\i);
	\draw[->, red!70!black] (b1) to (b6);
}

\foreach \i in {1,...,3}
{
	\pgfmathparse{int(2*\i-1)};
	\edef\j{\pgfmathresult};
	\pgfmathparse{int(2*\i)};
	\edef\k{\pgfmathresult};
	\draw[->, red!70!black] (a\j) to (p\j);
	\draw[->, red!70!black] (p\j) to (p\k);
	\draw[->, red!70!black] (p\k) to (a\k);
}

\draw[-, MidnightBlue] (a2) to [out=-97, in=157] (a3);
\draw[-, MidnightBlue] (a4) to [out=23, in=-83] (a5);
\draw[-, MidnightBlue] (a6) to [out=143, in=37] (a1);
\end{scope}

\begin{scope}[shift = {(12,0)}]

	\draw (0,0) circle (5);
	\node[draw, fill] (X) at (-90:5) {};
	\node[draw, fill] (Y) at (30:5) {};
	\node[draw, fill] (Z) at (150:5) {};
	\node[draw, fill] (C) at (0,0) {};
    \draw[MidnightBlue,postaction={decorate,decoration={markings,mark=at position 0.7 with
        {\node[OliveGreen] at (0,-0.4) {\scriptsize $\boldsymbol{2}$};}}}] (X) to (C);
    \draw[MidnightBlue,postaction={decorate,decoration={markings,mark=at position 0.7 with
        {\node[OliveGreen] at (0,-0.4) {\scriptsize $\boldsymbol{1}$};}}}] (Y) to (C);
    \draw[MidnightBlue,postaction={decorate,decoration={markings,mark=at position 0.7 with
        {\node[OliveGreen] at (0,-0.4) {\scriptsize $\boldsymbol{3}$};}}}] (Z) to (C);

	\draw[red!70!black] (C) to [out=90, in=180] (-0.1,2) to (0.1,2) to [out=0, in=90] (C);
	

\end{scope}
\end{tikzpicture}
\caption{Presentation of a punctured triangle $\DD_3^\circ$ as in~\cref{eqn:gluing} and \cref{eqn:gluing2}, the corresponding fencing graph with a distinguished vertex, and the notched ideal triangulation with the total numbering on the edge-ends adjacent to the puncture. The location of the distinguished gate in $\DD_3^\circ$, and the distinguished vertex of the fencing graph, are marked with \protect\rotatebox[origin=c]{180}{$\Cross$}.}
\label{fig:peace}
\end{figure}


Hence, we have a decomposition
\[
    \Tc_\Gamma = \Tc_{P^c} \bt \Tc_{P}\bt \Tc_{M}
\]
\begin{remark} Let us discuss the rationale for introducing additional $T$-gates. We will see below that opening gates produces quantum tori $\widetilde\zeta(\trib)$ with more generators than the tori $\zeta(\trib)$, in which we are ultimately interested. However, the larger tori $\widetilde\zeta(\trib)$ have simpler commutation relations which are expressed more locally in the triangulation and hence easier to compute, as in Propositions~\ref{basic-hex-rels} and~\ref{basic-sq-rels} below.  So we compute with the algebra $\widetilde\zeta(\trib)$ while performing excision, and close the additional gates only at the end.
\end{remark}

\subsection{Charts on $\Zc(\DD_2)$ and $\Zc(\DD_3)$}

Let us first consider the two cases when $\decS$ is a triangle or digon.  Recall from \cref{triangle-open-subcat} the open subcategories $\Zc(\underline{\DD_3}) \subset \Zc(\DD_3)$ and $\Zc(\underline{\DD_3}) \subset \Zc(\DD_3)$ defined, respectively, by the condition that elements of $S_{\underline{\DD_2}}$ and $S_{\underline{\DD_3}}$ act invertibly.  We consider the restriction of the distinguished object to each subcategory, and define
\begin{align*}
\zeta(\underline{\DD_2}) &:= \iEnd_{T^2}(\Dist_{\underline{\DD_2}}) \cong \Oq\hr{\Conf_2[\underline{\DD_2}^{-1}]}^{U_q(\g)}, \\
\zeta(\underline{\DD_3}) &:= \iEnd_{T^3}(\Dist_{\underline{\DD_3}}) \cong \Oq\hr{\Conf_3[\underline{\DD_3}^{-1}]}^{U_q(\g)},
\end{align*}
where the $T$-actions are induced by introducing one $T$-gate in each $T$-region.  By \cref{thm-inv}, we have equivalences of categories
$$
\Zc(\underline{\DD_2})\simeq \zeta(\underline{\DD_2})\mod_{\rev T^2},\quad
\Zc(\underline{\DD_3})\simeq \zeta(\underline{\DD_3})\mod_{\rev T^3}.
$$

It is evident that we have $\zeta(\underline{\DD_2})\simeq \mathbb{C}(q)\ha{D^{\pm1}}$, and the following Lemma shows that the algebra $\zeta(\underline{\DD_3})$ is also a quantum torus.

\begin{lemma}
The algebra $\zeta(\underline{\DD_3})$ is generated by the elements $D_{ij}^{\pm1}$ for $1 \le i < j \le 3$ given in the \cref{defn-Z}.
\end{lemma}

\begin{proof}
The algebra $\Oq(\Conf_3)^{U_q(\g)}$ is the direct sum of its weight subspaces with respect to $U_q(\mathfrak{t})^{\otimes 3}$:
$$
\Oq(\Conf_3)^{U_q(\g)} = \bigoplus_{l,m,n\geq0} \hr{\Oq(\Conf_3)^{U_q(\g)}}^{l,m,n},
$$
and each of these subspaces is at either 1-dimensional if $l+m+n$ is even and there exists a Euclidean triangle with side lengths $l,m,n$, or zero-dimensional otherwise.
But it is straightforward to see that the submodule of $\mathbb{Z}^3$ generated by the weights $(1,1,0),(1,0,1),(0,1,1)$ of the $D_{ij}$ is exactly the lattice of all vectors with even coordinate sum. Thus for each triple $(l,m,n)$ such that the corresponding subspace of $\Oq(\Conf_3)^{U_q(\g)}$ is non-zero, there exists a triple of integers $n_1,n_2,n_3 \in \Z$ such that
$$
D_{12}^{n_3} D_{13}^{n_2} D_{23}^{n_1} \in \hr{\Oq(\Conf_3)^{U_q(\g)}}^{\alpha,\beta,\gamma}.
$$
\end{proof}

Opening an additional gate at each $T$-region equips $\Zc(\DD_2)$ and $\Zc(\DD_3)$ with $T^4$- and $T^6$-actions, respectively, over which the restricted distinguished objects are still relative generators. Thus writing
$$
\widetilde\zeta(\underline{\DD_2}) := \iEnd_{\rev{T}^{\nsp 4}}\hr{\Dist_{\underline{\DD_2}}},\quad
\widetilde\zeta(\underline{\DD_3}) := \iEnd_{\rev{T}^{\nsp 6}}\hr{\Dist_{\underline{\DD_3}}},
$$
we obtain equivalences of categories
$$
\Zc(\underline{\DD_2})\simeq \widetilde{\zeta}(\underline{\DD_2})\mod_{T^4},\quad
\Zc(\underline{\DD_3})\simeq \widetilde{\zeta}(\underline{\DD_3})\mod_{T^6}.
$$

\begin{lemma}\label{basic-sq-rels}
The algebra $\widetilde\zeta(\underline{\DD_2})$ is generated by elements $A_1,A_2$ in correspondence with the long edges of the fencing graph for $\underline{\DD_2}$, as well as generators $a_1,a_2$ in correspondence with its short edges. These generators are subject to the relations
$$
[A_1,A_2]=[a_1,a_2]=0, \qquad a_i A_j = q^{\frac{1}{2}} A_ja_i,
\qquad\text{and}\qquad
a_1A_2 = a_2A_1.
$$
\end{lemma}

\begin{proof}
The Lemma is a special case of \cref{ex:gate-opening}, where we open one additional gate in each $T$-region. Introducing the weights $\eps_1=(1,0),\eps_2=(0,1)$ for the $\Repq(T^2)$--action on $\zeta(\underline{\DD}_2)$ and setting $\mathbb{C}_{\lambda,\mu} = \mathbb{C}_{\lambda}\boxtimes\mathbb{C}_{\mu}\in\Rep_q(T^2)\boxtimes \Rep_q(T^2)$, we take the generators $A_i,a_i$ to be
\begin{align*}
A_1 &= D \ot \chi_{\eps_1,\eps_2}, \qquad a_1 = 1 \ot \chi_{\eps_1,-\eps_1}, \\
A_2 &= D \ot \chi_{\eps_2,\eps_1}, \qquad a_2 = 1 \ot \chi_{\eps_2,-\eps_2}.
\end{align*}
The asserted relations then follow from the general multiplication rule~\eqref{induction-mult-rule}.
\end{proof}
Similarly, in the case of $\DD_3$ we have
\begin{lemma}
\label{basic-hex-rels}
The algebra $\widetilde\zeta(\underline{\DD_3})$ is a quantum torus with generators  $A_{12},A_{13},A_{23}$ in correspondence with the long edges of the fencing graph for $\underline{\DD_3}$, as well as generators $a_1,a_2,a_3$ in correspondence with its short edges. The relations among these generators are
$$
[A_{ij},A_{kl}]=[a_i,a_j]=0
\qquad\text{and}\qquad
a_k A_{ij} = q^{\frac{\delta_{ik}+\delta_{kj}}{2}} A_{ij} a_k.
$$
\end{lemma}

\begin{proof}
The proof is identical to that of \cref{basic-sq-rels}. The generators are given by
$$
A_{12} = D_{12} \ot \chi_{\eps_2,\eps_1}, \qquad
A_{23} = D_{23} \ot \chi_{\eps_3,\eps_2}, \qquad
A_{13} = D_{13} \ot \chi_{\eps_1,\eps_3},
$$
and
$$
a_k = 1 \ot \chi_{\eps_k,-\eps_k}, \qquad\text{for}\qquad k=1,2,3,
$$
with the asserted relations between them following immediately from formula~\eqref{induction-mult-rule}.
\end{proof}

The above computations were all for $G=\SL_2$.  However, we have the following parallel descriptions for $G=\PGL_2$.  The proof in this case is identical except that one requires all representations which appear to have exclusively even weights.
\begin{cor}
\label{cor:PGL2-D2-D3-gen}
We have the following:
\begin{itemize}
    \item Each algebra $\zeta_{\PGL_2}(\underline{\DD_2})$ is a quantum sub-torus of $\zeta_{\SL_2}(\underline{\DD_2})$ generated by $D^{\pm 2}$.
    \item Each algebra $\zeta_{\PGL_2}(\underline{\DD_3})$ is a quantum sub-torus of $\zeta_{\SL_2}(\underline{\DD_2})$ generated by
    \[
    D_{12}^{\pm 2}, D_{13}^{\pm 2}, D_{23}^{\pm 2},(D_{12}D_{13}D_{23})^{\pm 1}.
    \]
    \item Each algebra $\widetilde\zeta_{\PGL_2}(\underline{\DD_2})$ is a quantum sub-torus of $\zeta_{\SL_2}(\underline{\DD_2})$ generated by
    \[A_1^{\pm 2}, A_2^{\pm 2}, a_1^{\pm 2}, a_2^{\pm 2}.
    \]
    \item Each algebra $\widetilde\zeta_{\PGL_2}(\underline{\DD_3})$ is a quantum sub-torus of $\zeta_{\SL_2}(\underline{\DD_3})$ generated by
    \[
    A_{12}^{\pm 2}, A_{13}^{ \pm 2}, A_{23}^{\pm 2},(A_{12}A_{23}A_{13}a_1a_2a_3)^{\pm 1}, a_1^{\pm 2}, a_2^{\pm 2}, a_3^{\pm 2}.
    \].
\end{itemize}
\end{cor}

\begin{remark}
We remark that a monomial subalgebra of a quantum torus is necessarily a quantum torus, so that even though there are relations above the generators listed above one can always (non-canonically) choose free generators.  For instance, in the second list one may take $D_{12}^{2}, D_{13}^{2}, D_{12}D_{13}D_{23}$.
\end{remark}

\begin{remark}
Note that the orientation of $\DD_3$ defines a cyclic ordering on its $T$-regions. In what follows, we regard the generator $A_{ij}$ as being in correspondence with the long edge of the fencing graph $\Gamma_{\DD_3}$ connecting the $T$-regions labelled $i$ and $j$, and the generator $a_k$ as being in correspondence with the short edge of $\Gamma_{\DD_3}$ sitting on the boundary of the $k$-th $T$-region.
\end{remark}

\subsection{Charts on $\Zc(\decS)$ from triangulations}
\label{sec:general-charts}
Suppose now that $\tri$ is a triangulation of a general decorated surface $\decS$ consisting of $t$ triangles and $l$ digons. To organize computations, we shall fix a framing of $\decS$, in which all triangles are drawn in the plane of the blackboard with their long edges parallel to its $x$-axis. Each digon we represent as a ribbon connecting a pair of long edges of the triangles.  In particular, such a framing of $\decS$ fixes for each long edge of the fencing graph $\Gamma_{\tri}$ an ordering on the vertices at its endpoints. An example of such a framing in the case that $\decS$ is a punctured torus is drawn in~\cref{fig:torus}.  We wish to emphasize that this additional choice of a framing is \emph{immaterial} to the determination of the charts we will construct, but is merely \emph{convenient} for giving presentations, and formulas for mutations.  See~\cref{sec:modularity} for more details.

\begin{figure}[t]
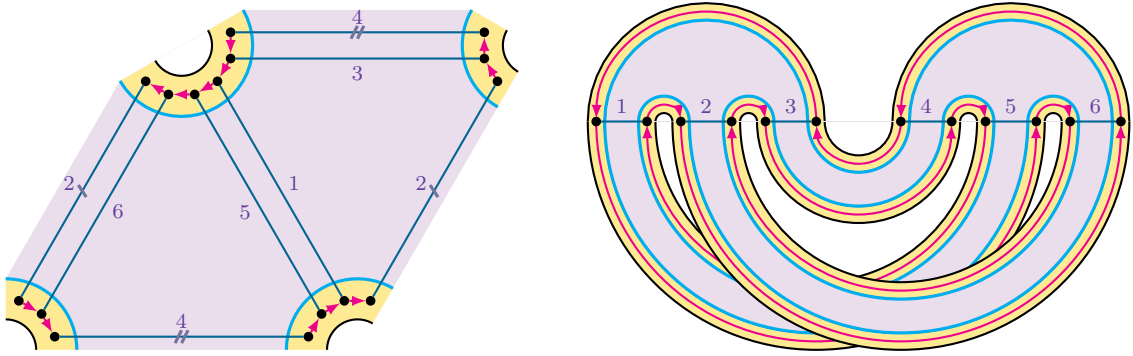

\subfile{punctured-torus.tex}
\caption{The once-punctured torus with its fencing graph overlaid, and its 2-framing as presented by a pair of jellyfish.}
\label{fig:torus}
\end{figure}

The triangulation $\tri$ is a presentation of $\decS$ as
\[
\decS = (\DD_3)^{\sqcup t} \bigsqcup_{(\DD_2)^{\sqcup 2\ell}} (\DD_2)^{\sqcup \ell},
\]
where the $2l$ digons over which we glue are obtained as small neighborhoods of the long edges of the fencing graph  $\Gamma_{\tri}$ as depicted in~\cref{fig:torus}.
Applying excision, we thus obtain an equivalence of categories
\begin{align}
\label{eqn:gluing2}
\Zc(\decS) \simeq \Zc(\DD_3)^{ t} \bigboxtimes_{\Zc(\DD_2)^{2 l}} \Zc(\DD_2)^{ l}
\end{align}
which allows us to define a open subcategory $\Zc(\tri)$ of $\Zc(\decS)$ by
\begin{align}
\label{triang-subcat}
{\Zc(\tri)} := {\Zc(\underline{\DD_3})}^{ t} \bigboxtimes_{{\Zc(\underline{\DD_2})}^{ 2l}} \Zc(\underline{\DD_2})^{ l}.
\end{align}
By construction, $\widetilde{\zeta}(\underline{\DD_2})^{\otimes l}$ is a commutative algebra object in the braided tensor category $\Tc_\Gamma$. 
Moreover, in view of Lemmas~\ref{basic-hex-rels} and~\ref{basic-sq-rels} we have equivalences
$$
\Zc(\underline{\DD_3})^{ t}\simeq \widetilde{\zeta}(\underline{\DD_3})^{\otimes t}\mod_{\Tc_\Gamma}, \qquad \Zc(\underline{\DD_2})^{ 2l}\simeq \widetilde{\zeta}(\underline{\DD_2})^{\otimes 2l}\mod_{\Tc_\Gamma},
$$
where $\widetilde{\zeta}(\underline{\DD_3})^{\otimes t}$ and $\widetilde{\zeta}(\underline{\DD_2})^{\otimes 2l}$ are both  $\widetilde{\zeta}(\underline{\DD_2})^{\otimes l}$--algebras. We thus find ourselves in the setting of \cref{lem:excision}, which we may apply to obtain the following description of the open subcategory corresponding to the triangulation $\tri$: 
$$
\Zc(\tri)\simeq \widetilde{\zeta}(\tri)\mod_{\Tc_\Gamma},
$$
where
$$
\widetilde{\zeta}(\tri):=\iEnd_{\Tf_\tri}(\Dist_\tri) \cong \widetilde{\zeta}(\underline{\DD_3})^{\otimes t}\otimes_{{\zeta}(\underline{\DD_2})^{\otimes 2l}}\widetilde{\zeta}(\underline{\DD_2})^{\otimes l}.
$$

The quantum torus $\widetilde{\zeta}(\tri)$ can be described explicitly with the help of the relations from Lemmas~\ref{basic-hex-rels} and~\ref{basic-sq-rels}. Let us write $E(\Gamma_\tri)$ for the set of edges of the fencing graph $\Gamma_\tri$, so that
$$
E(\Gamma_\tri) = E_l(\Gamma_\tri) \sqcup E_s(\Gamma_\tri),
$$
where $E_l(\Gamma_\tri)$ and $E_s(\Gamma_\tri)$ denote the subsets of long and short edges respectively. Recall that the union of all short edges coincides with the set of walls $\Cc$ between $G$ and $T$-regions. Let us now orient the short edges so that they travel around $T$-regions in a clockwise manner.  Given a short edge $e \in E_s(\Gamma_\tri)$, we define its \defterm{head} $h(e)$ and its \defterm{tail} $t(e)$, to be its endpoints so that $e$ is oriented from $t(e)$ to $h(e)$.  
\begin{prop}
\label{pre-glue-alg}
The algebra object $\widetilde{\zeta}(\tri)$
in the category $\Tc_\Gamma$ is a quantum torus with generators $A_\ell^{\pm1}$, $\ell \in E_l(\Gamma_\tri)$,  $a_e^{\pm1}$, $e \in E_s(\Gamma_\tri)$. These generators are subject to the following \defterm{fencing graph relations} which hold for all long edges $\ell, \ell'$ and short edges $e, e'$: 
\beq
\label{fencing-rels1}
A_\ell A_{\ell'}=A_{\ell'} A_{\ell}, \qquad a_1A_2 = a_2A_1,
\eeq
if $a_1, a_2, A_1, A_2$ are the four edges in a digon as in \cref{basic-sq-rels},
\begin{align*}
a_e A_\ell &= \begin{cases} q^{\frac12} A_\ell a_e & \text{if} \quad \ell \cap e \ne \emptyset, \\ A_\ell a_e & \text{otherwise};\end{cases}, \\
a_e a_{e'} &= \begin{cases} q^{\frac12} a_{e'} a_e & \text{if} \quad h(e)=t(e') \quad \text{and} \quad t(e) \ne h(e'), \\ a_{e'} a_e & \text{otherwise}.
\end{cases},
\end{align*}

The $\Tf_\tri$-equivariant structure is as follows:
\[
\wt_i(A_\ell) =
    \begin{cases}
        1 &\text{if} \quad i \in \ell, \\
        0 &\text{if} \quad i \notin \ell,
    \end{cases}
\qquad
\wt_i(a_e) =
    \begin{cases}
        1 &\text{if} \quad i = t(e), \\
        -1 &\text{if} \quad i = h(e), \\
        0 &\text{otherwise},
    \end{cases}
\]
where $\wt_i(x)\in \Lambda(T)$ indicates weight of an element $x \in \widetilde\zeta(\tri)$ with respect to the torus at the $i$-th vertex in $\Gamma_\tri$.
\end{prop}

\begin{cor}\label{cor:tildezetaPGL2}
The algebra object $\widetilde\zeta_{\PGL_2}(\tri)$ is the subalgebra of $\widetilde\zeta_{\SL_2}(\tri)$ generated by $(a_i)^{\pm 2}$ for each short edge, $A_\ell^{\pm 2}$ for each long edge, and $(A_{\ell_1}A_{\ell_2}A_{\ell_3}a_1a_2a_3)^{\pm 1}$ for each triangle.
\end{cor}

\subsection{Closing $T$-gates}
\label{gluing-pattern-subsec}

\begin{defn}
\label{A-hat-torus}
If $\trib$ is a notched triangulation of a decorated surface $\decS$, we define the quantum torus $\zeta(\trib)$ to be the subalgebra of $\Tc_{P^c}$-invariants in $\widetilde\zeta(\tri)$.
\end{defn}

By \cref{lem:opening}, the functor of taking $\Tc_{P^c}$ invariants is conservative. Hence we have an equivalence of categories,
$$
\zeta(\trib)\mod_{\Tc_{P\cup M}} \simeq \widetilde\zeta(\tri)\mod_{\Tc_\Gamma}.
$$

The generators of $\zeta(\trib)$ are labelled by the edges and the notch of the notched ideal triangulations. We shall now describe these generators and their relations, and express them in terms of generators of $\widetilde \zeta(\tri)$.  We will make heavy use of the Weyl ordering on the quantum torus $\widetilde\zeta(\tri)$.

\begin{defn}
\label{weyl-order}
Suppose that $\zeta$ is a quantum torus with a basis of generators $\{X_i \}$. Then the \defterm{Weyl ordering} with respect to the generating set $(X_i)$ is defined recursively by declaring $\nord{X_i}=X_i$ for all $i$, and if $Y$ and $Z$ are two monomial elements satisfying $YZ = q^{2r}ZY$,
$$
q^{-r}\!\!\nord{Y}\!\nord{Z} = \nord{YZ} = q^r\!\!\nord{Z}\!\nord{Y}.
$$
\end{defn}

Let $s$ be an edge in the notched triangulation $\trib$, and pick a long edge $l$ in the fencing graph corresponding to $s$. Then the endpoints of $l$ are vertices $w_1,w_2$ of the fencing graph contained in regions whose distinguished vertices are $v_1$ and $v_2$. Denote by $a(l_1)$ the set of short edges in the minimal path in the fencing graph starting at distinguished vertex $v_1$ and following the short edges clockwise to the vertex $w_1$. Similarly we define $a(l_2)$ to be the set of short edges in the minimal path in the fencing graph starting at the distinguished vertex $v_2$ and travelling clockwise to the vertex $w_2$. Then we define an element $Z_s$ of $\zeta(\trib)$ by
$$
Z_s := \nord{a(l_1)a(l_2)A_l}.
$$
It follows from the relations~\eqref{fencing-rels1} in \cref{pre-glue-alg} that the element $Z_s$ does not depend on the choice of long edge $l$ projecting to the edge $s$ in the underlying notched triangulation. 

Given a puncture $p \in \hatS$, let $v \in \Gamma_\tri$ be the corresponding distinguished vertex, and $\gamma_p$ be the path which connects $v$ to itself and winds once around the corresponding $T$-region. Then we define another element of $\zeta(\trib)$ as the Weyl-ordered product
$$
\alpha_p := \nord{\prod_{s \in \gamma_p} a_s},
$$
where the product of short edges is taken in clockwise order starting at any vertex in $\gamma_p$.
\
We now describe the weights and commutation relations between the edge variables $Z_s$. In order to do this, we will make use of the two edge-ends $s_1,s_2$ associated to each edge in a notched triangulation. Each edge-end $s_i$ has a a boundary vertex $a = \partial s_i$. Note that the weights of the generators $Z_s$ and $\alpha_p$ with respect to the $i$-th factor of $\mathcal{T}_{P\cup M}$ are as follows:
\beq
\label{wt_z}
\wt_i(Z_s) = \#\hc{\text{edge--ends of $s$ lying in the region $i$}}
\qquad\text{and}\qquad
\wt_i(\alpha_p) = 0.
\eeq

\begin{defn}
Let $s$ be an edge-end in a notched triangulation $\trib$, and $p = \partial s$ be the vertex at which $s$ terminates. We write $|s|$ for the number of $s$ in the total ordering on the set of edge-ends terminating at the vertex $p$ defined by the distinguished gate.
\end{defn}
For a pair of edge-ends $s_1$ and $s_2$ we now define
$$
\omega(s_1,s_2) =
\begin{cases}
-1 & \text{if $p := \partial s_1 = \partial s_2$ and $|s_1| < |s_2|$}, \\
1 & \text{if $p := \partial s_1 = \partial s_2$ and $|s_2| < |s_1|$}, \\
0 & \text{otherwise}.
\end{cases}
$$

Now let $s, t \in \trib$ be a pair of edges with edge-ends $s_1, s_2$ and $t_1,t_2$ respectively. We define $\omega(s,t)$ by the following formula:
$$
\omega(s,t) = \sum_{i,j=1}^2 \omega(s_i,t_j).
$$

Then from the presentation of $\widetilde\zeta(\tri)$ described in \cref{pre-glue-alg} we deduce:

\begin{theorem}
\label{thm:zeta-rel}
For any notching $\trib$ of a triangulation $\tri$, we have an equivalence of categories
\[
\Zc(\tri)\simeq \zeta(\trib)\mod_{\mathcal{T}_{P\cup M}}.
\]
The algebra $\zeta(\trib)$ is a quantum torus with generators $Z_s^{\pm1}$ labelled by the edges $s$ of the notched ideal triangulation $\trib$, and by an invertible generator $\alpha_p$ for each puncture $p$ of $\decS$. The relations between these generators are as follows:
\begin{itemize}
\item for a pair of edges $s, t \in \trib$ we have
$$
Z_{s} Z_{t} = q^{\frac{1}{2}\omega(s,t)} Z_{t} Z_{s};
$$
\item for any puncture $p$ and an edge $s$ with vertices $\partial s_1, \partial s_2$ we have
$$
\alpha_p Z_s = q^{\delta_{p,\partial s_1}+\delta_{p, \partial s_2}} Z_s \alpha_p.
$$
\end{itemize}

\end{theorem}

Recall now the subalgebra $\chi\hr{\tri}$ of $\zeta(\trib)$ obtained by closing the distinguished gate at each puncture:
$$
\chi\hr{\tri} = \zeta(\trib)^{\mathcal{T}_P}.
$$
\begin{cor}
Suppose that either $G=\SL_2$ and we have at least one marked point, or that $G=\PGL_2$. Then the functor of taking $\Tc_P$ invariants defines an equivalence,
\[
\Zc(\tri) \simeq \chi(\tri)\mod_{\Tc_M}.
\]
\end{cor}

\begin{proof}

First we treat the case $G=\SL_2$.  Since the functor of taking $\Tc_P$-invariants is exact, we need only show that it is conservative, i.e. that no object is sent to zero, equivalently that for any $\Lambda(\Tc_{P\cup M})$-graded $\zeta(\trib)$-module, the degree-zero subspace with respect to the $\Lambda(\Tc_P)$-grading is non-zero.  For this, simply note that the invertible elements $D_{ij}$ emanating from any puncture have weight $1$ at the puncture.  Hence, taking ratios $D_{ij}D_{jk}^{-1}$ for any triangle with vertices $ijk$, where $i$ is the puncture, we can shift the degree to vertex $k$.  In this way, we can move any non-vector to another non-zero vector with degree supported exclusively over the non-empty set of marked points.

In the case of $\PGL_2$, we do not need to assume existence of a marked point.  In this case, we have generators $D_{ij}^2$ and $D_{ij}D_{jk}D_{kl}$, and by definition we consider only modules of even total degree at each gate.  For any triangle with vertices $ijk$, the invertible element $(D_{ij}D_{jk}D_{ki})^2/D_{jk}^2$ has degree $2$ at vertex $i$ and degree zero at all other vertices, so it can be used directly to reduce the degree at any vertex $i$, in particular at the puncture.
\end{proof}

\begin{cor}
Let $G=\PGL_2$, and let $\eta(\tri)=\chi(\tri)^{\Tc_M}$.  Then we have an equivalence of categories,
\[
\Zc_{\PGL_2}(\tri) \simeq \eta(\tri)\mod.
\]
\end{cor}

\begin{remark}\label{rmk:SL2-eta} For $\SL_2$, the functor of taking $T$-invariants at the final gate is not an equivalence for obvious parity reasons (owing to existence of a center of $\SL_2$).  However, one can show that instead of $\Dist_\tri$ being a generator of each chart, we have instead that the sum $\Dist_\tri \oplus \Dist_\tri[1]$ is a generator, where the latter denotes a parity shift applied to $\Dist_\tri$ is a generator, so that one has an equivalence:
\[
\Zc_{\PGL_2}(\tri) \simeq \hr{\eta_{\SL_2}(\tri)\oplus \eta_{\SL_2}(\tri)}\mod.
\]
\end{remark}
\subsection{Comparison of $\PGL_2$ and $\SL_2$ charts}\label{sec:PGL2-comparison}
Let us consider the commuting diagram of inclusions of quantum tori, where below each one we recall which $T$-gates remain open when defining that algebra.  Hence $M$ bijects with the marked points, $P$ bijects with punctures, and $\widetilde{P}$ bijects with the set of incidences $(v,T)$, where $v$ is a vertex of $\tri$, $T$  is a triangle of $\tri$, and $v$ is a vertex of $T$.

$$
\begin{tikzcd}[column sep = large]
\eta_{\SL_2}(\tri) \arrow[hook]{r} &\chi_{\SL_2}(\tri) \arrow[hook]{r} & \zeta_{\SL_2}(\trib)\arrow[hook]{r} & \widetilde\zeta_{\SL_2}(\tri)\\ \\
\eta_{\PGL_2}(\tri)\arrow[hook]{uu} \arrow[hook]{r} &\chi_{\PGL_2}(\tri) \arrow[hook]{uu} \arrow[hook]{r} & \zeta_{\PGL_2}(\trib) \arrow[hook]{uu} \arrow[hook]{r} &  \widetilde\zeta_{\PGL_2}(\tri) \arrow[hook]{uu}\\[-12pt]
\emptyset & M & M\cup P & M\cup \widetilde{P}
\end{tikzcd}
$$

Let $H_1(S,X)$ denote the homology of $S$ relative to $X$, with $\mathbb{Z}/2\mathbb{Z}$ coefficients.

\begin{prop}
\label{prop:cover}
We have:
\begin{enumerate}
    \item The algebra $\eta_{\SL_2}(\tri)$ is a free module over $\eta_{\PGL_2}(\tri)$ of rank $|H_1(S)|$.
    \item The algebra $\chi_{\SL_2}(\tri)$ is a free module over $\chi_{\PGL_2}(\tri)$ of rank $|H_1(S,M)|$.
    \item The algebra $\zeta_{\SL_2}(\trib)$ is a free module over $\eta_{\PGL_2}(\tri)$ of rank $|H_1(S,M\cup P)|$.
    \item The algebra $\widetilde\zeta_{\SL_2}(\trib)$ is a free module over $\widetilde\zeta_{\PGL_2}(\tri)$ of rank $|H_1(S,M\cup\widetilde{P})|$.
\end{enumerate}
Moreover in each case, a free basis can be taken consisting of $M_1^{\eps_1} \cdots M_N^{\eps_N}$, for some monomials $M_1,\ldots, M_N$, with $\eps_i\in\{0,1\}$, and such that $M_i^2$ lies in the $\SL_2$ quantum torus.
\end{prop}
\begin{proof}
We prove (4) first; each of the other statements will follow by computing the appropriate subalgebras of invariants.  First, an elementary computation with the long exact sequence in relative homology yields $|H_1(S,M\cup\widetilde{P})| = 2^N$, where $N=E-F$ denotes the number of (long and short) edges in the fencing graph, and $F$ denotes the number of (hexagonal or rectangular) faces.  Here, we regard the fencing graph as giving a simplicial decomposition of the surface. Then the rank asserted in formula (4) and the required basis in that case follows immediately from \cref{cor:tildezetaPGL2}: a basis may be given by monomials of degree either one or zero with respect to each (long and short) edge, modulo an equivalence relation identifying any edge with the product of all other edges on the same face.

Each time we close a gate by taking $T$-invariants at some vertex, the parity requirement there becomes vacuous, and elements of our basis become redundant.  It is then elementary to verify each remaining formula (3), (2), and (1).
\end{proof}

\subsection{Comparison of $\Xc^q_{\tri}$ and $\chi_{\PGL_2}(\tri)$}

Let $s$ be an edge in the triangulation $\tri$, and $s_1$, $s_2$ be its edge ends incident to vertices $v_1$ and $v_2$. Denote by $s_i^+$ and $s_i^-$ the edge-ends incident to $v_i$ which immediately follow and precede $s_i$ in the clockwise order. We then define an element $\widetilde X_s \in \chi_{\SL_2}(\tri)$ as follows:
\begin{itemize}
    \item if $s$ is not a boundary edge and is not the internal edge of a self-folded triangle,
    $$
    \widetilde X_s = \nord{Z_{s_1^-}Z_{s_1^+}^{-1}Z_{s_2^-}Z_{s_2^+}^{-1}}
    $$
    \item if $s$ is a boundary edge, and $v_2$ follows $v_1$ as we traverse the corresponding boundary component in a counterclockwise direction,
    $$
    \widetilde X_s = \nord{Z_sZ_{s_1^-}Z_{s_2^+}^{-1}}
    $$
    \item if $s$ is the internal edge of a self-folded triangle with the outer edge $s'$, we set
    $$
    \widetilde X_s = \widetilde X_{s'}
    $$
\end{itemize}

Given a puncture $p$ and an edge $s$ in $\trib$, we say that $s$ is \defterm{adjacent to the notch} at $p$ if an edge-end of $s$ is either the first or the last one at $v_i$. We say that $s$ is \defterm{opposing the notch} at $p$ if there exist edges $s'$ and $s''$, both adjacent to notch at $p$, such that $\hc{s,s',s''}$ form a triangle in $\trib$, possibly self-folded or with some of the vertices coinciding. We then set
$$
\epsilon_p(s) =
\begin{cases}
1, &\text{if $s$ is adjacent to the notch at $p$}, \\
-1, &\text{if $s$ is opposing the notch at $p$}, \\
0, &\text{otherwise}.
\end{cases}
$$
If $s$ is not the internal edge of a self-folded triangle, we now define
$$
\alpha_s = \prod_{p \in P} \alpha_p^{\epsilon_p(s)},
$$
whereas if $s$ is the internal edge of a self-folded triangle with the internal puncture $p$ and the outer edge $s'$, we set
$$
\alpha_s = \alpha_{s'}\alpha_p^2.
$$

Now, let $\tri$ be the ideal triangulation underlying $\trib$. Denote by $\Xc^q_\tri$ the quantum chart of the cluster Poisson variety defined by~$\tri$, see~\cite{GS19}, and set $X_s$ to be the quantum cluster $\Xc$-variable labelled by the edge $s$.  We have:

\begin{theorem}
We have a well-defined homomorphism,
\begin{align*}
    \iota_{\trib} \colon \Xc^q_\tri &\longrightarrow \zeta_{\SL_2}(\trib),\\
    X_s &\longmapsto \nord{\widetilde X_s \alpha_s}.
\end{align*} 

Moreover, $\iota_{\trib}$ defines an isomorphism $\Xc^q_\tri\cong \chi_{\PGL_2}(\tri). $
\end{theorem}

\begin{proof}
The map $\iota_{\trib}$ is a minor modification of the (quantum) cluster ensemble map, and can be shown to be well-defined by a straightforward calculation with the powers in the $q$-commutation relations.  It is also straightforward to see that $\iota_{\trib}(\Xc^q_\tri) \subset \chi_{\PGL_2}(\tri)$.

Let us now show that $\iota_{\trib}$ is surjective onto $\chi_{\PGL_2}(\tri)$. In what follows we will abuse notation and denote $\iota_{\trib}(X_s)$ by $X_s$. Given a puncture $p$ and a triangle $\Delta$ in $\trib$, we set
$$
\upsilon_p(\Delta) =
\begin{cases}
1, &\text{if two sides of $\Delta$ are adjacent to the notch at $p$}, \\
0, &\text{otherwise}.
\end{cases}
$$
Then, for a triangle $\Delta$ with the sides $s_1,s_2,s_3$ (two of which coincide if $\Delta$ is self-folded) we set
$$
M_\Delta = \nord{Z_{s_1}Z_{s_2}Z_{s_3}\prod_{p\in P}\alpha_p^{\upsilon_p(\Delta)}}.
$$
From~\cref{cor:PGL2-D2-D3-gen} and~\cref{prop:cover}, we see that $\chi_{\PGL_2}(\tri)$ is generated by $\alpha_p^{\pm2}$ for all punctures $p$, and the products
$$
\prod_s Z_s^{2n_s} \prod_\Delta M_\Delta^{n_\Delta}, \quad n_s, n_\Delta \in \Z,
$$
which have zero weight with respect to the torus action at each puncture. At the same time, for any pair of triangles $\Delta,\Delta'$ which share a common edge $s$, we have
$$
M_{\Delta'} = X_s M_{\Delta} \cdot \prod_{s'} Z_{s'}^{2n_{s'}} \prod_p \alpha_p^{2n_p},
$$
for some $n_{s'}, n_p \in \Z$. Therefore, for any fixed triangle $\Delta$ in $\tri$, the quantum torus $\chi_{\PGL_2}(\tri)$ is generated by $X_s$ for all edges $s$, $\alpha_p^{\pm2}$ for all punctures, and the products 
$$
M_\Delta^{n_\Delta} \prod_s Z_s^{2n_s}, \quad n_s, n_\Delta \in \Z,
$$
invariant under the $T$-action at punctures. Note however, that the $T$-invariance condition forces $n_\Delta$ to be even, and $M_\Delta^2$ itself is a product of elements of the form $Z_s^2$ and $\alpha_p^2$.

Now, let $\gamma = (s_1, \dots, s_{2n})$ be an \defterm{even path in $\tri$}, that is an ordered collection of edges, such that $s_i$ and $s_{i+1}$ share a common vertex $v_i$ for any $i \in \Z/2n\Z$. Define a product
$$
M_\gamma = \prod_{i=1}^{2n} {Z_{s_i}^{(-1)^i}}.
$$
Then the above discussion yields that $\chi_{\PGL_2}(\tri)$ is generated by $X_s$ for all edges $s$, $\alpha_p^{\pm2}$ for all punctures $p$, and monomials $M_\gamma^{\pm2}$ for all even paths $\gamma$. If $p$ is not the internal puncture of a self-folded triangle, one can check that
$$
\alpha_p^2 = \nord{\prod_{p \in \partial s} X_s},
$$
where the product is taken over all edges $s$ incident to $p$. Otherwise, if $p$ is the internal puncture of the self-folded triangle with the inner edge $s$ and the outer one $s'$, we get
$$
\alpha_p^2 = X_sX_{s'}^{-1}.
$$
Similarly, we have
$$
M_\gamma^2 = \prod_{i=1}^{2n} \prod_{e_i < e < e_{i+1}} X_e^{(-1)^i},
$$
where the product is taken over all edges $e$ incident to the vertex $v_i$ and sitting between $e_i$ and $e_{i+1}$ with respect to the clockwise cyclic order at $v_i$. This completes the proof of the surjectivity of $\iota_{\trib}$.

Finally, let us prove that $\iota_{\trib}$ is injective.  We note first that as a monomial map between quantum tori, the kernel is flat in $q$.  Hence it suffices to show at $q=1$ that $\Spec(\chi_{\PGL_2}(\tri))$ is an algebraic torus of dimension $|E|$, where $|E|$ is the number of edges in $\tri$, and is the dimension of $\Spec(\Xc^q_\tri)$. Indeed, we know that $\Spec(\zeta_{\SL_2}(\trib))$ has dimension $|P|+|E|$, where $|P|$ is the number of punctures. Since the action of $T$ on $\Spec(\zeta_{\SL_2}(\trib))$ is free at every puncture, the dimension of $\Spec(\chi_{\SL_2}(\tri))$ equals $|E|$, and so does the dimension of $\Spec(\chi_{\PGL_2}(\tri))$.
\end{proof}

\subsection{The flip on $\Zc(\DD_4)$}
The study of flips between triangulations of $\decS$ starts with the quadrilateral $\DD_4$ and its two triangulations $\tri$ and $\tri'$.
An arbitrary flip can be understood in terms of the one between $\tri,\tri'$, using excision to isolate the pair of adjacent triangles
where the flip takes place. In what follows, we retain our convention from \cref{sec:disk_charts} of enumerating the $T$--regions $\{1,2,3,4\}$ of $\DD_4$ in counter-clockwise order. 

Recall from \cref{disk-computation} that the category $\Zc(\DD_4)$ is equivalent to the orthogonal complement in $\Oq(\Conf_{4})\mod_{G\times T^4}$ of the localizing subcategory of torsion modules. By \cref{triangle-open-subcat}, we have full reflective subcategories $\Zc(\tri),\Zc(\tri')$ of $\Zc(\DD_4)$, each equivalent to the full subcategory of $\Oq(\Conf_{4})$--modules on which all elements $\{\Delta_{e}\}$ associated to the edges of the corresponding triangulation act invertibly. By the same argument used to prove \cref{triangle-open-subcat}, there is another open reflective subcategory 
$$
\Zc(\tri,\tri')\subset\Zc(\mathbb{D}_4)\into\Oq(\Conf_{4})\mod_{G\times T^4},
$$ 
equivalent to the full subcategory of $\Oq(\Conf_{4})$--modules on which all elements $D_{ij}$ with $1 \le i < j \le 4$ act invertibly.

To describe concretely the category $\Zc(\tri,\tri')$ we open four gates in each $T$-region, thereby endowing $\DD_4$ with a $T^{16}$-action.  We denote by $\Dist_{\tri,\tri'}$ the restriction of the distinguished object, and we set
\[\widetilde\zeta(\tri,\tri') = \iEnd_{T^{16}}(\Dist_{\tri,\tri'}),\]
giving an equivalence,
\[\Zc(\tri,\tri') \simeq \widetilde\zeta(\tri,\tri')\mod_{T^{16}}
\]
and the following:
\begin{prop}
\label{D4-flip-functor}
The transition functor between charts $\Zc(\tri)$ and $\Zc(\tri')$ is given by:
\begin{align*}
\mu_{\tri,\tri'} \colon \Zc(\tri)&\longrightarrow \Zc(\tri')\\
m &\longmapsto \zeta(\trib,\trib')\underset{\zeta(\trib)}{\otimes}m. 
\end{align*}
\end{prop}

This transition functor can be described completely explicitly as follows. Consider the ideal tetrahedron whose four triangular faces are partitioned into two copies of $\DD_4$. On the boundary of this tetrahedron we have a fencing graph $\Gamma_{\tri,\tri'}$ illustrated in~\cref{fig:tetrahedron}.

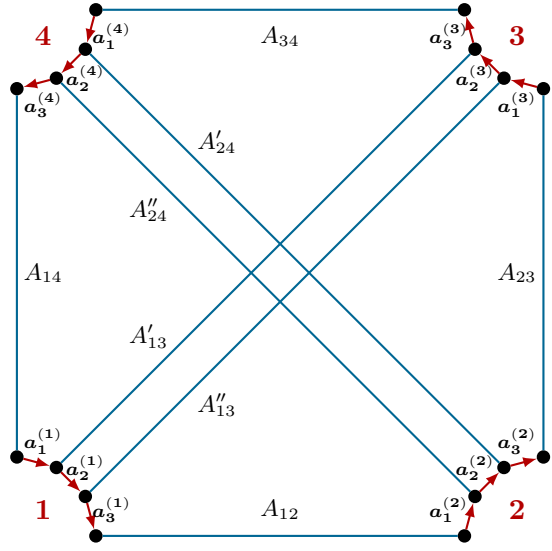
\begin{figure}[t]
\begin{tikzpicture}[every node/.style={inner sep=0, minimum size = 0.15cm, thick, circle, draw, fill}, thick, x=0.7cm, y=0.7cm]

\def\x{5};
\def\y{3.5};
\def\z{4.25};
\def\t{3.7};

\node(n1) at (-\x,-\y) {};
\node(n2) at (-\z,-\t) {};
\node(n3) at (-\t,-\z) {};
\node(n4) at (-\y,-\x) {};
\node(n5) at (\y,-\x) {};
\node(n6) at (\t,-\z) {};
\node(n7) at (\z,-\t) {};
\node(n8) at (\x,-\y) {};
\node(n9) at (\x,\y) {};
\node(n10) at (\z,\t) {};
\node(n11) at (\t,\z) {};
\node(n12) at (\y,\x) {};
\node(n13) at (-\y,\x) {};
\node(n14) at (-\t,\z) {};
\node(n15) at (-\z,\t) {};
\node(n16) at (-\x,\y) {};

\draw[MidnightBlue] (n4) to (n5);
\draw[MidnightBlue] (n8) to (n9);
\draw[MidnightBlue] (n12) to (n13);
\draw[MidnightBlue] (n16) to (n1);

\draw[MidnightBlue] (n2) to (n11);
\draw[MidnightBlue] (n3) to (n10);
\draw[MidnightBlue] (n6) to (n15);
\draw[MidnightBlue] (n7) to (n14);

\draw[red!70!black,->] (n1) to (n2);
\draw[red!70!black,->] (n2) to (n3);
\draw[red!70!black,->] (n3) to (n4);

\draw[red!70!black,->] (n5) to (n6);
\draw[red!70!black,->] (n6) to (n7);
\draw[red!70!black,->] (n7) to (n8);

\draw[red!70!black,->] (n9) to (n10);
\draw[red!70!black,->] (n10) to (n11);
\draw[red!70!black,->] (n11) to (n12);

\draw[red!70!black,->] (n13) to (n14);
\draw[red!70!black,->] (n14) to (n15);
\draw[red!70!black,->] (n15) to (n16);

\node[draw=none,fill=none] at (-4.5,-3.2) {\tiny $\boldsymbol{a_1^{(1)}}$};
\node[draw=none,fill=none] at (-3.7,-3.7) {\tiny $\boldsymbol{a_2^{(1)}}$};
\node[draw=none,fill=none] at (-3.2,-4.5) {\tiny $\boldsymbol{a_3^{(1)}}$};

\node[draw=none,fill=none] at (3.2,-4.5) {\tiny $\boldsymbol{a_1^{(2)}}$};
\node[draw=none,fill=none] at (3.7,-3.7) {\tiny $\boldsymbol{a_2^{(2)}}$};
\node[draw=none,fill=none] at (4.5,-3.2) {\tiny $\boldsymbol{a_3^{(2)}}$};

\node[draw=none,fill=none] at (4.5,3.2) {\tiny $\boldsymbol{a_1^{(3)}}$};
\node[draw=none,fill=none] at (3.7,3.7) {\tiny $\boldsymbol{a_2^{(3)}}$};
\node[draw=none,fill=none] at (3.2,4.5) {\tiny $\boldsymbol{a_3^{(3)}}$};

\node[draw=none,fill=none] at (-3.2,4.5) {\tiny $\boldsymbol{a_1^{(4)}}$};
\node[draw=none,fill=none] at (-3.7,3.7) {\tiny $\boldsymbol{a_2^{(4)}}$};
\node[draw=none,fill=none] at (-4.5,3.2) {\tiny $\boldsymbol{a_3^{(4)}}$};

\node[draw=none,fill=none,red!70!black] at (-4.5,-4.5) {$\boldsymbol{1}$};
\node[draw=none,fill=none,red!70!black] at (4.5,-4.5) {$\boldsymbol{2}$};
\node[draw=none,fill=none,red!70!black] at (4.5,4.5) {$\boldsymbol{3}$};
\node[draw=none,fill=none,red!70!black] at (-4.5,4.5) {$\boldsymbol{4}$};

\node[draw=none,fill=none] at (0,-4.5) {\footnotesize $A_{12}$};
\node[draw=none,fill=none] at (4.5,0) {\footnotesize $A_{23}$};
\node[draw=none,fill=none] at (0,4.5) {\footnotesize $A_{34}$};
\node[draw=none,fill=none] at (-4.5,0) {\footnotesize $A_{14}$};
\node[draw=none,fill=none] at (-2.5,-1.2) {\footnotesize $A'_{13}$};
\node[draw=none,fill=none] at (-1.2,-2.5) {\footnotesize $A''_{13}$};
\node[draw=none,fill=none] at (-1.2,2.5) {\footnotesize $A'_{24}$};
\node[draw=none,fill=none] at (-2.5,1.2) {\footnotesize $A''_{24}$};

\end{tikzpicture}
\caption{The fencing graph $\Gamma_{\tri,\tri'}$ in \cref{prop:D4-flip}}
\label{fig:tetrahedron}
\end{figure}

\begin{prop}\label{prop:D4-flip}
The algebra $\widetilde\zeta(\tri,\tri')$ has generators $A_\ell^{\pm1}$ with $\ell \in E_l(\Gamma_{\tri,\tri'})$ and $a_e^{\pm1}$ with $e \in E_s(\Gamma_{\tri,\tri'})$, in correspondence with the edges of the fencing graph $\Gamma_{\tri,\tri'}$ on the surface of the ideal tetrahedron shown in~\cref{fig:tetrahedron}. These generators satisfy the fencing graph relations described in~\cref{pre-glue-alg}, and in addition to these the \defterm{exchange relation}
\begin{align}
    \label{proto-exchange}
Z_{1,3}Z_{2,4} = q^{-\frac{1}{2}}Z_{1,2}Z_{3,4} + q^{\frac{1}{2}}Z_{2,3}Z_{1,4},
\end{align}
where
\begin{align*}
    Z_{1,2} &= \nord{a^{(1)}_{1}a^{(1)}_{2}a^{(1)}_{3}A_{1,2}}, &
    Z_{3,4} &= \nord{a^{(3)}_{1}a^{(3)}_{2}a^{(3)}_{3}A_{3,4}}, &
    Z_{1,3} &= \nord{a^{(1)}_{1}A_{1,3}'a_{1}^{(3)}a_{2}^{(3)}}, \\
    Z_{2,3} &= \nord{a^{(2)}_{1}a^{(2)}_{2}a^{(2)}_{3}A_{2,3}}, &
    Z_{1,4} &= \nord{a^{(4)}_{1}a^{(4)}_{2}a^{(4)}_{3}A_{1,4}}, &
    Z_{2,4} &= \nord{a^{(2)}_{1}a^{(2)}_{2}A_{2,4}'a_{1}^{(4)}}\hspace{-2.7pt}.
\end{align*}
In particular, we have injective algebra homomorphisms
\[
\widetilde\zeta(\tri)\to \widetilde\zeta(\tri,\tri'),\qquad \widetilde\zeta(\tri')\to \widetilde\zeta(\tri,\tri'),
\]
corresponding at level of generators to the inclusions of fencing graphs $\Gamma_{\tri}\subset\Gamma_{\tri,\tri'}\supset\Gamma_{\tri'}$.
\end{prop}
\begin{proof}
The generators of $\widetilde\zeta(\tri,\tri')$ are defined  in complete analogy with Propositions~\ref{basic-sq-rels} and~\ref{basic-hex-rels}. For instance, if we enumerate the tensor factors in the $i$--th $T$-region of $\DD_4$ in clockwise order, then for boundary long edges $\overline{ij}$ we have
$$
A_{i,j} = D_{i,j}\otimes \chi^{(i)}_{(0,0,0,1)}\boxtimes\chi^{(j)}_{(1,0,0,0)},
\qquad\text{while}\qquad
a^{(i)}_{j}= \chi^{(i)}_{\eps_j-\eps_{j+1}}.
$$
The diagonal long edge variables are
\begin{align*}
A'_{1,3} &= D_{1,3}\otimes\chi^{(1)}_{(0,1,0,0)}\boxtimes\chi^{(3)}_{(0,0,1,0)}, &
A''_{1,3} &= D_{1,3}\otimes\chi^{(1)}_{(0,0,1,0)}\boxtimes\chi^{(3)}_{(0,1,0,0)}, \\
A'_{2,4} &= D_{2,4}\otimes\chi^{(2)}_{(0,0,1,0)}\boxtimes\chi^{(4)}_{(0,1,0,0)}, &
A''_{2,4} &= D_{2,4}\otimes\chi^{(2)}_{(0,1,0,0)}\boxtimes\chi^{(4)}_{(0,0,1,0)}.
\end{align*}
The fencing graph relations easily follow from this description, the $q$-commutation relations among the $D_{i,j}$ in Corollary~\eqref{cor-deltas}, and the multiplication rule~\eqref{induction-mult-rule}. The exchange relation follows by observing that
$$
Z_{i,j} = q^{-\frac{3}{4}} D_{i,j} \otimes \chi^{(i)}_{(1,0,0,0)}\boxtimes\chi^{(j)}_{(1,0,0,0)},
$$
and then applying the three-term relation in $\Oq(\Conf_4)$
$$
D_{1,3} D_{2,4} = q^{-\frac{1}{2}} D_{1,2} D_{3,4} + q^{\frac{1}{2}} D_{2,3} D_{1,4}
$$
established in \cref{cor-mut}.
\end{proof}
\begin{remark}
The exchange relation~\eqref{proto-exchange} can be rewritten in terms of Weyl ordered quantum torus elements as
\begin{align}
    \label{weyl-exrel}
    Z_{1,3} &= ~\nord{Z_{2,4}^{-1} Z_{1,2} Z_{3,4}}+ \nord{Z_{2,4}^{-1} Z_{2,3} Z_{1,4}},
\end{align}
or alternatively as
\begin{align*}
    Z_{2,4} &= ~\nord{Z_{1,3}^{-1} Z_{1,2} Z_{3,4}}+ \nord{Z_{1,3}^{-1} Z_{2,3} Z_{1,4}}.
\end{align*}
\end{remark}

\subsection{Flips on $\Zc(\decS)$}
Now let us consider a general decorated surface $\decS$, and two notched triangulations $\trib$ and $\trib'$ differing by the flip of a single edge.  Let us denote by $\widetilde{\decS}$ the decorated surface obtained as union of all but the two triangles sharing the edge, and by $\widetilde{\tri}$ the resulting triangulation of $\widetilde{\decS}$. Then we have a decomposition
\[\decS = \widetilde{\decS} \sqcup_{\DD_2^{\sqcup 4}} \DD_4.
\]
By excision, the flip functor $\mu_{\tri,\tri}$ thus admits a natural description as
\[
\mu_{\tri,\tri'} = \id_{\widetilde{\decS}} \boxtimes~ \mu^{\DD_4},
\]
where $\mu^{\DD_4}$ is the functor between open subcategories of $\Zc(\DD_4)$ from \cref{D4-flip-functor} induced by changing its triangulation. The only subtlety arises when we wish to describe this functor at the level of generators
and relations for the algebras $\zeta(\trib),\zeta(\trib')$. Recall that if $\decS$ has punctures, it is the the data of the notching that singles out a distinguished set of generators for the quantum torus $\zeta(\trib)$. Keeping track of the effect of flips on the notch necessitates a small amount of additional bookkeeping, which we shall now describe. 

Let us again enumerate the (possibly coincident) $T$-regions of the 4-gon $\{1,2,3,4\}$ in counterclockwise order. Then we have edges $\hc{e_{1,2}, e_{2,3}, e_{3,4}, e_{4,1}}$ common to both triangulations $\trib,\trib'$, where $\partial e_{i,j} = \hc{i,j}$ and we allow for the possibility that some of the $e_{i,j}$ are in fact identified. 
Using the notchings, we assign to each edge--end $e_i$ of a diagonal of $\DD_4$ an integer $\gamma_{i}$ defined by
\begin{align}
\label{special-angles}
\gamma_{i} = \begin{cases}
1 & \text{if} \quad |e_i| \text{ is minimal,} \\
-1 & \text{if} \quad |e_i| \text{ is maximal}, \\
0 & \text{otherwise}.
\end{cases}
\end{align}
More informally, we declare $\gamma_i=1$ if $e_i$ is the first edge-end with respect to the total order from the notched triangulation in which $e_i$ appears, while $\gamma_i=-1$ if $e_i$ is the last edge-end with respect to the notching; in all other cases we set $\gamma_i=0$.
Then we have the following exchange relation between the distinguished generating sets of the quantum tori $\zeta(\trib),\zeta(\trib')$ corresponding to the notched triangulations $\trib,\trib'$:
\begin{prop}
\label{prop:Ahat-flip}
Suppose that $\trib,\trib'$ are two notched triangulations whose underlying triangulations differ by flipping a single edge. Let ${Z}_{i,j} \in \zeta(\trib,\trib')$ denote the element corresponding to the diagonal $e_{i,j}$ with edge-ends $e_i$ and $e_j$, and let $(\gamma_i)_{i=1}^4$ be the integers defined in~\eqref{special-angles}. Let us set

\begin{align*}
  \alpha^{12}_{34} &= \alpha_1^{\delta_{\gamma_1,-1}} \alpha_2^{-\delta_{\gamma_2,1}} \alpha_3^{\delta_{\gamma_3,-1}} \alpha_4^{-\delta_{\gamma_4,1}}, \\
  \alpha^{14}_{23} &= \alpha_1^{-\delta_{\gamma_1,1}} \alpha_2^{\delta_{\gamma_2,-1}} \alpha_3^{-\delta_{\gamma_3,1}} \alpha_4^{\delta_{\gamma_4,-1}}.
\end{align*}
Then we have the following exchange relations in $\zeta(\trib,\trib')$:
\begin{align}
\label{skeinish}
Z_{1,3} &= \nord{\alpha^{12}_{34} Z_{2,4}^{-1} Z_{1,2} Z_{3,4}}+ \nord{\alpha^{14}_{23}Z_{2,4}^{-1} Z_{2,3} Z_{1,4}},
\end{align}
\begin{align}
\label{skeinish-2}
Z_{2,4} &= \nord{\alpha^{12}_{34} Z_{1,3}^{-1} Z_{1,2} Z_{3,4}}+ \nord{\alpha^{14}_{23}Z_{1,3}^{-1} Z_{2,3} Z_{1,4}},
\end{align}
where the Weyl ordering is taken with respect to the set of generators $\{{Z}_e,\alpha_i\}$.
\end{prop}
\begin{proof}
We must translate the identity~\eqref{weyl-exrel}, which holds in the larger algebra 
$$
\widetilde{\zeta}(\tri,\tri')\simeq \widetilde\zeta({\tri})\otimes_{\zeta(\DD_2)^{\otimes4}}\widetilde\zeta_{\DD_4}(\tri,\tri'),
$$
into an identity between elements of the subalgebra ${\zeta}(\trib,\trib')$ of $\Tf'_\tri$--invariants in $\widetilde{\zeta}(\tri,\tri')$. To do this, for each $T$--region $i$ of $\DD_4$, let $w_i$ be the first vertex of the fencing graph for $\DD_4$ with respect to the clockwise order within $i$. Moreover, let $v_i$ be the distinguished vertex in the $T$-region of $\decS$ that contains the $T$--region $i$ of the subsurface $\DD_4$. Now write $\pi(i)$ for the set of short edges that make up the minimal path in the fencing graph $\Gamma_{\tri,\tri'}$ starting at  $v_i$  and travelling clockwise to vertex $w_i$. Then multiplying both sides of the relation~\eqref{weyl-exrel} by the product of short edges 
$$
\nord{\prod_{e\in\pi(1)}\prod_{f\in\pi(3)} a_e a_f},
$$
we obtain the relation~\eqref{skeinish}.
\end{proof}

The following is an easy direct computation.

\begin{cor}
The isomorphisms $\iota_{\trib}$ between $\Xc^q_\tri$ and $\chi_{\PGL_2}(\tri)$ intertwine the cluster mutation and the restricted flips functors $\chi_{\PGL_2}(\tri,\tri')$ obtained from $\zeta_{\SL_2}(\trib,\trib')$.  
\end{cor}

\subsection{Topological remarks}
\label{sec:modularity}
In this final section we collect a number of remarks concerning the topological invariance built into our construction.

Let $\Aut(\decS)$ denote the 2-group of automorphisms of $\decS$ in $\Surp$, i.e. the category with a single object, whose 1-morphisms are self-diffeomorphisms of $\decS$ which respect the stratification and the labeling by $G$ and $T$, and whose 2-morphisms are stratified isotopies of such (equivalently $\Aut(\decS)$ may be regarded as a monoidal category in which all objects and morphisms are invertible).  We note that the $1$-truncation $\pi_{\leq 1}\Aut(\decS)$ is the precisely the marked mapping class group, since by definition this means that we take isotopy equivalence classes.  Given a set $\Gc$ of $G$- and $T$-gates, let us also denote by $\Aut_\Gc(\decS)$ the sub $2$-group of $\Aut(\decS)$ of diffeomorphisms and isotopies which fix $\Gc$ componentwise.

Simply because $\Zc$ is a functor, the category $\Zc(\decS)$ inherits a canonical action of the 2-group $\Aut(\decS)$.  This data can be understood as an action of the mapping class group, in which functors compose by the group law only up to a coherent isomorphism, encoded by a 2-cocycle with values in the center of the category.  We note that the distinguished object, being defined by the empty embedding is preserved canonically (i.e. up to a canonical isomorphism) by any diffeomorphism, so that it obtains a canonical structure of a $\Aut(\decS)$-fixed point of the category.  In particular the mapping class group acts strictly on the vector space $\End(\Dist_\decS)$, and in such a way that it intertwines the cluster charts coming from different triangulations.

Suppose we are given two triangulations $\tri$ and $\tri'$ related by an isotopy $\gamma$.  We may assume $\gamma$ at time $t=0$ to be the identity automorphism of $\decS$, so that as $t$ runs to $t=1$, it induces isomorphisms between objects of $\Zc(\tri)$ and $\Zc(\tri')$.  Since both categories are closed under isomorphisms, it means two subcategories related by an isotopy are simply equal.  

An interesting extra feature comes from the $\Repq(\Gc)$-action at the gates.  In the presence of gates $\Gc$, any functor arising coming from $\Aut_\Gc(\decS)$ is canonically promoted to a $\Gc$-module functor, and any isotopy to a $\Gc$-module natural isomorphism.  This additional naturality has a consequence for the charts $\zeta(\trib)$ and their mutations.  The algebras $\zeta(\trib)$ are defined with respect to a notched triangulation, yet their presentation, as encoded by the fencing graph, depended only on the associated notched ideal triangulation.  We note that two isotopy classes of notched triangulations give rise to the same notched ideal triangulation if and only if they differ by iterated application of the Dehn twist implementing $2\pi$ rotation at some puncture.  This Dehn twist preserves the notching, hence it gives a $\Repq(T)$-module auto-equivalence of $\Zc(\decS)$.  However, the natural isotopy undoing the Dehn twist \emph{does not} preserve the notch.

This means that the Dehn twist \emph{is isomorphic} to the identity functor as a plain functor of categories, but is \emph{not isomorphic} as a $\Repq(T)$-module category.    Hence the Dehn twist does not act trivially on internal endomorphism algebras such as $\zeta(\trib)$.  Indeed, a direct computation implies that the Dehn twist induces a homomorphism $\widetilde\zeta(\trib)\to\widetilde\zeta(\trib')$ given by $Z_e\mapsto qZ_e\alpha_p$, for all edges $e$ incident to the puncture.

And additional interesting subtlety comes when we consider the appearance of topological framings of $\decS$.  In \cref{sec:general-charts}, we have implicitly fixed a 2-framing of the surface $\decS$ when computing the charts $\widetilde\zeta$.  The data of the 2-framing can be prescribed by a \emph{jellyfish} diagram as in \cref{fig:torus}: the triangles are drawn in some order in a line, and the attaching digons are stretched out to long legs connecting the triangles.  Such a prescription topologically fixes a 2-framing on the surface -- the blackboard framing -- and also a total ordering on edge-ends of the triangulation.  This data entered into our computation of the generators and relations for $\zeta(\trib)$.  However, the reader will note that the algebra $\widetilde\zeta(\tri)$ has been defined in such a way that it \emph{does not depend} on this data of framing.

Hence, while neither the \emph{algebras} $\widetilde\zeta(\tri)$, $\zeta(\trib)$, nor the \emph{flips} $\widetilde\zeta(\tri,\tri')$, $\zeta(\trib,\trib')$ depend on the choice of 2-framing, their \emph{presentation} by generators and relations (that is, their distinguished choice of generators), do indeed depend on such a choice (in fact, less data, the choice of a 3-framing at the quantum level, a spin structure classically).  Algebraically, this is because we have chosen the elements $D_{ij}$ as generators of a 1-dimensional space of invariants in $\Oq(N\backslash G)\otimes \Oq(N\backslash G),$ but this space doesn't naturally come with a distinguished vector.

We wish to stress that in traditional approach to quantization via quantum cluster algebras, one \emph{must fix} these choices {\it a priori} in order to \emph{define} the quantum cluster charts and their flips in the first place.  By contrast, we \emph{may make} such choices in order to \emph{compute} our quantum cluster charts and their flips {\it a posteriori}.

This distinction also resolves a potential point of confusion about $\Ac$-varieties in our set-up.  Traditionally, these are defined as moduli spaces of \emph{twisted} local systems on the surface with $B$-reductions.  Any choice of spin structure on the surface (in particular any 3-framing, and in particular any 2-framing) gives an isomorphism between the moduli spaces of twisted and ordinary local systems.  If defining charts and flips by formulas, one must work with twisted local systems and non-canonically identify them with ordinary local systems by a choice of spin structure.  In our approach, this is reversed, and ordinary local systems are what naturally appear.

\section{Examples}

In this section we consider the examples when $S$ is a cylinder with one marked point on each boundary circle, a punctured torus, and a three punctured sphere. On~\cref{fig:examples}, we show a pair $\tri, \tri'$ of superimposed triangulations of each surface. The triangulation $\tri$ consists of edges labelled $\hc{1,2,3}$ or $\hc{1,2,3,4}$, and $\tri'$ is obtained from $\tri$ by flipping the edge 3 to the one labelled $3'$. We denote the tori $\zeta_{\SL_2}(\tri)$ and $\zeta_{\SL_2}(\tri')$ by $\zeta$ and $\zeta'$ respectively. Similarly, we let $\chi$ and $\chi'$ denote the tori $\chi_{\PGL_2}(\tri)$ and $\chi_{\PGL_2}(\tri')$. We let $Z_j,\alpha_p$ be the generators of~$\zeta$, while setting $Z'_3 := Z_{3'} \in \zeta'$. The generators of $\chi$ and $\chi'$ are denoted respectively by~$\hc{X_i}$ and $\hc{X'_i}$. Finally, we let $\mu := \mu_{\tri,\tri'}$ be the flip between the triangulaitons $\tri$ and $\tri'$.

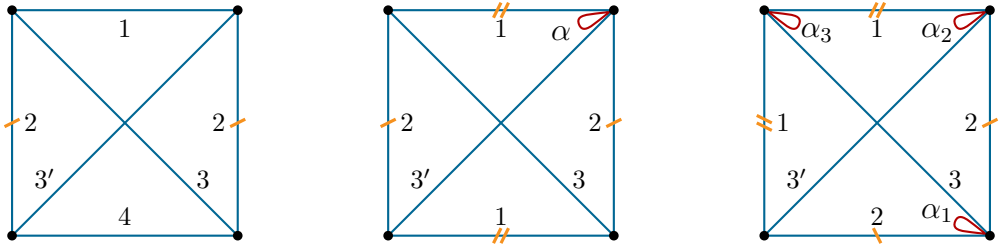
\begin{figure}
\begin{tikzpicture}[every node/.style={inner sep=0.5, thick, circle}, thick, x=0.5cm, y=0.5cm]

\def\z{0.2}

\draw[MidnightBlue, postaction={decorate,decoration={markings,mark=at position 0.5 with{
            \draw[-,BurntOrange,very thick] (\z/2,-\z) -- (-\z/2,\z);
            \node[black] at (0,-0.5) {\small $2$};
        }}}] (0,0) to (0,6);

\draw[MidnightBlue, postaction={decorate,decoration={markings,mark=at position 0.5 with{
            \node[black] at (0,-0.5) {\small $1$};
        }}}] (0,6) to (6,6);

\draw[MidnightBlue, postaction={decorate,decoration={markings,mark=at position 0.5 with{
            \draw[-,BurntOrange,very thick] (\z/2,-\z) -- (-\z/2,\z);
            \node[black] at (0,-0.5) {\small $2$};
        }}}] (6,6) to (6,0);

\draw[MidnightBlue, postaction={decorate,decoration={markings,mark=at position 0.5 with{
            \node[black] at (0,-0.5) {\small $4$};
        }}}] (6,0) to (0,0);

\draw[MidnightBlue, postaction={decorate,decoration={markings,mark=at position 0.2 with{
            \node[black] at (0,0.47) {\small $3'$};
        }}}] (0,0) to (6,6);

\draw[MidnightBlue, postaction={decorate,decoration={markings,mark=at position 0.8 with{
            \node[black] at (0,0.4) {\small $3$};
        }}}] (0,6) to (6,0);

\draw[fill] (0,0) circle (0.1);
\draw[fill] (6,0) circle (0.1);
\draw[fill] (0,6) circle (0.1);
\draw[fill] (6,6) circle (0.1);

\begin{scope}[shift={(10,0)}]

\draw[red!70!black, shift={(6,6)}, rotate=-157.5] (0,0) to [out=0, in=-90] (1,0) to [out=90,in=0] (0,0);

\draw[MidnightBlue, postaction={decorate,decoration={markings,mark=at position 0.5 with{
            \draw[-,BurntOrange,very thick] (\z/2,-\z) -- (-\z/2,\z);
            \node[black] at (0,-0.5) {\small $2$};
        }}}] (0,0) to (0,6);

\draw[MidnightBlue, postaction={decorate,decoration={markings,mark=at position 0.5 with{
            \draw[-,BurntOrange,very thick] (-\z,-\z) -- (0,\z);
            \draw[-,BurntOrange,very thick] (0,-\z) -- (\z,\z);
            \node[black] at (0,-0.5) {\small $1$};
        }}}] (0,6) to (6,6);

\draw[MidnightBlue, postaction={decorate,decoration={markings,mark=at position 0.5 with{
            \draw[-,BurntOrange,very thick] (\z/2,-\z) -- (-\z/2,\z);
            \node[black] at (0,-0.5) {\small $2$};
        }}}] (6,6) to (6,0);

\draw[MidnightBlue, postaction={decorate,decoration={markings,mark=at position 0.5 with{
            \draw[-,BurntOrange,very thick] (-\z,-\z) -- (0,\z);
            \draw[-,BurntOrange,very thick] (0,-\z) -- (\z,\z);
            \node[black] at (0,-0.5) {\small $1$};
        }}}] (6,0) to (0,0);

\draw[MidnightBlue, postaction={decorate,decoration={markings,mark=at position 0.8 with{
            \node[black] at (0,0.4) {\small $3$};
        }}}] (0,6) to (6,0);

\draw[MidnightBlue, postaction={decorate,decoration={markings,mark=at position 0.2 with{
            \node[black] at (0,0.47) {\small $3'$};
        }}}] (0,0) to (6,6);
        
\node at (4.6,5.4) {\large $\alpha$};

\draw[fill] (0,0) circle (0.1);
\draw[fill] (6,0) circle (0.1);
\draw[fill] (0,6) circle (0.1);
\draw[fill] (6,6) circle (0.1);

\end{scope}

\begin{scope}[shift={(20,0)}]

\draw[red!70!black, shift={(6,0)}, rotate=157.5] (0,0) to [out=0, in=-90] (1,0) to [out=90,in=0] (0,0);
\draw[red!70!black, shift={(6,6)}, rotate=-157.5] (0,0) to [out=0, in=-90] (1,0) to [out=90,in=0] (0,0);
\draw[red!70!black, shift={(0,6)}, rotate=-22.5] (0,0) to [out=0, in=-90] (1,0) to [out=90,in=0] (0,0);

\draw[MidnightBlue, postaction={decorate,decoration={markings,mark=at position 0.5 with{
            \draw[-,BurntOrange,very thick] (-\z,-\z) -- (0,\z);
            \draw[-,BurntOrange,very thick] (0,-\z) -- (\z,\z);
            \node[black] at (0,-0.5) {\small $1$};
        }}}] (0,0) to (0,6);

\draw[MidnightBlue, postaction={decorate,decoration={markings,mark=at position 0.5 with{
            \draw[-,BurntOrange,very thick] (-\z,-\z) -- (0,\z);
            \draw[-,BurntOrange,very thick] (0,-\z) -- (\z,\z);
            \node[black] at (0,-0.5) {\small $1$};
        }}}] (0,6) to (6,6);

\draw[MidnightBlue, postaction={decorate,decoration={markings,mark=at position 0.5 with{
            \draw[-,BurntOrange,very thick] (\z/2,-\z) -- (-\z/2,\z);
            \node[black] at (0,-0.5) {\small $2$};
        }}}] (6,6) to (6,0);

\draw[MidnightBlue, postaction={decorate,decoration={markings,mark=at position 0.5 with{
            \draw[-,BurntOrange,very thick] (\z/2,-\z) -- (-\z/2,\z);
            \node[black] at (0,-0.5) {\small $2$};
        }}}] (6,0) to (0,0);

\draw[MidnightBlue, postaction={decorate,decoration={markings,mark=at position 0.2 with{
            \node[black] at (0,0.47) {\small $3'$};
        }}}] (0,0) to (6,6);

\draw[MidnightBlue, postaction={decorate,decoration={markings,mark=at position 0.8 with{
            \node[black] at (0,0.4) {\small $3$};
        }}}] (0,6) to (6,0);
        
\node at (4.6,0.6) {\large $\alpha_1$};
\node at (4.6,5.4) {\large $\alpha_2$};
\node at (1.4,5.4) {\large $\alpha_3$};

\draw[fill] (0,0) circle (0.1);
\draw[fill] (6,0) circle (0.1);
\draw[fill] (0,6) circle (0.1);
\draw[fill] (6,6) circle (0.1);

\end{scope}

\end{tikzpicture}
\caption{At left: a cylinder with a marked point on each boundary circle. In the middle: a punctured torus. At right: a 3-punctured sphere.}
\label{fig:examples}
\end{figure}

\subsection{The cylinder}
We have
\begin{align*}
\zeta_1 &= \C[q^{\pm\frac12}]\ha{Z_1^{\pm1},Z_2^{\pm1},Z_3^{\pm1},Z_4^{\pm1}}, \\
\zeta_2 &= \C[q^{\pm\frac12}]\big\langle Z_1^{\pm1},Z_2^{\pm1},{Z'_3}^{\pm1},Z_4^{\pm1}\big\rangle,
\end{align*}
where the elements $Z_1$ and $Z_4$ are central in both tori, the other relations read
$$
Z_2 Z_3 = q Z_3 Z_2 \qquad\text{and}\qquad Z'_3 Z_2 = q Z_2 Z'_3,
$$
and the flip $\mu$ takes form
$$
Z'_3Z_3 = Z_1Z_4 + qZ_2^2.
$$

The $X$-variables of the torus $\chi_1$ then read
$$
X_1 = \nord{Z_1Z_2Z_3^{-1}}, \qquad X_2=\nord{Z_3^2Z_1^{-1}Z_4^{-1}}, \qquad X_3 = \nord{Z_1Z_4Z_2^{-2}}, \qquad X_4 = \nord{Z_4Z_2Z_3^{-1}},
$$
and satisfy the relations
\begin{align*}
X_1X_2 &= q^2X_2X_1, & X_1X_3 &= q^{-2}X_3X_1, & X_1X_4 &= X_4X_1, \\
X_4X_2 &= q^2X_2X_4, & X_4X_3 &= q^{-2}X_3X_4, & X_2X_3 &= q^4X_3X_2.
\end{align*}
Similarly, the $X$-variables of the torus $\chi_2$ are given by
$$
X'_1 = \nord{Z_1Z'_3Z_2^{-1}}, \qquad X'_2 = \nord{Z_1Z_4{Z'_3}^{-2}}, \qquad X'_3 = \nord{Z_2^2Z_1^{-1}Z_4^{-1}}, \qquad X'_4 = \nord{Z_4Z'_3Z_2^{-1}},
$$
and satisfy the relations
\begin{align*}
X'_1X'_2 &= q^{-2}X'_2X'_1, & X'_1X'_3 &= q^2X'_3X'_1, & X'_1X'_4 &= X'_4X'_1, \\
X'_4X'_2 &= q^{-2}X'_2X'_4, & X'_4X'_3 &= q^2X'_3X'_4, & X'_2X'_3 &= q^{-4}X'_3X'_2.
\end{align*}
Finally the flip $\mu$ reads
$$
X'_3 = X_3^{-1}, \qquad X'_1 = X_1(1+qX_3), \qquad X'_4 = X_4(1+qX_1),
$$
and
$$
X'_2 = X_2(1+qX_3^{-1})^{-1}(1+q^{3}X_3^{-1})^{-1}.
$$

\subsection{The punctured torus}

In this case, we get
\begin{align*}
\zeta_1 &= \C[q^{\pm\frac12}]\ha{Z_1^{\pm1},Z_2^{\pm1},Z_3^{\pm1},\alpha^{\pm1}}, \\
\zeta_2 &= \C[q^{\pm\frac12}]\big\langle Z_1^{\pm1},Z_2^{\pm1},{Z'_3}^{\pm1},\alpha^{\pm1}\big\rangle.
\end{align*}
The generators satisfy relations
\begin{align*}
Z_2Z_1 &= qZ_1Z_2, & Z_2Z_3 &= qZ_3Z_2, & Z_3Z_1 &= qZ_1Z_3, \\
&& Z_2Z'_3 &= qZ'_3Z_2, & Z_1Z'_3 &= qZ'_3Z_1,
\end{align*}
as well as
$$
\alpha Z_x = q^2 Z_x \alpha
$$
for any edge $x \in \hc{1,2,3,3'}$. The flip $\mu$ reads
$$
Z'_3Z_3 = q^2Z_1^2 \alpha + qZ_2^2.
$$

The $X$-generators of the torus $\chi_1$ read
$$
X_1 = \nord{Z_2^2Z_3^{-2}\alpha}, \qquad X_2 = \nord{Z_3^2Z_1^{-2}\alpha}, \qquad X_3 = \nord{Z_1^2Z_2^{-2}\alpha^{-1}},
$$
and satisfy the relations
$$
X_iX_{i+1} = q^4X_{i+1}X_i,
$$
for $i \in \Z/3\Z$. Similarly, the $X$-generators of the torus $\chi_2$ take form
$$
X'_1 = \nord{{Z'_3}^2Z_2^{-2}\alpha}, \qquad X'_2 = \nord{Z_1^2{Z'_3}^{-2}\alpha^{-1}}, \qquad X'_3 = \nord{Z_2^2Z_1^{-2}\alpha},
$$
and satisfy the relations
$$
X'_iX'_{i+1} = q^{-4}X'_{i+1}X'_i,
$$
for $i \in \Z/3\Z$. Finally, the flip of $X$-variables is
$$
X'_1 = X_1(1+qX_3)(1+q^{3}X_3), \qquad X'_2 = X_2(1+qX_3^{-1})^{-1}(1+q^{3}X_3^{-1})^{-1}, \qquad X'_3 = X_3^{-1}.
$$

\subsection{The 3-punctured sphere}
We have
\begin{align*}
\zeta_1 &= \C[q^{\pm\frac12}]\ha{Z_1^{\pm1},Z_2^{\pm1},Z_3^{\pm1},\alpha_1^{\pm1},\alpha_2^{\pm1},\alpha_3^{\pm1}}, \\
\zeta_2 &= \C[q^{\pm\frac12}]\big\langle Z_1^{\pm1},Z_2^{\pm1},{Z'_3}^{\pm1},\alpha_1^{\pm1},\alpha_2^{\pm1},\alpha_3^{\pm1}\big\rangle.
\end{align*}
The generators satisfy relations
\begin{align*}
Z_2Z_1 &= q^{\frac12}Z_1Z_2, & Z_3Z_2 &= q^{\frac12}Z_2Z_3, & Z_3Z_1 &= q^{\frac12}Z_1Z_3, \\
&& Z'_3Z_2 &= Z_2Z'_3, & Z'_3Z_1 &= qZ_1Z'_3,
\end{align*}
as well as
$$
[\alpha_i,\alpha_j]=0 \qquad\text{and}\qquad \alpha_i Z_j = q^{1-\delta_{ij}} Z_j \alpha_i,
$$
for any $1 \le i,j \le 3$. The flip $\mu$ reads
$$
Z'_3Z_3 = q^{\frac12}\!\!\nord{\alpha_1\alpha_2Z_1Z_2}+q^{-\frac12}\!\!\nord{\alpha_3Z_1Z_2}.
$$

The torus $\chi_1$ is commutative and is generated by the elements
$$
X_j = \frac{\alpha_1\alpha_2\alpha_3}{\alpha_j^2},
$$
for $1 \le i,j \le 3$. The torus $\chi_2$ is commutative as well, and is generated by
$$
X'_1 = X_1, \qquad X'_2 = X_2, \qquad X'_3 = X_3^{-1}.
$$

\printbibliography
\end{document}